\documentclass[11pt]{amsart}
\usepackage[T1]{fontenc} 
\usepackage[latin1]{inputenc}
\usepackage{a4}
\usepackage{setspace}
\usepackage{xypic}
\usepackage{amssymb}
\usepackage{amsthm}
\usepackage[english]{babel}
\usepackage{latexsym}
\date{}

\usepackage{srcltx}

\newtheorem{thm}{Theorem}[section]

\newtheorem{example}[thm]{Example}
\newtheorem{definition}[thm]{Definition}
\newtheorem{rem}[thm]{Remark}
\newtheorem{conj}[thm]{Conjecture}

\newtheorem{conjL'}[thm]{Conjecture $\textbf{L}(\overline{\mathcal{X}}_{et},d)_{\geq0}$}

\newtheorem{prop}[thm]{Proposition}
\newtheorem{prop-definition}[thm]{Proposition--Definition}

\newtheorem{lem}[thm]{Lemma}
\newtheorem{cor}[thm]{Corollary}

\newenvironment{f-proof}[1][\sc Proof.]{\begin{trivlist}
\item[\hskip \labelsep {\bfseries #1}]}{\hfill{$\square$}\end{trivlist}}

\newcommand{\fonc}[5]{
 \begin{array}{cccc}
 #1: & #2 & \longrightarrow & #3\\
     & #4 & \longmapsto & #5
 \end{array}
}

\newcommand{\G}{\mbox{$\Gamma$}}

\begin{document}
\title[Quadratic structure of torsors]{On the Quadratic Structure of Torsors over Affine Group Schemes}

\author[Ph. Cassou-Nogu\`es]{Ph. Cassou-Nogu\`es}
\address{Philippe Cassou-Nogu\`es, IMB\\ Univ. Bordeaux 1\\  33405 Talence, France.\\}
 
 \email{Philippe.Cassou-Nogues@math.u-bordeaux1.fr} 

\author[M. Taylor]{M. J. Taylor}
\address{Martin J. Taylor, Merton College \\ 
Oxford OX1 4JD, U.K.}
\email{martin.taylor@merton.ox.ac.uk}

\maketitle
\centerline {(with an appendix by Dajano Tossici)}
\begin{abstract}
Let $\mathcal{G}=\mathrm{Spec}(A)$ be a finite and flat group scheme over the ring of algebraic integers $R$ of a number field $K$ and suppose that the generic fiber of $\mathcal{G}$ is the constant group scheme over $K$ for a finite group $G$. Then the $R$-dual $A^D$of $A$ identifies as a Hopf $R$-order in the group algebra $K[G]$. If $B$ is a principal homogeneous space for $A$, then it is known that $B$ is a locally free $A^D$-module. By multiplying the trace form of $B_K/K$ by a certain scalar we obtain a $G$-invariant form $Tr'_B$ which provides a non-degenerate 
$R$-form on $B$. If $G$ has odd order,  we show that the $G$-forms $(B, Tr'_B)$ and  $(A, Tr'_A)$ are locally isomorphic and we study the question of when they are   globally isomorphic.

Suppose now that $K$ is a finite extension of $\mathbb Q_p$ with  valuation ring $R$.  In the course of our study we are led to consider the extension of scalars map 
$\varphi_K: G_0(A^D)\rightarrow G_0(A^D_K)=G_0(K[G])$. When $A^D$ is the group ring $R[G]$, Swan showed that $\varphi_K$ is an isomorphism. Jensen and Larson proved that $\varphi_K$ is also an isomorphism for any  Hopf $R$-order $A^D$ of $K[G]$ when $G$ is abelian and $K$ is large enough. Here we prove that $\ker \varphi_K$ is at most a finite abelian $p$-group.   However, numerous examples lead us to conjecture that Swan's result extends to all Hopf $R$-orders in $K[G]$, i.e.  $\ker \varphi_K$ is always trivial.

\end{abstract}

\section {Introduction} 

\subsection {Notation}
 If $K$ is a number field, we write $R=O_{K}$  for its ring of integers and  we
refer to $K$ as a global field. If $K$ is a finite extension of a $p$-adic
field $\mathbb{Q}_{p}$,  then we write $R$ for the valuation ring of $K$; 
here we  refer to $K$ as a local field. $G$ is a finite group, $A$ is
an $R$-Hopf order in the Hopf $K$-algebra $A_{K}=A\otimes _{R}K=\mathrm{Map}(G, K) $ and $A^{D}$ is the $R$-dual Hopf order of $A$ in $K[G]$. Note that $R[G] $ is the minimal Hopf order
in $A_{K}^{D}=A^{D}\otimes _{R}K=K[G] $ and that $R[G]
\subset A^{D}.$ Let $x\rightarrow \overline{x}\ $\ denote the standard
involution (antipode) of $A^{D}$ given on $A_{K }^{D}$ by extension by $K$-linearity from inversion of elements of $G$.

Recall that an $R$-order $B$ with an action by $A^{D}$ is a principal
homogeneous space for $A$ if there is an isomorphism of $\ B-A^{D}$ algebras 
\begin{equation}\label{phs}
B\otimes _{R}B\cong B\otimes _{R}A.
\end{equation}
Here the $B$-algebra structure of both terms comes through the two left
hand terms and the $A^{D}$-action derives from the $A^{D}$-action on the two
right hand terms and both these actions respect their $R$-algebra
structures. In other words $B$ is the algebra of a $\mathcal{G}$-torsor where $\mathcal{G}=\mathrm{Spec}(A)$. We let $\mathrm{PH}( A)$ denote  the set of
isomorphism classes of principal homogeneous spaces for $A$.  We
know that any principal homogeneous space is a rank one locally free $A^{D}$-module. In the global situation $\mathrm{Cl}
( A^{D}) $ denotes the classgroup of locally free $A^{D}$-modules and we have the usual class invariant map 
\begin{equation}\label{init1}
\psi :\mathrm{PH}( A) \rightarrow \mathrm{Cl}( A^{D}) 
\end{equation}
which maps the principal homogeneous space $B$ to the $A^{D}$-class of $B$
minus the isomorphism class of the principal $A^{D}$-class of $A$.

 In \cite{CCMT15} the
foundations were laid for the study of quadratic forms over principal homogeneous spaces.  In this set-up,  a $G$-form or a $G$-quadratic space is a pair $(M, q)$ where $M$ is a finitely generated locally free 
$A^D$-module and $q: M\times M\rightarrow R$  is an $R$-bilinear, symmetric and non-degenerate form  such that $q(gm_1, gm_2)=q(m_1, m_2)$ for all $g\in G$ and $m_1, m_2 \in M$. 
\textit{Inter alia\ }the paper \cite{CCMT15}  provides a framework which enables us
to study  arithmetic questions of the quadratic structure of principal
homogeneous spaces for $R$-Hopf orders. 

 We let $\varepsilon_A$ be the counit of $A$. In Section 2.1 we shall define the ideal of integrals $I(A)$ of $A$ and we shall  prove  that,  when $G$ is of odd order, then $\Lambda=\varepsilon_A (I(A)) $ is the square of an ideal of $R$.  Throughout this paper we shall suppose:\medskip

\textbf{Hypothesis}. {\it The $R$-ideal $\Lambda:=\varepsilon _{A}( I( A))$  is the square of a principal $R$-ideal so that we may write $\Lambda=\alpha_A^2R$ with $\alpha_A\in R$; we set  $\lambda
_{A}=\alpha_A^2$ and, when $A$ is clear from the context, we simply write $\lambda$. By \cite{CCMT15} Proposition 3.2 we know that $D_{A/R}^{-1}=\lambda_A^{-1} A$.  }
\vskip 0.1 truecm

When $G$ is a group of odd order and $R$ is a principal ideal domain, then  this hypothesis is satisfied (see Proposition \ref{H1}). 
 
 The trace of $A_K/K$ induces by linearity  a non-degenerate symmetric $
G $-invariant $R$-pairing
\begin{equation*}
Tr_{A_K/K}:D_{A/R}^{-1}=\lambda^{-1} A\times A\rightarrow R
\end{equation*}
and hence the form $Tr_{A}^{\prime }=\lambda^{-1} Tr_{A_K/K}$ is a non-degenerate $R$%
-pairing 
\begin{equation}
Tr_{A}^{\prime }:A\times A\rightarrow R.
\end{equation}
Because $ B$ is finite and flat over $R$, we know that   $D_{B/R}^{-1}=\Lambda^{-1}B=
\lambda^{-1} B$ (see \cite{CCMT15} Corollary 3.3), and again we see that, if we set $Tr_{B}^{\prime }=\lambda^{-1}
Tr_{B_K/K}$,  then we have a non-degenerate $G$-invariant $R$-pairing 
\begin{equation}
Tr_{B}^{\prime }:B\times B\rightarrow R.
\end{equation}

\begin{rem} It is important to note that the $G$-forms  considered in this article  depend upon  the choice of the generator $\lambda$  of $I(A)$. However, for the choice of $\lambda$  prescribed under  our hypothesis,  the forms are independant of this choice,  up to isometry. \end{rem}

In this paper all rings are unital and, unless indicated the contrary, modules over a ring are always taken to be left modules. 

\subsection{Results.}  
We begin by setting the scene for the presentation
of our results by explaining successively: what happens over a field $K$ for
the Hopf algebra $\mathrm{Map}(G,K)$; what happens when $
A=\mathrm{Map}( G,R)$;  and then we move on to present our
results for general Hopf orders $A$ in $\mathrm{Map}(G, K)$. 

We therefore begin by first recalling some key-results in the case when we
work over a  field $K$ and the Hopf algebras $A_K= \mathrm{Map}
(G, K) $ and $A_K^{D}$ is $K[G] $. We attach to the   field $K$ and  the  group $G$ the unit form;  this is the $G$-form $(K[G], \kappa)$ such that $\kappa(g, h)=\delta_{g, h}$  (Kronecker symbol) for $g, h\in G$. Then a Galois
extension $N/K$ with $G=\mathrm{Gal}(N/K) $ is a principal homogeneous
space for $A_K$ and the quadratic spaces  $( N,Tr_{N/K}) $ and 
$( K[ G] , \kappa) $ are isomorphic if, and only if, 
$N$ has a self-dual normal basis $n,\ $in the sense that $n$ is a normal
basis for $N/K$ with 
\[
Tr_{N/K}( g( n) .h( n)) =\delta _{g,h}.
\]

One fundamental piece of work which underlies this article\ comes from [2] \
- from which we quote the following particular case:

\begin{thm}\cite{Bayer90} Let $N/K$ be a finite Galois extension with Galois group $G$. \begin{enumerate}
\item If $K$ has characteristic different from  $2$ and  $G$ is of odd order,
then there exists a self-dual normal basis of $N/K$. 

Suppose now that $G$ is an \textbf{abelian} group:

\item If $\ K\ $has characteristic different from  $2$, then $N$ has a
self-dual normal basis if, and only if, $G$ has odd order.

\item If $\ K\ $has characteristic $2$, then $N$ has a self-dual normal basis
if, and only if, the exponent of $G$ is not divisible by $4$.
\end{enumerate}
\end{thm}

The situation in characteristic 2 for arbitrary (not necessarily abelian) $%
G\ $is completely settled by the following result of Serre in \cite{Serre14}:

\begin{thm}\label{Ser2}
For a Galois extension of fields $N/K$ with $G=\mathrm{Gal}\left( N/K\right)
,\ $if $K\ $has characteristic $2$, then $N$ has a self-dual normal basis if,
and only if, $G$ is generated by elements of odd order and elements of order
$2$.
\end{thm}

In this paper we shall be almost exclusively interested in the situation
where $G$ has odd order and so we note:

\begin{cor}\label{SN}
For a Galois extension of fields $N/K$ with $G=\mathrm{Gal}( N/K)$ having odd order and $K\ $having arbitrary characteristic, then 
$N$ always has a self-dual normal basis.
\end{cor}

\bigskip

We now replace the category of Galois \ field extensions $N/K$ with Galois
group $G$ by the larger category of $G$-Galois $K$-algebras  $B_K/K$; such $B_K$  are the principal homogeneous spaces for $A_K=%
\mathrm{Map}\left( G,K\right)$. This change of category enables us to
state results in a way that fits optimally with our other results.

In this regard it is interesting to mention the following  theorem  from [4].

\begin{thm}\label{basep}
Let $\ K$ be a number field and let $G\ $denote a finite group. With the
above notation,  if $\left( B_{K_{v}},Tr_{B_{Kv}/K_{v}}\right) $ is isomorphic to 
$\left( A_{K_{v}}, Tr_{A_{Kv}/K_{v}}\right) $ as a $G$-quadratic space for all
places $v$ of $K,\ $then $\left( B_{K,}Tr_{B_{K}/K}\right) $ is isomorphic
to $\left( A_{K},Tr_{A_K/K}\right) $ as a $G$-quadratic space.\bigskip
\end{thm}

We should also mention that in [5]\ \ a full set of cohomological conditions are
given for $\left( B_{K,}Tr_{B_{K}/K}\right) $ to be isomorphic to $\left(
A_{K},Tr_{A_K/K}\right) $ as a $G$-quadratic space in the situation where $K\ $%
is local. \bigskip 

We shall now start to consider some integral results when$\ K\ $is\ a number
field with ring of integers $R\ $and $A=\mathrm{Map}\left( G,R\right) .\ \ $%
In this situation a principal homogeneous space for \ $A=\mathrm{Map}
\left( G,R\right) $ is the ring of integers of a non-ramified $G$-Galois $K$%
-algebra\ $N$. From \cite{CPT} we have:\medskip 

\begin{thm}\label{CPTA}
The class of $O_{N}$ as a locally free $O_{K}\left[ G\right] $-module in $%
\mathrm{Cl}\left( O_{K}\left[ G\right] \right) $ is annihilated by the
exponent of $G^{ab}.$
\end{thm}

For a Galois extension of number fields$\ N/K$ with $G=\mathrm{Gal}(N/K)$ by the main
Theorem in \cite{ET} we have the following integral quadratic result:

\begin{thm}\label{ETA}
If\ $N/K\ $is at most tamely ramified and if $G$ has odd order, then the
inverse different $D_{N/K}^{-1}$ is the square of a fractional $O_{N}$%
-ideal, denoted $D_{N/K}^{-1/2},$ and the restriction to $\mathbb{Z}[G]$ of   the quadratic modules  $(D_{N/K}^{-1/2},Tr_{N/K})$ and$\ (O_K[G], \kappa)$ have the same class in the Grothendieck group of
locally free $\mathbb{Z}\left[ G\right] $-modules supporting a $G$-invariant
non-degenerate $\mathbb{Q}$-form.\bigskip 
\end{thm}

Our aim here is to examine the questions raised by the results above in a more general framework, where we consider  quadratic spaces  over rings instead of fields  and  a Hopf $R$-order 
$A$ in $
\mathrm{Map}( G,K) $ so that $A^{D}$ is a Hopf $R$-order\ in $K
[ G]$  which is not necessarily equal to $R[G]$.

 In his doctoral thesis [23], Paul Lawrence considered
the quadratic structure of $\left( B,Tr_{B_K}\right) $ when $A$ is a Hopf
order in $\mathrm{Map}( G,K) $ for a number field $K$ and when $\ B$
is a principal homogeneous space for $A$. In particular he asked   whether $\left( B,Tr_{B_K}\right) $ and $\left( A,Tr_{A_K}\right) $, considered as locally free $A^D$-modules endowed with  $G$-invariant non-degenerate $K$-forms,  are isomorphic when they are locally
isomorphic  and when $\ B$ is a free $A^{D}$-module.
He showed that, \ when$\ G$ is abelian, then the question may be resolved by
restriction to the $2$-Sylow subgroup of $G$.

 Let $PH'(A)$ denote the subset of $PH(A)$ of isomorphism classes of principal homogenous spaces $B$ of $A$ which have the property that $(B_{\frak{p}}, Tr'_{B_{\frak{p}}})$ and 
$(A_{\frak{p}}, Tr'_{A_{\frak{p}}}) $ are  isometric for all prime ideals of $R=O_K$. In this article (see (\ref{CU})) we construct a refinement $CU(A^D)$ of the locally free class group $Cl(A^D)$ and a map $\phi: PH'(A)\rightarrow CU(A^D)$ so that the class map $\psi$ of (\ref{init1}) factors through $\phi$.

In the following theorems $B$ is a principal homogeneous space of $A$. We start by considering the abelian case. 
\begin{thm}\label{totab} Suppose that $G$ is an abelian group of odd order and $K$ is a number field. Then $(B, Tr'_B)$ and $(A, Tr'_A)$ are isomorphic $G$-quadratic spaces if and only if $B$ is a free $A^D$-module. 
\end{thm}
In this article we wish   to consider a number of other more general situations.    Kneser's Strong Approximation Theorem provides the key tool to deal with these different cases. Three such situations arise in higher rank when we consider $(B, Tr'_B)^{\perp n } $ and $(A, Tr'_A)^{\perp n }$,  where for a given positive integer $n$ we let $\left( B,Tr_{B}^{\prime }\right)
^{\perp n}$ denote the orthogonal sum of $\ n$ copies of $\left(
B,Tr_{B}^{\prime }\right)$.

\begin{thm} \label{choc}Let $G$ be a finite group of odd order and let $K$ be a non-totally real number field. Suppose that $\phi(B)^n=1$. Then 
$(B, Tr')^{\perp{m_n}}$ and  $(A, Tr')^{\perp{m_n}}$ are isomorphic $G$-quadratic spaces,  where $m_n=1$ (resp. $2n$) if $n=1$ (resp. $n>1$).
\end{thm}

We assume  $G$ to be  of odd order and we consider separately the three cases where: $G$ is abelian, $G$ is not abelian and finally the group scheme $\mathrm{Spec}(A)$ is constant. The following three theorems then follow in a relative straightforward manner. 
\begin{thm}\label{LA}
 If $G$ is an abelian  group of  odd order and $K$ is a non-totally real number field. If $e=e\left( G\right) $ is the exponent of $G,$
then $\left( B,Tr_{B}^{\prime }\right) ^{\perp 2e}$ and $\left(
A,Tr_{A}^{\prime }\right) ^{\perp 2e}\ $are\ isomorphic $G$-quadratic spaces.
\end{thm}

\begin{thm}\label{gen}
Suppose that $G$ is a group of  odd order and $K$  is a  non-totally real number field.   If  $\psi(B)^m=1$,   then 
$(B,Tr_{B}^{\prime }) ^{\perp 4m}$ and $( A,Tr_{A}^{\prime })
^{\perp 4m} $ are isomorphic  $G$-quadratic spaces. 
\end{thm}

Finally, when $\ \mathrm{Spec}\left( A\right) $ is the constant group
scheme, we show:

\begin{thm}\label{const} Suppose that$\ G$ has odd order and that  $ \mathrm{Spec}( A) $ is a constant group
scheme. If $K/\mathbb{Q\,}\ $ is a non-totally real number field,  unramified at the primes dividing the order
of $G$, 
 then $( B,Tr_B^{\prime }) ^{\perp 2e^{ab}}$ and $
( A,Tr_{A}^{\prime }) ^{\perp 2e^{ab}} $ are isomorphic $G$-quadratic spaces, where $e^{ab}=e( G^{ab}) $
denotes the exponent of $G^{ab}$. 
\end{thm}

\begin{rem} One can associate to any Hopf $R$-order $A$ of $\mathrm{Map}(G, K)$ a unitary form. This is the $G$-form $(A^D, \kappa_{A^D} )$, defined by $\kappa_{A^D}(x, y)=\lambda_Al(S^D(x)y)$ for $x, y\in A^D$ where 
$l\in \mathrm{Map}(G, K)$  is given  on elements of $G$ by $l(g)=1$ if $g$ is the unit element and $0$ otherwise. Moreover, for any $B\in PH(A)$,  we can introduce  the square root of the codifferent  $D^{-1}_{B/R}$ 
 by setting $D_{B/R}^{-1/2}=\alpha_A^{-1}B$. One easily checks (see \cite{CCMT15} proof of Proposition 5.1) that the maps 
$$A^D\rightarrow A,\ u\rightarrow u.(\lambda_Al)$$
and 
$$B\rightarrow D_{B/R}^{-1/2},\ b\rightarrow \alpha_A^{-1}b$$
induce   isomorphisms of $G$-forms:  
$$(A^D, \kappa_{A^D})\simeq (A, Tr'_A) \ \mathrm{and}\ (D_{B/R}^{-1/2}, Tr_B)\simeq (B, Tr'_B).$$
Therefore  in the above theorems, the $G$-forms $(A, Tr'_A)$ and $(B, Tr'_B)$ can be respectively replaced by $(A^D, \kappa_{A^D})$ and $(D_{B/R}^{-1/2}, Tr_B)$. In particular,  under the hypotheses of Theorem 1.12,  for any non ramified Galois extension  $N/K$,  where $G=\mathrm{Gal}(N/K)$ is a  group of odd order,  we obtain an isomorphism of $G$-forms 
$$(O_N, Tr_{N/K})^{\perp 2e^{ab}}\simeq (O_K[G], \kappa_{A^D})^{\perp 2e^{ab}}.$$ We note that this is  an isomorphism of  $G$-forms over $O_K$ whereas      Theorem \ref{ETA}  deals with their restrictions to $\mathbb Z$.
\end{rem}

\subsection{Methods} In the course of the paper we shall need\ to obtain a
number of results from both the cohomology of unitary groups and also from\
the representation theory of various Hopf algebras; in particular we shall
need to extend a number of results from the theory of integral
representations and modular representations of group rings.

We start by considering the cohomological results that we require. We
suppose now that $R$ is a henselian local ring with residue field $k$;
otherwise we maintain the above notation. We let $\mathcal{G}=\mathrm{Spec}%
\left( A\right) \ $so that $\mathcal{G}$ is a finite flat group scheme over $%
\mathrm{Spec}\left( R\right)$.   In Section 2.3 we introduce 
 the notion of the unitary group scheme $U_{A^D} $ of
the $R$-algebra with involution $(A^D, S^D)$. We shall see that
there is a natural morphism of group schemes $\mathcal{G\rightarrow }\ U_{
A^{D}} $ and this induces a map of pointed sets in \textit{fppf}
cohomology (see 2.3)%
\[
u:H^{1}\left( R,\mathcal{G}\right) \mathcal{\rightarrow } H^{1}(R,U_{
A^{D}}).
\]%
In Theorem 3.1 we shall show that the map $u$ is trivial when $G$ has odd
order and either $k$ has odd \ characteristic or $k$ has even
characteristic and $\mathcal{G}$ is generically constant. Since $H^{1}\left(
R,\mathcal{G}\right) $ classifies the twisted forms of $\mathcal{G}$ and
since $H^{1}(R,U_{ A^{D}} )$ classifies the
isomorphism  classes of quadratic ${G}$-spaces which become
isomorphic to $(A, Tr'_{A })$ over \ a finite flat cover of $\mathrm{Spec}( R)$, 
this result is of vital importance in understanding the  local quadratic
structure of the $G$-forms $(B, Tr'_B)$ when $B$ is a twist of $A$.
\medskip

We conclude this subsection by highlighting some results on the
representation theory that we \ require.  We begin by very briefly recalling a
number of group ring results, and then go on to describe the generalizations
that we shall need. 

Suppose now that $R$ is the valuation ring of a finite extension of $\mathbb{%
Q}_{p}$ with field of fractions $K$ and with residue field $k.$ Let $G$
again denote a finite group. For an arbitrary integral domain $S,\ $we let $%
G_{0}\left( S\left[ G\right] \right) $ denote the Grothendieck group,  for exact sequences, of  the category of 
finitely generated $S\left[ G\right] $-modules. Then we have the extension
of scalars map $\varphi _{R\left[ G\right] }:G_{0}\left( R\left[ G\right]
\right) \rightarrow G_{0}\left( K\left[ G\right] \right).$ By \ a theorem
of R. Swan (see \cite{CR2} Theorem 39.10)  we know that $\varphi _{R\left[ G\right]
}\ $is an isomorphism. Using Theorem 33 in \cite{Serrerep} one can easily show that
the reduction map $\delta _{R\left[ G\right] }:G_{0}\left( R\left[ G\right]
\right) \rightarrow G_{0}\left( k\left[ G\right] \right) \ $is surjective.

Suppose now that $A^{D}$ is any Hopf $R$-order in $K\left[ G\right] .\ $\ We
then again have extension of scalars and reduction maps 
\[
\varphi _{A^{D}}:G_{0}\left( A^{D}\right) \rightarrow G_{0}\left( K\left[ G%
\right] \right) ,\ \ \ \delta _{A^{D}}:G_{0}\left( A^{D}\right) \rightarrow
G_{0}\left( A^{D}\otimes _{R}k\right) 
\]%
and we shall need to know when $\varphi _{A^{D}}$ is an isomorphism and  to be able to estimate $\mathrm{coker}\  \delta _{A^{D}}$. When $%
\varphi _{A^{D}}\ $is an isomorphism, we shall say that $A^{D}$ has the {\it
Swan property}.

Let\ $e$ denote \ the exponent of $G.\ $We shall follow Serre and say that $K
$ is \textit{assez gros} for $G$ if $\ K$\ contains the group of $e$-th roots
of unity. A result of Jensen and Larsen \cite{JL} shows that $A^{D}$\ has the Swan property when 
$\ G$ is abelian and when$\ K\ $is assez gros for $\ G.$ In Proposition \ref{el2}
we shall  prove an analogue of their result for  $l$-elementary groups $G\ $for\ $l\neq p.\ $%
We are then able to use Brauer induction and the theory of Frobenius modules
to show that  $\mathrm{Ker}(\varphi_{A^{D}}) $ is a finite group of 
$p$-power order.

As we point out in Section 4.3, there are very many situations where $%
\phi _{A^{D}}$ is an isomorphism. Indeed, we are not aware of any situation
where $\phi _{A^{D}}$ is not an isomorphism, and this leads us to
formulate:\medskip 

\begin{conj} With the above notation we conjecture that ${A^{D}}$ has
the Swan property for any Hopf $R$-order in $K\left[ G\right] $.
\end{conj}
\bigskip 

\textbf{1.4 Structure of paper}. In Section 2 we recall a number of basic
notations and preliminary results that we shall require in this work. In
particular, we recall some standard Hopf theory and the basic theory of
quadratic forms over Hopf orders together with their associated unitary
groups. We also recall the construction of the locally free class group of
an order, and we go on to construct a quadratic generalization, the locally
free unitary class group, \ which classifies quadratic forms over Hopf orders.
Then, in Section 3, we introduce the cohomology groups that we require and
we prove the result mentioned above concerning the triviality of the image
of $H^{1}\left( R,\mathcal{G}\right) \mathcal{\ }$\ in$\ H^{1}(R,U_{
A^{D}}) .$

Next in Section 4 we obtain the Hopf representation theoretic results and
especially the modular representation theoretic results that we require. Here we
state our conjecture on the Swan property for local Hopf orders.  In order to be able to calculate with the unitary
class group of a Hopf order we need a good understanding of the determinants
of its unitary group, and we achieve  this in Section 5. We present the proofs of
our main results in Section 6 and conclude with an  appendix, due to D. Tossici,  which contains a generalisation of the Lang-Steinberg theorem for algebraic connected group schemes used in Section 3. 
\bigskip

\textbf{1.5 Acknowledgements}.\ It is a great pleasure to express our
deepest thanks to Eva Bayer, Dajano Tossici,\ Qing Liu and Jean-Pierre Serre for
their help and advice. Without their input this paper would never have seen
the light of day. 

We are also extremely grateful to both the Institute  of Mathematics of the
University of Bordeaux and Merton College, Oxford, for their generous
financial support for our research.

\section{Basic notions and preliminary results}
\subsection{Hopf algebras}
In this subsection we gather together  the various results that we shall
need on Hopf algebras and especially Hopf orders. Here we shall only require
that $R$ be a Dedekind domain with field of fractions $K$.

Recall that an $R$-Hopf order in a finite dimensional $K$-algebra $A_{K}$ is
an $R$-order $A$ endowed with a comultiplication $\Delta :A\rightarrow
A\otimes _{R}A,$ an antipode $ S:A\rightarrow A,$ and counit $\varepsilon
:A\rightarrow R$ satisfying the relations:
\begin{eqnarray*}
\left( \Delta \otimes \mathrm{id}\right) \circ \Delta &=&\left( \mathrm{id}%
\otimes \Delta \right) \circ \Delta \\
( \varepsilon \otimes \mathrm{id})\circ \Delta &=&\mathrm{id} \\
\left( S\otimes \mathrm{id}\right) \circ \Delta &=&\varepsilon .
\end{eqnarray*}
A right $A$-comodule $M$ is an  $R$-module endowed with a structure map 
$\rho :M\rightarrow M\otimes _{R}A$ such that $(\rho \otimes id)\circ \rho=(id\otimes \Delta)\circ \rho$  and $(id\otimes \varepsilon)\circ \rho=id$. A right $A$-comodule $M$ becomes a left $A^D$-module via 
\begin{equation}A^D\otimes_RM\overset{id \otimes \rho}\rightarrow A^D\otimes_RM\otimes_RA\overset{
T\otimes id}\rightarrow M\otimes_RA^D\otimes_RA\overset{id\otimes ev}\rightarrow M\otimes_RR=M
\end{equation}
where $ev:A^D\otimes _{R}A\rightarrow R$ is the evaluation map and $T: A^D\otimes_RM\rightarrow M\otimes_RA^D$ is the twist map. Note that $
A$ is an $A$-comodule via its comultiplication and is therefore naturally a
left $A^{D}$-module. To understand  the structure of $A$ and $A^{D}$ we
need:
 \begin{definition} The (left) integrals $I\left( A\right) \ $of $A$ are
defined as the $R$-module%
\begin{equation*}
I\left( A\right) =\{x\in A\mid a.x=\varepsilon _{A}\left( a\right).x ~\text{%
for all }a\in A\}.
\end{equation*}%
Similarly the (left) integrals$\ I\left( A^{D}\right)$ of $A^{D}$ are
defined as the $R$-module%
\begin{equation*}
I\left( A^{D}\right) =\{f\in A^{D}\mid u.f=\varepsilon _{A^{D}}\left(
u\right).f ~\text{for all }u\in A^{D}\}.
\end{equation*}%
\end{definition}
In the following proposition we state some of the key-results that we shall
require (see \cite{Ch} Section 3 and \cite{L72}):
\begin{prop}\label{H0}   The following properties hold:

\begin{enumerate}

\item $I\left( A\right) $ and $I\left( A^{D}\right) $ are both rank one
locally free $R$-modules.

\item $A=A^{D}I\left( A\right) $ and so $A$ is a rank one locally free $A^{D}$%
-module.

\item If $n=\dim _{K}\left( A_{K}\right) $ then $\varepsilon _{A^{D}}\left(
I\left( A^{D}\right) \right) .\varepsilon _{A}\left( I\left( A\right)
\right) =nR$.

\item If $C$ is an $R$-Hopf suborder of $A$ in $A_{K},$ then $\varepsilon
\left( I\left( C\right) \right) A\subset \varepsilon \left( I\left( A\right)
\right) C$.
\end{enumerate}
\end{prop}
We shall almost exclusively be interested in the following situation where $
G $ is a finite group: we shall require $A$ to be an $R$-Hopf order in the
Hopf $K$-algebra $A_{K}=\mathrm{Map}( G,K) \simeq \mathrm{Hom}_{K}(K[G] ,K) $ where explicitly: for $f\in \mathrm{Map}(
G,K) $ and $g,h\in G$
\begin{eqnarray*}
\Delta \left( f\right) \left( g\otimes h\right) &=&f\left( g.h\right) \\
S\left( f\right) \left( g\right) &=&f\left( g^{-1}\right) \\
\varepsilon \left( f\right) &=&f\left( 1_{G}\right) 
\end{eqnarray*}
and $A^{D}$ is an $R$-Hopf order in the Hopf $K$-algebra $A_{K}^{D}=K[ G] $ where explicitly: for $g,h\in G$ 
\begin{eqnarray*}
\Delta \left( g\right) &=&g\otimes g \\
S\left( g\right) &=&g^{-1} \\
\varepsilon \left( g\right) &=&1.
\end{eqnarray*}
We note that $A=\mathrm{Map}(G, R)$ is the maximal $R$-order in $\mathrm{Map}(G, K)$ and that $R[G]$ is the minimal $R$-Hopf order in $K[G]$. Moreover one easily checks in this situation that
$\varepsilon_A(I(A))=R$ and $\varepsilon_{A^D}(I(A^D))=nR$. 

We denote by $\Lambda$ the $R$-ideal $\varepsilon_A(I(A))$. From  \cite{CCMT15} Section 3.1 we know that $D_{B/R}^{-1}=\Lambda^{-1}B$ for any $B\in PH(A)$.  
  We recall that we let $l$ be the element of $A_K$ defined by $l(g)=1$ (resp. $0$ ) if $g=1$ (resp. $g\neq 1$).

For future reference we note:
\begin{prop}\label{eg}Let $A$ be an $R$-Hopf order of $A_K=\mathrm{Map}(G, K)$. Assume that $\Lambda$ is a principal ideal of $R$ generated by $\lambda$. Then 
\begin{enumerate}
\item $A$  is a free $A^D$-module of rank $1$  with basis $\theta=\lambda l$. 
\item For  any   element $f$ of $A$, then $\lambda^{-1}\sum_{g\in G}f(g^{-1})g$ is the unique element $u \in A^D$ such that 
$f=u\theta$. 
\end{enumerate} 
\begin{proof} It follows from \cite{CCMT15}  Section 3.1  that $\theta$ is a basis of $I(A)$ over $R$.  Therefore we deduce from Proposition \ref{H0} that $\theta$ is also a basis of $A$ as an $A^D$-module. 

 Let $f\in A$. Since $f\in A_K$ we can write $f=\sum_{g\in G}f(g)l_g$, where $l_g\in  A_K$ is defined on the elements of $G$ by $l_g(h)=\delta^1_{g, h}$. Using that $l$ is a basis of $A_K$ as  an  $A_K^D$-module  we deduce that there exists a unique 
$v:=\sum_{g\in G}v_gg$ such that $f=vl$. We know that $\Delta(l=l_e)=\sum_{g\in G}l_g\otimes l_{g^{-1}}$  and so that $v l=\sum_{g\in G}v_{g^{-1}}l_g$. It follows from the equality $f=v l$ that $v_g=f(g^{-1})$ and so that  
$$f=(\sum_{g\in G}f(g^{-1})g)l=(\lambda^{-1}\sum_{g\in G}f(g^{-1})g)(\lambda l)=u\theta.$$  Since $\theta$ is a basis of $A_K$ as an $A_K^D$-module, we conclude that 
$u=\lambda^{-1}\sum_{g\in G}f(g^{-1})g$ as required.

\end{proof}
\end{prop}
\begin{prop}\label{H1} Let $A$ be an $R$-Hopf order of $A_K=\mathrm{Map}(G, K)$. Then 
\begin{enumerate}
\item$ \varepsilon(I(A))A^D\subset R[G]$.

\item If the group $G$ has odd order then  $\varepsilon _{A}( I(
A)) $ is the square of an $R$-ideal.  
\end{enumerate}
\end{prop}
\begin{proof}    

Let $t, g\in G$. Since $K[G]$ is the dual of $\mathrm{Map}(G, K)$, we can consider $t(l_g)$.  We deduce from  \cite{Ch} Section 2  the equality: $$t(l_g)=\varepsilon(t.l_g)=\delta^1_{t, g}=l_g(t).$$ From now on,  for $u\in K[G]$ and $v\in \mathrm{Map}(G, K)$,  with $u=\sum_{t\in G}u_tt$ and $v=\sum_{g\in G}v_gl_g$ we set 
$$<u, v>:=u(v)=v(u)=\sum_{t\in G}u_tv_t.$$ Suppose now that $u\in A^D$. Then $<u, \theta>=\varepsilon(u.\theta)=\lambda u_e$. Since $u.\theta$ belongs to $A$ we obtain  that $\lambda u_e\in R$. Since  $R[G]$ is contained in $A^D$, then $t^{-1}u=
\sum_{s\in G} u_{ts}s$ belongs to $A^D$ for any $t\in G$. Therefore  we know that   $\varepsilon ((t^{-1}u).\theta)=\lambda u_t$ belongs to $R$.  We conclude  that for any $u\in A^D$ then $\lambda u\in R[G]$ and so that $\lambda A^D\subset R[G]$.

We consider the quadratic form   $\kappa_{A^D}$  defined  on $A^{D}$  by :
$$\kappa_{A^D} (u, v)=<S^{D}(u)v, \theta>.  $$
This is a $G$-invariant $R$-non-degenerate form (see \cite{CCMT15} Proposition 5.1).  Indeed $A^D$ and $R[G]$ are both lattices in the non degenerate quadratic space  
$(A^D, \kappa)\otimes _RK$. We consider the discriminant of these lattices.  They satisfy the equality 
\begin{equation}\label{disc}\frak d(R[G])=\frak d(A^D)[A^D:R[G]]^2.\end{equation}  We consider the lattice $R[G]$ and its basis $\{g\in G\}$ over $R$. 
We then have 
$$\frak d(R[G])=\mathrm{det}(\kappa_{A^D}(g, h))R.$$  For $u$ and $v$ in $A^D$, written in  $K[G]$ as  $u=\sum_{t\in G}u_tt$ and $v=\sum_{t\in G}v_tt$, one easily checks  that $\kappa_{A^D}(u, v)=\lambda \sum_{t\in G}u_tv_t$. Therefore  we deduce from this equality that $ \kappa_{A^D}(g, h)=\lambda \delta_{g, h}$ and  so that $\frak d(R[G])=\lambda^nR$. 
Because the form $(A^D,  \kappa_{A^D})$ is non-degenerate we know that $\frak d(A^D)= R$ and so we deduce from (\ref{disc}) that $$\varepsilon _{A}( I(
A))^n=\lambda^nR=[A^D:R[G]]^2.$$ As $n$ is odd we conclude that  that $\varepsilon _{A}( I(
A))$ is the square  of an $R$-ideal.
\end{proof}

\subsection{The Hermitian form associated to a quadratic form}
Let $M$ be a finitely generated locally free left $A^{D}$-module. Since $R[G]\subset A^D$ we may define a $G$-form on $M$ as  an $R$-bilinear symmetric form 
$q:M\times M\rightarrow R$ such that $q( gm_{1},gm_{2}) =q( m_{1},m_{2}) $ for all 
$g\in G,$ and $m_{1},\ m_{2}\in M.$ The form is non-degenerate if $q$  identifies $M$ with
the $R$-linear dual $\mathrm{Hom}_{R}( M,R) $.

The Hopf algebra $A^D$,  endowed with its antipode $S^D$,  is an $R$-algebra with involution. For the sake of simpliciy we set $S^D(x)=\bar x$ for $x$ in $A^D_K=K[G]$. A  Hermitian $ G$-form on $M$ is a biadditive map $h:M\times
M\rightarrow A^{D}$ with the property that for $\nu ,\mu \in A^{D}$ and $m_1, m_2 \in M$ we have 
\begin{equation}
h( \nu m_{1} ,\mu m_{2}) ={\nu }h(m_1, m_2){\overline \mu} \ \ \mathrm{and}\ \ \ h(m_1, m_2)=\overline{h(m_2, m_1)}.
\end{equation}
The form is non-degenerate  if $h$ identifies  $M$ with the $A^{D}$-linear dual $\mathrm{Hom}_{A^{D}}(
M,A^{D}) .$

We denote by $\mathcal{Q}(A^D)$ the category of finitely generated locally free $A^D$-modules supporting a $G$-form and by $\mathcal{H}(A^D)$ the category of finitely generated locally free $A^D$-modules supporting a hermitian $G$-form. We assume that $\Lambda=\varepsilon_A(I(A)$ is a principal ideal generated by $\lambda$.

To a  quadratic $G$-form $q$, we may 
associate  the hermitian $G$-form $h_{q}:M\times M\rightarrow A_K^{D}$ given by 
\begin{equation*}
h_{q}( m_{1},m_{2}) =\lambda ^{-1}\sum_{g\in G}q(
m_{1},g m_{2}) g.
\end{equation*}
 To a  hermitian$\ G$-form $h$\ on $M,\ $we may conversely associate
the quadratic $G$-form $q_{h}: M\times M\rightarrow K$ defined by
\begin{equation*}
q_{h}( m_{1},m_{2}) =\lambda l( h( m_{1},m_{2})) 
\end{equation*}%
where $l$ is the $R$-linear extension of the map $g\longmapsto \delta
_{g,1}$ for $g\in G.$

The relation between quadratic  and hermitian $G$-forms is classical when working over the field $K$ instead of the ring $R$ and  over the group algebra $ K[G]$ instead of the order $A^D$. The situation that we are considering in this section has been studied in a more general set up in \cite{FMc1}. The following proposition is inspired by \cite{FMc1} Theorem 7.1:

\begin{prop}\label{isom} The functor
$$
\fonc{F}{\mathcal{H}(A^D)}{\mathcal{Q}(A^D)}{(M,h)}{(M, q_h)}
$$
is an isomorphism  of categories. An inverse for $F$ is given by
$$
\fonc{F'}{\mathcal{Q}(A^D)}{\mathcal{H}(A^D)}
{(M, q)}{(M, h_q)}.
$$Moreover,  $h$ is non-degenerate if and only if $q_h$ is non-degenerate.

\end{prop}

 \begin{proof} In order to prove Proposition \ref {isom} it suffices to prove the following  two lemmas. 
\begin{lem}\label{herm} If $(M,h)$ is a  hermitian $G$-form, then $(M,q_h)$ is a  $G$-form. Moreover if $h$ is non-degenerate, then $q_h$ is. 
\end{lem}
\begin{proof} For the sake of simplicity we write $q$ for $q_h$. Since $\lambda A^D\subset R[G]$ (see Proposition \ref{H1}) we observe that $\lambda h(x, y) \in R[G]$  and so  $$q(x, y)=\lambda.l (h(x, y))\in R,  \ \forall x,  y\in A^D.$$ One easily deduces from the properties of $h$ that $q$ is a $G$-form.   It remains to prove that if $h$ is non-degenerate then $q$ is non-degenerate.   We consider the map $$
\fonc{\tilde l}{A^D\times A^D}{R}{(x, y)}{\lambda.l(xy).} 
$$
 We  consider   $x
, y\in A^D$ and we  write  $x=\sum_{g\in G}x_gg$ and $y=\sum_{g\in G}y_gg$ with   $x_g, y_g \in K, \forall g\in G$ and  we check the equalities
$$\tilde l(x, y)=\lambda.\sum_{g\in G}x_{g^{-1}}y_g=\tilde l(y, x).$$
Therefore  $\tilde l$ is an $R$-bilinear symmetric form. Moreover, since $\theta=\lambda.l$ is a basis of $A$ as an $A^D$-module it follows from \cite{Ch} Corollary 3.5 that $\tilde l$ is non-degenerate. 

 We now let $\varphi \in \mathrm{Hom}_R(M, R)$;  we set $\tilde \varphi(m):=\lambda ^{-1}\sum_{g\in G}\varphi(g^{-1}m)g$. We note that the map $x\rightarrow \varphi(xm)$ belongs to 
$\mathrm{Hom}_R(A^D, R)$. Since $\tilde l$ is non-degenerate there exists $\alpha(m)\in A^D$ such that $\lambda^{-1}\varphi(xm)=l(\alpha(m)x), \forall x\in A^D$. Therefore 
$$\tilde \varphi (m)=\sum_{g\in G}l(\alpha(m)g^{-1})g=\alpha(m).$$
We conclude that $\tilde \varphi \in \mathrm{Hom}_{A^D}(M, A^D)$. Since $h$ is non-degenerate there exists a unique $n\in M$ such that $\tilde \varphi(m)=h(m, n), \forall m\in M$. By applying $\lambda l$ to each member of this equality we obtain 
$$\varphi (m)=\lambda l(h(m, n))=q(m, n)\ \forall m\in M.$$
We have now shown that  $q$ is non-degenerate. 

\end {proof}

 \begin{lem}\label{quad} If $q$ is a  $G$-form, then $h_q$ is a  hermitian $G$-form. Moreover,  if $q$ is non-degenerate, then $h_q$ is.
\end{lem}
\begin{proof}  We write $h=h_q$. We must first prove that $h$ takes values in $A^D$. We  let $m, n \in M$ and we consider the element of $\mathrm{Hom}_R(A^D, R)$ given by $u\rightarrow q(um, n)$. Using the canonical isomorphism $A\simeq A^{DD}$ we deduce that there exists $f(m, n)\in A$ such that 
\begin{equation}\label{eg0}q(um, n)=<u, f(m,n)>, \forall u\in A^D 
.\end{equation}  Recall  that in the proof of Proposition 2.4 we have set $<u, \theta>:=u(\theta)=\theta(u)$.

It follows from Proposition \ref{eg} that 
\begin{equation}\label{fond}f(m,n)=(\frac{1}{\lambda}\sum_{g\in G}<g^{-1}, f(m,n)>g)\theta.
\end{equation}
Since $q$ is a $G$-form,  we know that $$<g^{-1}, f(m,n)>=q(g^{-1}m, n)=q(m, gn)$$ and so we deduce from (\ref{fond}) that 
$f(m, n)=h(m, n)\theta$. Since $f(m, n)\in A$ we conclude that $h(m, n)\in A^D$.  We check easily that $h( \nu m ,\mu n) ={\nu }h(
m, n){\overline \mu}$. Therefore  we have proved that $(M, h)$ is an object of ${\mathcal{H}(A^D)}$.

We now assume that $q$ is non-degenerate. In order to complete the proof of the lemma we have to prove that  $h$ is non-degenerate. We let  $\varphi \in \mathrm{Hom}_{A^D}(M, A^D)$ and we seek for $n\in M$ such that $\varphi(x)=h(x, n)\ \forall x\in M$. Since $\theta$ is a basis of $A$ as an $A^D$-module, this is equivalent to finding $n\in M$ 
such that $\varphi(x)\theta=h(x, n)\theta, \  \forall x \in M$.  This last equality can be re-written as 
$$<u, \varphi(x)\theta>=<u, f(x, n)>=q(ux, n), \forall u \in A^D $$
(see (\ref{eg0}) for the last equality). Let $u$ be the unit element $1_{A^D}$  of $A^D$, given by $t\rightarrow \varepsilon (t)$. The map $x\rightarrow \varepsilon (\varphi(x)\theta)$ belongs to $\mathrm{Hom}_R(M, R)$. Since $q$ is non-degenerate there exists a unique $n\in M$ such that 
$$\varepsilon (\varphi(x)\theta)=q(x,  n)\  \forall x \in M. $$
Therefore we obtain that 
$$\varepsilon (\varphi(ux)\theta)=q(ux,  n)\  \forall x \in M, \forall u\in A^D, $$
and since $\varphi$ is an $A^D$-morphism 
$$\varepsilon (u\varphi(x)\theta)=q(ux,  n)\  \forall x \in M, \forall u\in A^D. $$
It now follows from \cite{Ch} Section (2.3) that $\varepsilon (u\varphi(x)\theta)=<u, \varphi(x)\theta>$.
Therefore as required, we have found that there exists  a unique $n\in M$ such that 
$$<u, \varphi(x)\theta>=<u, f(x, n> \forall x\in M, \forall u\in A^D$$ and so that $h$ is non-degenerate.
\end{proof}
In order to complete the proof of the proposition we have to prove that $F$ and $F'$ are functorial and mutually inverse. This can be easily checked from the definitions.
\end{proof}

\begin{example} For a generator $\lambda_A$ of $\varepsilon(I(A))$ we have introduced  in \cite{CCMT15} Section 5.1 {\it the unit $G$-form} as the the non-degenerate form $\kappa$ on $A^D$ defined by 
$$\kappa_{A^D} (m, n)=<S^D(m)n, \theta> $$ where $\theta=\lambda_A l$. Following Proposition \ref{isom} we can associate to $\kappa_{A^D}$ a hermitian form on $A^D$ given by 
$$h_{\kappa_{A^D}}(m, n)=\lambda_A^{-1}\sum_{g\in G}\kappa_{A^D}(m, gn)g. $$
One can prove by an easy computation that $h_{\kappa_{A^D}}$ is the rank one unit hermitian form over $(A^D, S^D)$ given by 
$$h_{\kappa_{A^D}}(m, n)=mS^D(n).$$
We note that when the $A^D=R[G]$, then $A=\mathrm{Map}(G, R)$ and so  $\varepsilon (I(A))=R$. Therefore we can take $\lambda_A=1$ and the unit form  $(R[G], \kappa_{A^D})$ is defined by 
$\kappa_{A^D}(g, h)=\delta_{g, h}$ for $g, h \in G$. 

\end{example}

\subsection{The unitary group of a form}
We consider $(A^D, S^D)$ as an $R$-algebra with involution;  we recall that we set $S^D(x)=\bar x$ for $x$ in $A^D_K=K[G]$.
\begin{definition}  For any $R$-algebra $S$,  we  set $$U(A^D\otimes_RS)=\{u\in
(A^{D}\otimes _{R}S)^\times\mid u\overline{u}=1\}.$$
\end{definition}
Let $\mathrm{Aut}( A\otimes _{R}S, Tr_{A\otimes S}^{\prime }) $ be 
 the group of automorphisms   of the $G$-form $(A\otimes_RS, Tr_{A\otimes S}^{\prime })$.  
Since $\theta=\lambda l$ is a basis of $A$ as an $A^D$-module,  any such automorphism $\psi$   is defined by $x_\psi \in A^{D \times}$ such that $\psi(a\theta)=ax_\psi\theta\  \forall  a\in A^D$. 

\begin{lem}\label{aut} The map $\psi\rightarrow x_{\psi}$ induces a group isomorphism
$$\mathrm{Aut}( A\otimes _{R}S, Tr_{A\otimes S}^{\prime })\simeq U(A^D\otimes_RS).$$
\end{lem}
\begin{proof}  The automorphism of $A^D\otimes_RS$-modules  of $A\otimes_RS$ given by $x$ is an automorphism of quadratic $G$-spaces iff 
$$Tr_{A\otimes S}^{\prime }( ax\theta,bx\theta)
=Tr_{A\otimes S}^{\prime }( a\theta,b\theta)\  \forall a,b\in
A^{D}\otimes _{R}S.$$
For the sake of simplicity we take $S=R$ and we set $x=x_{\psi}$.
For elements  $v=a\theta$ and $w=b\theta$ in $A$,  we have the equality
$$Tr'_A(v, w)=\lambda^{-1}\sum_{g\in G} g(a\theta)g(b\theta)=<\theta,\bar ab>=\kappa_{A^D} (a, b)$$ where $\kappa_{A^D}$ is the unit $G$-form.
We note that if   $a$ and $b$ are elements of $A^D\subset K[G]$ with $a=\sum_ta_tt$ and $y=\sum_ub_uu$, then  
$$l(\bar ab)=l(a\bar b)=\sum_va_vb_v.$$ a $\kappa_{A^D}(a, b)=\lambda l(\bar ab)=\lambda l(a\bar b)$.  Therefore $x$ defines an automorphism of the $G$-form $(A, Tr'_A)$ iff 
$$\kappa_{A^D}(ax, bx)=\lambda l(ax\bar x\bar b)=\lambda l(a\bar b)=\kappa_{A^D}(a, b)\  \forall a, b \in A^D$$
and so, since  $\kappa_{A^D}$ is non-degenerate,  iff $x\bar x=1$. Conversely for any $x\in U(A^D)$, the map  $a\theta\rightarrow ax\theta$ defines an automorphism  of the $G$-form $(A, Tr_{A}^{\prime })$.
\end{proof}
In the sequel we shall  frequently identify the groups $\mathrm{Aut}( A\otimes _{R}S, Tr_{A\otimes S}^{\prime }) $ and  $U(A^D\otimes_RS)$;  
we note that when replacing the basis $\theta$ of $A$ by $\theta'=u\theta$,  with $u\in A^{D \times}$, then $x_{\psi}$ is replaced by its conjugate 
$ux_\psi u^{-1}$. 

The functor $S\rightarrow U( A^{D}\otimes_{R}S)$ from the category of commutative $R$-algebras to the category of groups is the functor of points of a scheme over $R$ that we denote by $U_{A^D}$.  According to  Section 2.2,   this is the  group scheme of automorphisms of the rank one hermitian form of $(A^{D}, S^D)$. This is a finitely presented affine group scheme over $R$ which is smooth if $2$ is a unit in $R$ (see \cite {Bayer16}).  

We consider the group scheme $\mathcal{G}:=\mathrm{Spec}(A)$ and for any commutative algebra $S/R$ we identify the group $\mathcal{G}(S)$ of $S$-points of $\mathcal{G}$ with the group $\mathrm{Hom}^{alg}_R(A, S)$. Since $A$ is a finite and projective $R$-module we have an isomorphism of $R$ modules

$$\fonc{\nu}{A^D\otimes_RS}{\mathrm{Hom}_R(A, S)}
{f\otimes s}{(x\rightarrow sf(x)).}$$
By composing the canonical injection $\mathrm{Hom}^{alg}_R(A, S)\rightarrow \mathrm{Hom}_R(A, S)$ with $\nu^{-1}$  we obtain a map 
$u_S: \mathcal{G}(S)\rightarrow A^D\otimes_R S$.
\begin{lem}\label{ind}The map $u_S$ induces a morphism of group schemes 
$$u: \mathcal{G}\rightarrow U_{A^D}.$$
\end{lem}
\begin{proof}We start by proving that $u_S$ induces a group homomorphism. We let $f$ and $ g \in \mathcal{G}(S)$, so that we may write 
$f=\nu(\sum_i f_i\otimes s_i)$ and $g=\nu(\sum_j g_j\otimes t_j)$. The product $f.g \in \mathcal{G}(S)$ is given by $x\rightarrow (f.g)(x)=\sum_{(x)} f(x_{(0)})g(x_{(1)})$ with 
$\Delta(x)=\sum_{(x)} x_{(0)}\otimes x_{(1)}$. Therefore we obtain 
$$(fg)(x)=\sum_{(x)}\sum_{i j}s_it_jf_i(x_{(0)})g_j(x_{(1)})=\sum_{i, j}(s_it_j)(f_ig_j)(x), \  \forall x\in A.$$
This implies that $f.g=\nu((\sum_if_i\otimes s_i)(\sum_jg_j\otimes s_j))$ and so that 
\begin{equation}\label{prod}u_S(fg)=u_S(f)u_S(g).\end{equation} It now  follows from the definition of $u$  that 
$$u_S(\varepsilon)=\varepsilon\otimes1=1_{U_{A^D(S)}}\ \mathrm{and}\ u_S(S^D(f ))=\overline{u_S(f)}.$$
Taking  any  $f\in \mathcal{G}(S)$,  then we know that $fS^D(f)=S^D(f)f=\varepsilon$.  Therefore,  applying $u_S$,  we obtain that 
\begin{equation}\label{inv}u_S(f)\overline{u_S(f)}=\overline{u_S(f)}u_S(f)=1. \end{equation} 
It follows from (\ref{prod}) and (\ref{inv}) that $u_S$ induces a group homomorphism,  as required. Moreover one easily checks that the family of morphisms ${u_S}$ defines a morphism of functors. We conclude that $u$ induces a morphism of group schemes when we use the identification $U(A^D\otimes_RS)=\mathrm{Aut}( A\otimes _{R}S, Tr_{A\otimes S}^{\prime })$

\end{proof}
 We now consider the cohomology sets $H_{fppf}^1(\mathrm{Spec}(R), \mathcal{G})$ and $H_{fppf}^1(\mathrm{Spec}(R), U_{A^D})$. For the sake  of simplicity we set $H^1(\ ,\ ):=H^1_{fppf}(, )$. 
 The morphism $u: \mathcal{G}\rightarrow U_{A^D}$ of  group schemes induces a morphism of pointed sets 
 $$H^1(\mathrm{Spec}(R), \mathcal{G})\rightarrow H^1(\mathrm{Spec}(R), U_{A^D})$$ that we also denote by $u$. The pointed set $H^1(\mathrm{Spec}(R),\mathcal{G})$ classifies the principal homogeneous spaces of $A$ and may therefore be identified with $\mathrm{PH}(A)$; hence we can attach to $B\in \mathrm{PH}(A)$ an element $[B]\in H^1(\mathrm{Spec}(R), \mathcal{G})$.
 
  The pointed set  $H^1(\mathrm{Spec}(R), U_{A^D})$ classifies the twists of the form $(A, Tr')$;  that is to say, non degenerate $G$-forms which are locally isomorphic to  $(A, Tr')$ in the fppf-topology.  Recall that we have seen that   a principal homogeneous space $B$ of $A$,  when endowed with the form $Tr_{B}^{\prime }$, 
is such a twist since the isomorphism  
\begin{equation*}
B\otimes _{R}B\simeq B\otimes _{R}A
\end{equation*}%
shows that $B\otimes_R( B,Tr_{B}^{\prime }) $ and $B\otimes_R(
A,Tr_{A}^{\prime })$ are isomorphic $G$-quadratic spaces
over $B$. We denote by  $[B, Tr_B']$  the element of $H^1(\mathrm{Spec}(R), U_{A^D})$ defined by the $G$-form $(B, Tr_B')$.
\begin{prop}\label{local} For a principal homogeneous space $B$ of $A$,  one has
$$u([B])=[B, Tr'_B].$$

\end{prop}
\begin{proof} Let $B$ be a principal homogeneous space of $A$. As seen before, according to the terminology of Milne \cite{Mi}, then $B$,  as an $A^D$-module,  is a twisted-form of $A$ and $(B, {Tr}'_{B})$, as a $G$-form, is a  twisted-form of $(A, {Tr}'_{A})$. Moreover we observe that $\mathcal{U}:=(\mathrm{Spec}(B)\rightarrow \mathrm{Spec}(R))$ is a flat covering which trivializes simultaneously   $B$ as an algebra with an $A^D$-action  and $(B, {Tr}'_{B})$ as a $G$-form. Therefore $B$ (resp. $(B, {Tr}'_{B})$) defines a class in 
$H^1(\mathcal{U}/R, \mathcal{G})$  (resp. $H^1(\mathcal{U}/R, U_{A^D})$). 

We first recall how to describe  a $1$-cocycle representative of $[B] \in H^1(\mathcal{U}/R, \mathcal{G})$ (see \cite{CCMT15}). Since $B$ is a principal homogeneous space of $A$ there exists an isomorphism of $B$-algebras and $A^D$-modules 
$$\psi: B\otimes_RB\simeq B\otimes_RA. $$
We set $C:=B\otimes_RB$ and we denote by $i_1$ and  $i_2$ the morphisms of $R$-algebras $B \rightarrow C$ respectively defined by $(b \rightarrow b\otimes 1$) and  $(b \rightarrow 1\otimes b)$. We let $$\psi_j : C\otimes_RB\simeq C\otimes_RA$$ be the isomorphism of $C$-algebras and $A^D_C$-modules induced from $\psi$ by scalar extension along $i_j$. We  identify 
$\mathcal{G}(C)$ with $\mathrm{Hom}_C^{alg}(A_C, C)$. Then we know that   $[B]$ is represented by the $1$-cocycle $\varepsilon\circ (\psi_2\circ \psi_1^{-1})\in \mathcal{G}(C)$, where $\varepsilon$ is the counit of $A_C$. 

We now want to describe a representative of $[B, {Tr}'_{B}]\in H^1(\mathcal{U}/R, U_{A^D})$. 
Indeed $\psi_2\circ \psi_1^{-1}$is an automorphism of the $G$-form $(A_C, {Tr}'_{A_C})$. It follows from \cite{Mi} chapter III that it is an element of $U_{A^D}(C)$ which represents 
$[B, {Tr}'_{B}]$. We recall from Lemma 2.1 that we have identified   the groups  
$$U_{A^D}(C)\simeq \{ x\in A^{D \times}_C\\  \mid  x\bar x=1\}$$ 
under the map $\psi\rightarrow 
x_{\psi}$ where $x_{\psi}$ is the unique element of $A^D_C$ such that $\psi(\theta)=x_\psi. \theta$. Therefore,  in order to prove Proposition \ref{local}   it suffices to prove that 
$$\psi_2\circ \psi_1^{-1}(\theta)=(\varepsilon\circ (\psi_2\circ \psi_1^{-1})).\theta.$$
This follows from the next Lemma:
\begin{lem}Let $\varphi$ be an automorphism of $C$-algebras and $A^D_C$-modules of $A_C$, then 
$$\varphi(\theta)=(\varepsilon\circ\varphi).\theta.$$
\end{lem}
\begin{proof} Since  $\varphi$ is a morphism of $A_C$-comodules,  it satisfies the equality
  $$(\varphi\otimes \mathrm{id})\circ\Delta=\Delta\circ \varphi $$ where $\Delta$ is the comultiplication of $A_C$. 
  This implies that  for all  $x\in A_{C}$,
  $$\Delta(\varphi(x))=\sum_{(x)}\varphi(x_{0})\otimes x_{1} $$
  where $\Delta(x)=\sum_{(x)} x_{0}\otimes x_{1}$ and so $\varphi(x)=\sum_{(x)} (\varepsilon\circ\varphi)(x_{0})x_{1}$. We recall that $\theta=\lambda l$, therefore $\Delta(\theta)=
  \lambda\sum_{u\in G}l_u\otimes l_{u^{-1}}$ and so 
  $$\varphi(\theta)=\lambda\sum_{u\in G}(\varepsilon\circ\varphi)(l_u) l_{u^{-1}}=\lambda\sum_{u\in G}(\varepsilon\circ\varphi)( l_{u^{-1}})l_u=(\varepsilon\circ\varphi).\theta$$ as required.
\end{proof}

\end{proof}

\subsection{Explicit bases}
Let $\ R$ be either a local or global ring of integers with field of
fractions $K$. We let $\Sigma $ denote the trace element $\sum_{g\in G}g$ and let 
$l\in \mathrm{Map}( K[ G] ,K) $ be defined \ as the $K$-linear extension of the map $g\mapsto \delta _{g,1}.$ Let $I( A )
$ and $I( A^{D}) $ be the integrals of $A$ and $A^D$ defined in subsection 2.1.   We have seen that
there is an $R$-ideal $\Lambda $ such that $I( A) =\Lambda l$
with $\varepsilon (I( A)) =\Lambda \subset R$ where $
\varepsilon :A_{K}\rightarrow K$ is the counit given by $\varepsilon (
f) =f( 1_{G}) .$ By hypothesis we know that $\Lambda
=R\lambda $ and from \cite{CCMT15} Section 3 (see also Proposition \ref{eg}) we know that $\lambda l$ is a basis of $A$ as an $A^D$-module.
\begin{definition} For $x\in B_K$ we define the resolvent
$$r(x)=\sum_{g\in G} (gx)g^{-1} \in B_K[G].$$
\end{definition}
We note the following two facts for future use: 
\begin{equation}\label{act}
r(hx)=hr(x)\ \forall x\in B_K, h\in G
\end{equation} 
\begin{equation}
r(x)\bar r(x)=\sum_{g\in G}Tr(x(gx))g.
\end{equation} 
\begin{prop}\label{expli}
 Suppose  that $R$ is a ring of integers.\ Let $
 B$ be a principal homogeneous space for $A$. Suppose that $\mathfrak{p}$ is a
prime ideal of $R$  and  that we have an isomorphism
of quadratic $G$-spaces 
\begin{equation*}
\phi _{\mathfrak{p}}:( A_{\mathfrak{p}},Tr_{A_{\mathfrak{p}}}^{\prime
}) \cong ( B_{\mathfrak{p}},Tr_{B_{\mathfrak{p}}}^{\prime
}) .
\end{equation*}%
Set $b_{\mathfrak{p}}=\phi _{\mathfrak{p}}( \lambda l)$,  so that 
$b_{\mathfrak{p}}$ generates $B_{\mathfrak{p}}$ over $A_{\mathfrak{p}}^{D}.\ 
$Then $$r( \lambda ^{-1}b_{\mathfrak{p}}) \in U(B_{\mathfrak{p}} \otimes _{R_{\mathfrak{p}}} A_{\mathfrak{
p}}^{D}) .$$
\end{prop}
\begin{proof} We start by proving that $r( \lambda ^{-1}b_{\mathfrak{p}})$ is a unit in $B_{\mathfrak{p}} \otimes _{R_{\mathfrak{p}}} A_{\mathfrak{
p}}^{D}$. We recall from  Proposition \ref{eg} that for  $f \in A_{\frak{p}}$, then $\lambda^{-1}\sum_{g\in G}f(g^{-1})g$ is the unique element $t \in A_{\frak{p}}^D$ such that $f=t(\lambda l)$ and so that  
 $$\fonc{v}{A_{\frak{p}}}{A_{\frak{p}}^D}
{f}{\lambda^{-1}\sum_{g\in G}f(g^{-1})g}
$$
is an isomorphism of $A^D_{\frak{p}}$-modules; we keep the notation  $v:B_{\mathfrak{p}} \otimes _{R_{\mathfrak{p}}} A_{\mathfrak{
p}}\rightarrow B_{\mathfrak{p}} \otimes _{R_{\mathfrak{p}}} A_{\mathfrak{
p}}^{D}$ for the isomorphism of $B_{\mathfrak{p}} \otimes _{R_{\mathfrak{p}}} A_{\mathfrak{
p}}^{D}$ modules  induced from $v$ by scalar extension. We  now consider the composition of morphisms 
 \[\xymatrix{
B_{\frak{p}}\otimes_{R_{\frak{p}}}B_{\frak{p}}\ar[r]^{u}&B_{\mathfrak{p}} \otimes _{R_{\mathfrak{p}}} A_{\mathfrak{
p}}\ar[r]^{v}
& B_{\mathfrak{p}} \otimes _{R_{\mathfrak{p}}} A_{\mathfrak{
p}}^{D}\\
}\] 
 where $u$ is the isomorphism of $B_{\mathfrak{p}} \otimes _{R_{\mathfrak{p}}} A_{\mathfrak{
p}}^{D}$-modules attached to the structure  of $B_{\frak{p}}$ as a principal homogeneous space of $A_{\frak{p}}$. Let $b$ be a basis of $B_{\frak{p}}$  as an $A_{\frak{p}}^D$-module. Since $v\circ u$ is an isomorphism of $B_{\mathfrak{p}} \otimes _{R_{\mathfrak{p}}} A_{\mathfrak{
p}}^{D}$ free modules of rank $1$, then $(v\circ u)(1\otimes b)$ is a basis of $B_{\mathfrak{p}} \otimes _{R_{\mathfrak{p}}} A_{\mathfrak{
p}}^{D}$ and therefore is a unit. In order to show that $\lambda^{-1}r(b)$ is a unit,  it suffices to prove the equality
\begin{equation}\label{resol}(v\circ u)(1\otimes b)=\lambda^{-1}r(b).\end{equation}
From the definition of $u$ we obtain that $u(1\otimes b)=\sum_{i\in I}b_i\otimes f_i$ with $f_i\in A_{\frak{p}}$ and $b_i\in B_{\frak{p}}$,  where 
$\Delta(b)=\sum_{i\in I}b_i\otimes f_i$, with $\Delta$ being the comodule structure map.  Therefore we have   
\begin{equation}\label{def}(v\circ u)(1\otimes b)=\lambda^{-1}\sum_{i\in I}\sum_{g\in G}b_i\otimes f_i(g^{-1})g=\lambda^{-1}\sum_{g\in G}(\sum_{i\in I}f_i(g^{-1})b_i)\otimes g.\end{equation}
Using that  $g^{-1}b=\sum_{i\in I}f_i(g^{-1})b_i$ we deduce  (\ref{resol}) from (\ref{def}). 

We now complete the proof of the proposition.  We know that $\theta:=\lambda l$ is a basis of $A$ as an  $A^D$-module and so a basis of $ A_{\frak p}$ as an $A^D_{\frak p}$-module for every $\frak{p}$. Suppose we have a local isomorphism of ${G}$-forms
$$\varphi_{\frak p}: (A_{\frak p}, {Tr}'_{A_{\frak p}})\simeq (B_{\frak p}, {Tr}'_{B_{\frak p}}). $$ 
We set $b_{\frak{p}}:= \varphi_{\frak p}(\theta)$. Then we obtain 
$$r(b_{\frak{p}})\bar r(b_{\frak{p}})=\sum_{g\in G} \mathrm{Tr}_{B_{\frak p}}(b_{\frak{p}}(gb_{\frak{p}}))g=\lambda\sum_{g \in G}\mathrm{Tr}'_{B_{\frak p}}(b_{\frak{p}}(gb_{\frak{p}}))g$$$$=\lambda \sum_{g\in G} \mathrm{Tr}'_{A_{\frak p}}(\theta(g\theta))g
=\sum_{g \in G}\mathrm{Tr}_{A_{\frak p}}(\theta(g\theta))g.$$ 
Since we know that $\mathrm{Tr}_{A_{\frak p}}(\theta(g\theta))=\lambda^2\mathrm{Tr}_{A_{\frak p}}(l(gl))$, we conclude that 
$$r(b_{\frak{p}})\bar r(b_{\frak{p}})=\lambda^2$$ and so  $r(\lambda^{-1}b_{\frak{p}})\in U(A^D_{\frak{p}}\otimes B_{\frak{p}})$. 
\end{proof}
\subsection{Determinants}
Here $K$ denotes a field of characteristic zero and $K^{c}$ is a chosen
separable closure of $K$.  For a given finite group $G$ we have the
Wedderburn decomposition
\begin{equation*}
K[ G] =\prod _{\chi }M_{n_{\chi }}( D_{\chi })
\end{equation*}%
where $\chi $ ranges over the irreducible $K$-characters of $G,$ and\ each $
D_{\chi }$ is a division algebra whose center is denoted by $Z_{\chi }.$ The
above decomposition then induces an isomorphism of $K^{c}$-algebras 

\begin{equation*}K^{c}\left[ G\right] =\prod _{\chi ^{\prime }}M_{n_{\chi ^{\prime }}}(
K^{c})
\end{equation*}
where now $\chi ^{\prime }$ ranges over the irreducible $K^{c}$-characters
of $G.$ The determinant induces a homomorphism of groups
\begin{equation}
\mathrm{Det}:K[ G]^{\times}=\oplus _{\chi }GL_{n_{\chi }}(D_{\chi }) \hookrightarrow \oplus _{\chi ^{\prime }}GL_{n_{\chi{\prime
}}}( K^{c}) \overset{\det }{\rightarrow }\oplus _{\chi ^{\prime}}K^{c\times }
\end{equation}
and similarly for each positive integer $m,\ $we have $\mathrm{Det}
:GL_{m}(K[ G]) \rightarrow \oplus _{\chi ^{\prime
}}K^{c\times }$ where explicitly  for $z\in GL_m(K[G])$ and for an irreducible $K^c$-representation $T_{\chi'}$ as above we extend to an algebra homomorphism
$$T_{\chi'}:M_m(K[G]) \rightarrow M_{mn_{\chi'}}(K^c)$$ and we set $\mathrm{Det}(z)(\chi')=\mathrm{det}(T_{\chi'}(z))$. We note that for $\omega\in \Omega_K=\mathrm{Gal}(K^c/K)$ we have 
$\mathrm{Det}(z)(\chi')^\omega=\mathrm{Det}(z)(\chi'^\omega)$.  If we so wish, we may view $\mathrm{Det}$ in $K$-theoretic terms as a homomorphism
\begin{equation*}
\mathrm{Det}:K_{1}(K[G]) \rightarrow K_{1}(
K^{c}[G]) =\oplus _{\chi ^{\prime }}K_{1}(
K^{c}) .
\end{equation*}
We write $\mathrm{Det}(R[G]^\times) $ for the image of $
\mathrm{Det}$ on $R[G] ^{\times }.$ Note that if $R$ is a
semi-local ring then by multiplying on the left and right by elementary matrices
we can always write $x\in GL_{n}(R[G]) $ in the
form $x=e_{1}\delta e_{2}$ where the $e_{i}$ are elementary matrices over $R[G] $ and $\delta $ is diagonal matrix with all but the leading
terms equal to 1; hence in the semi-local case we have shown
\begin{equation}\label{eqn}
\mathrm{Det}(GL_{n}(R[G])) =\mathrm{Det
}(R[G] ^{\times }) .
\end{equation}

Let us also note for future reference that,  if $G$ is abelian, then $\mathrm{
Det}$ is an isomorphism
\begin{equation}\label{isab}
\mathrm{Det}:K[G]^{\times }\rightarrow \mathrm{Det}(K[G] ^{\times }) .
\end{equation}

By the theorem of Hasse-Schilling-Mass \cite{R75} Theorem 33.15  we know:
\begin{thm}
\begin{enumerate}
\item If $K$ is a finite extension of the $p$-adic field $\mathbb{Q}_{p}$, 
 then $$\mathrm{Det}( GL_{m}(K[G]))\cong 
\oplus _{\chi }Z_{\chi }^{\times }$$.

\item If $\ K=\mathbb{R}$,  then 
\begin{equation*}
\mathrm{Det}(K[G]^\times) \cong\oplus _{\chi }Z_{\chi
}^{\times +}
\end{equation*}
where $Z_{\chi }^{\times +}=\mathbb{C}^{\times }$ if $Z_{\chi }=\mathbb{C}, 
Z_{\chi }^{\times +}=\mathbb{R}^{\times }$ if $Z_{\chi }=\mathbb{R}$, 
unless $\chi $ is symplectic in which case $Z_{\chi }^{\times +}=\mathbb{R}
^{\times +}$ the group of positive real numbers.

\item If $\ K=\mathbb{C}$,\ then $$\mathrm{Det}(K[G]^\times)
\cong \oplus _{\chi }\mathbb{C}^{\times }.$$

\item If $K$ is a finite extension of $\mathbb{Q}$ and if $G$
has odd order, then $$\mathrm{Det}(K[G]^\times) \cong \oplus
_{\chi }Z_{\chi }^{\times }.$$
\end{enumerate}
\end{thm}
 Indeed,  from \cite{F83} II Section 2, if $K$ is a number field then 
$\mathrm{Det}(GL_m(K[G])$   may be described as the group of $\Omega_K$-equivariant maps from the virtual $K^c$-characters of $G$ to $K^{c\times}$ whose values on symplectic characters are real and positive at all real infinite places of $K$.  Note that for any extension $K'$ of $K$, with $K'$ assez gros for $G$  we have 
$$\mathrm{Hom}_{\Omega_K}(G_0(K^c[G]), K^{'\times})=\mathrm{Hom}_{\Omega_K}(G_0(K^c[G]), K^{c\times}).$$

For the remainder of this subsection we shall principally be interested in $
\mathrm{Det}(A^{D\times})$ when $R$ is the valuation ring
of a finite extension $K$ of the $p$-padic field $\mathbb{Q}_{p}$.

For our first result we suppose that\ $A^{D}$ is the group ring $R[G]$ and that $L\supseteq K$ are both finite non-ramified extensions of
the $p$-padic field $\mathbb{Q}_{p}.$ We set $\Delta =\mathrm{Gal}(
L/K)$. Then $\Delta$ acts on $\mathrm{Det}(O_{L}[G]^\times) $ by the rule that for $\delta \in \Delta $ and for $
\sum\nolimits_{g\in G}l_{g}g\in O_{L}\left[ G\right] ^{\times }$: 
\begin{equation*}
\mathrm{Det}(\sum_{g\in G}l_{g}g) ^{\delta }=\mathrm{
Det}( \sum_{g\in G}(\delta l _{g})g) .
\end{equation*}
From [T2] we have the Fixed Point Theorem:
\begin{thm}\label {fix}
\begin{equation*}
\mathrm{Det}(O_{L}[G]^\times) ^{\Delta }=\mathrm{Det}
(O_K[G]^{\times}) .
\end{equation*}
\end{thm}

Using (18) we also have the following very general fixed point theorem.

\begin{thm}
Suppose that $G$ is abelian and that $K$ is either a local or global field.
If $N/K$ is a Galois $\Delta$-algebra. 
Then the map $\mathrm{Det}:N[G]^\times  \rightarrow \mathrm{Det}(N[G]^\times)$ is an isomorphism. In particular, if $S$ is subring of $N$ then 
\begin{equation}
\mathrm{Det}(S[G]^\times) ^{\Delta}=(S[G]^\times)
^{\Delta}=S^{\Delta }[G]^\times =\mathrm{Det}(S^{\Delta }[G])^\times) .
\end{equation}
\end{thm}

We conclude this subsection by considering the determinants of unitary
groups. We state two theorems here; we then revisit the topic in much greater
detail in Section 5.

\begin{definition} For  $R$ either global or local we define the subgroup
of minus determinants of $\mathrm{Det}(A^{D\times })$ as:
\begin{equation}
\mathrm{Det} (A^{D\times})_{-}=\{\mathrm{Det}(x)
\in \mathrm{Det}(A^{D\times }) \mid  Det({x}
.\overline{x}) =1\}.\end{equation}
\end{definition}

Since $x\overline{x}=1$ for $x\in U(A^{D})$, it follows
immediately that 
\begin{equation}
\mathrm{Det}(U( A^{D})) \subset \mathrm{Det}(
A^{D\times})_{-}.
\end{equation}

In Section 5 we shall \textit{inter alia} show:

\begin{thm}\label{basieg}
\begin{enumerate}
\item If $K$ is a global field and if the group $G$ has odd order, then 
\begin{equation}
\mathrm{Det}(U( A_{K}^{D})) =\mathrm{Det}(
A_{K}^{D\times })_{-}.
\end{equation}

\item For $R$ local, if the group $G$ has odd order, then 
\begin{equation}
\mathrm{Det}(U( A^{D})) =\mathrm{Det}(
A^{D\times })_{-}.
\end{equation}
\end{enumerate}
\end{thm}
\subsection{ The locally free classgroup} 
Here we suppose that $R$ is the ring of integers $O_{K}$ of a number field $K$. Let $K_{0}(
A^{D})$ denote the Grothendieck group of finitely generated locally
free $A^{D}$-modules; as previously, $\mathrm{Cl}( A^{D})$
denotes the quotient of $K_{0}(A^{D})$ modulo the subgroup
generated by finitely generated free $A^{D}$-classes. This is a finite
abelian group and has a Fr\"{o}hlich description (see 52.17 in \cite{CR2}) 
\begin{equation}\label{cl}
\mathrm{Cl}(A^{D}) =\frac{\mathrm{Det}(\mathbb{A}_{K}[
G] ^{\times }) }{\mathrm{Det}(K[G] ^{\times
}) \cdot \mathrm{Det}(\prod_{\mathfrak{p}}A_{\mathfrak{
p}}^{D\times }) }
\end{equation}
where the product in the denominator extends over the maximal ideals of $R$
and where $\mathbb{A}_{K}$ is the restricted product $\prod_{\mathfrak{
p}}^{\prime }K_{\mathfrak{p}}$. 

{\bf Formation of classes.} Let $M$ be a locally free $A^{D}$-module of
rank $n$ and let $\phi _{\mathfrak{p}}:\oplus _{1}^{n}A_{\mathfrak{p}
}^{D}\cong M_{\mathfrak{p}}$ be isomorphisms of $A_{\mathfrak{p}}^{D}$
-modules for all $\mathfrak{p}\in \mathrm{Spec}(R)$ (which
therefore includes the prime ideal $( 0)$.  Let $\underline{e}=\{e_{1},\cdots ,e_{n}\}$ denote the standard $A^{D}$-basis of $\oplus
_{1}^{n}A^{D}$ and suppose that for each maximal $R$-ideal $\mathfrak{p}$ 
\begin{equation*}
\phi _{\mathfrak{p}}(\underline{e}) =\theta_{\mathfrak{p}}.\phi
_{0}(\underline{e}) \ \ \text{ with  }\theta _{\mathfrak{p}%
}\in GL_{n}(K_{\mathfrak{p}}[G]) .
\end{equation*}

Then the class of $M$ is represented under the above isomorphism  by 
\begin{equation}
\prod_{\mathfrak{p}}\mathrm{Det}(\theta _{\mathfrak{p}}) \in 
\mathrm{Det}(\mathbb{A}_{K}[G ]^{\times }) .
\end{equation}
{\bf Principal homogeneous spaces.} Since $A$ is a locally free rank one $A^{D}$-module, it follows from the defining isomorphism for a principal
homogeneous space $B$ of $A$ (see (\ref{phs})) that $B$ is locally free of rank one over $
A^{D}$.    We therefore have a map 
\begin{equation*}
\psi :\mathrm{PH}(A) \rightarrow \mathrm{Cl}(A^{D}) 
\end{equation*} from the set of isomorphism classes
of such principal homogeneous spaces, denoted $\mathrm{PH}(A)$, 
to $\mathrm{Cl}(A^{D})$ which maps the isomorphism class of $B$ to the $A^D$-class of $B$ minus the $A^D$-class of $A$. 

  From Waterhouse in \cite{w} we have:
\begin{thm}
If the group $G$ is abelian, then $\mathrm{PH}(A) $ is an
abelian group and the map $\psi$ is a group homomorphism.
\end{thm}

\noindent{\bf Remark.}  The construction of $\psi$ may also be deduced from the degeneration of the Leray spectral sequence associated to the morphism of sites  $\mathrm{Spec}(R)_{et}\rightarrow \mathrm{Spec}(A^D)_{et}$. Recall that $\mathcal{G}:=\mathrm{Spec}(A)$.  Since the group $G$   is abelian,  then  $A^D$ is a commutative Hopf $R$-algebra; we set $\pi: \mathrm{Spec}(A^D)\rightarrow \mathrm{Spec}(R)$.  We have a morphism of group schemes $\mathcal{G}\rightarrow \pi_*(\mathbb{G}_{m, A^D})$;  this is the morphism   $\rho: \mathcal{G}\rightarrow U_{A^D}$ (see Lemma \ref{ind}), followed   by $U_{A^D}\rightarrow \pi_*({\mathbb G}_{m, A^D})$. On  the group of points  this morphism   is given by composition of the group homomorphisms $$\mathcal{G}(S)\rightarrow U_{A^D}(S)\rightarrow (A^D\otimes_RS)^\times $$ for any commutative $R$-algebra $S$.  The morphism  $\mathcal{G}\rightarrow \pi_*(G_{m, A^D})$  induces a group homomorphism 
$H^1(R,  \mathcal{G})\rightarrow H^1(R,  \pi_*(\mathbb{G}_{m, A^D}))$. By using  the Leray spectral sequence  we obtain a group homomorphism $H^1(R, \pi_*(\mathbb{G}_{m, A^D}))\rightarrow H^1(A^D, \mathbb{G}_{m, A^D})$.  Since $\pi$ is finite the functor $\pi_{*}$ is exact and therefore  this is an isomorphism 
$$H^1(R, \pi_*(\mathbb{G}_{m, A^D}))\simeq H^1(A^D, \mathbb{G}_{m, A^D}).$$ 
Therefore  $\psi$ can be defined as the group  homomorphism: 
$$\mathrm{PH}(A)=H^1(R, \mathcal{G})\rightarrow H^1(R, \pi_*(\mathbb{G}_{m, A^D}))\simeq \mathrm{Cl}(A^D)$$ after identifying the groups 
$H^1(A^D, \mathbb{G}_{m, A^D})$ and $\mathrm{Cl}(A^D)$. 

We conclude with two results which provide some useful examples for
calculating such class invariants:

\begin{thm}
(See Theorem \ref{CPTA}  and \cite{CPT}.) If $A^{D}$ is the group ring $O_{K}[G] $, then the image of the class map $\psi :\mathrm{PH}(A)
\rightarrow \mathrm{Cl}(O_{K}[G]) $ is killed by
the exponent of $G^{ab}$. In particular, if $G$ is a perfect group, so
that $G^{ab}$ is trivial, then $\psi $ is the zero map.
\end{thm}
From (\ref{cl}) we have:

\begin{thm}
If $G$ is an abelian group and if the Hopf order $A^{D}$ is the split
maximal order $\oplus _{\chi }O_{K} $  in $K[G]$,  then  
\begin{equation*}
\mathrm{Cl}(A^{D}) =\oplus _{\chi }\mathrm{Cl}(O_{K})
\end{equation*}
where\ the direct sum is over the abelian characters $\chi $ of $G.$
\end{thm}
\subsection{The unitary classgroup}
In Section 2.3 we have introduced  the  $R$-group scheme $U_{A^D}$  whose functor of points is defined on commutative algebras $S/R$ by: 
$$U_{A^D}(S)=\{u\in (A^D\otimes_RS)^\times \ \mid \ u\bar u=1\}.$$
We have set $U(A^D\otimes_RS)=U_{A^D}(S)$; in particular we consider $$U(\mathbb A^D_K)=U_{A^D}(\mathbb A_K),\   U(A^D_K)=U_{A^D}(K)\  \mathrm{and}\  U(A^D_{\frak p})=U_{A^D}(O_{K,\mathfrak{p}}).$$  We abbreviate $A^D\otimes_{O_{K }}\mathbb A_K$ by $A^D_{\mathbb A_K}$, $A^D\otimes_{O_{K }}K$ by $A^D_K$\  and $A^D\otimes_{O_{K }}
O_{K, \frak p }$ by $A^D_{\frak p}$. We note the equalities $A^D_{\mathbb A_K}= \mathbb{A}_K[G]$ and $A^D_K=K[G]$. 

The unitary classgroup $\mathrm{CU}(A^{D}) $ is defined as 

\begin{equation}\label{CU}
\mathrm{CU}(A^{D}) =\frac{\mathrm{Det}( U( \mathbb A_K[G]))}{\mathrm{Det}( U(K[G])) .\prod_{\mathfrak{p}}\mathrm{Det}(U( A_{\mathfrak{p}
}^{D})) }.
\end{equation}

Since $U( A^{D}\otimes _{R}S)$ is
a subgroup of $(A^{D}\otimes _{R}S) ^{\times }$ we have the
natural map
\begin{equation}
\xi:\mathrm{CU}(A^{D}) \rightarrow \mathrm{Cl}(
A^{D}) .
\end{equation}
{\bf Formation of classes and unitary classes.}
With the notation of the previous subsection we suppose further that $M$ is a locally free rank
one $A^{D}$-module which supports a non-degenerate ${G}$-form $q: M\times M\rightarrow R=O_{K}$\ and that for all $
\mathfrak{p}\in \mathrm{Spec}(O_{K}) $ we have isomorphisms of ${G}$-quadratic spaces 

$$j_{\mathfrak{p}}:(A_{\mathfrak{p}},Tr_{A_{
\mathfrak{p}}}^{\prime })\cong ( M_{\mathfrak{p}},q_{\mathfrak{p}}).$$ 
Then, for each maximal ideal $\mathfrak{p}$ of $
O_{K}, $ we have the automorphism $j_{0}^{-1}\circ j_{\mathfrak{p}}$ of the $G$-form $  (A_{K_{\mathfrak{
p}}},Tr_{A_{K\mathfrak{p}}}^{\prime })$ over $K_{\mathfrak{p}}$. It follows from Lemma \ref{aut} that 
\begin{equation*}
\theta _{\mathfrak{p}}=j_{0}^{-1}\circ j_{\mathfrak{p}}\in U(K_{\mathfrak{%
p}}[G])
\end{equation*}%
with almost all $\theta _{\mathfrak{p}}\in U(A_{\mathfrak{p}}^{D}).~$\ 

 The unitary class of $( M,q) $ in $\mathrm{CU%
}( A^{D})$,  denoted $\phi( M,q) ,$ is defined to be the class represented by  
$\prod_{\mathfrak{p}}$ $\mathrm{Det}( \theta _{\mathfrak{p}%
})$. 
\medskip

{\bf The relationship between locally free and unitary classes}. Let $
\mathrm{PH}^{\prime }(A) $ denote the subset of isomorphism classes of principal
homogeneous spaces which are locally isomorphic as ${G}$-quadratic
spaces to $( A,Tr_{A}^{\prime })$; that is to say the principal
homogeneous spaces $B$ of $A$ with the property that for all $\mathfrak{q}
\in \mathrm{Spec}(O_{K})$ we have isometries $j_{\mathfrak{q}
}:(A_{\mathfrak{q}},Tr_{A_\mathfrak{q}}^{\prime }):\cong (B_{\mathfrak{q
}},Tr_{B_\mathfrak{q}}^{\prime })$.  We then have a natural map of sets $
\ \phi :\mathrm{PH}^{\prime }(A) \rightarrow $ $\mathrm{CU}(A^{D}) \ $which maps $B~$in $\mathrm{PH}^{\prime }(
A) $ to the unitary class of $(B,Tr_{B}^{\prime })$  in $\mathrm{CU}( A^{D})$. 

We can then draw together the map $\phi$,  the map $\mathfrak{\xi }$
and the map $\psi$ in the commutative triangle of pointed sets:
\begin{equation*}
\begin{array}{ccc}
\mathrm{PH}^{\prime }(A) & \overset{\phi }{\rightarrow } & 
\mathrm{CU}(A^{D}) \\ 
& \searrow \psi & \downarrow \xi  \\ 
&  & \mathrm{Cl}(A^{D}).
\end{array}
\end{equation*}

{\bf Higher ranks.}
 We shall mainly be interested in the case where $M$
has $A^{D}$-rank one. However, there are some places  where the locally free $A^{D}$-modules we consider can have 
rank $m$ greater than $1$.  We therefore note here how the above
construction generalizes to the case of higher rank. So now we suppose that 
$M$ supports a non-degenerate $G$-invariant form $h:M\times M\rightarrow R
$ and that for all $\mathfrak{p}\in \mathrm{Spec}(O_{K}) $ we
have isomorphisms of $G$-quadratic spaces $j_{\mathfrak{p}}: (A_{\mathfrak{p}},Tr_{A_{\mathfrak{p}}}^{\prime })^{\perp{m}}\cong
(M_{\mathfrak{p}},h_{\mathfrak{p}})$.   Then for each
maximal ideal $\mathfrak{p}$ of $O_{K}$ we have
\begin{equation*}
\theta _{\mathfrak{p}}=j_{0}^{-1}\circ j_{\mathfrak{p}}\in \mathrm{Aut}((A_{K_{\mathfrak{p}}},Tr_{A_{K\mathfrak{p}}}^{\prime
})^{\perp{m}})
\end{equation*}
with almost all  $\theta _{\mathfrak{p}}\in \mathrm{Aut}((A_{
\mathfrak{p}},Tr_{A_{\mathfrak{p}}}^{\prime })^{\perp{m}})$. 

For any commutative algebra $S/R$ the involution $x\rightarrow \bar x$ of $A^D$ induces an involution  on $A^D\otimes_RS$ and so yields  an involution   on $M_m(A^D\otimes_RS)$  defined by 
$$(c_{r,s})\rightarrow (\bar c_{s,r})$$ where $c_{r, s}$ is the $r,s$ entry of a matrix $C$. We refer to this involution as the extended involution. For the sake of simplicity we use the notation  $\sigma$ for the involution $x\rightarrow \bar x$ on $A^D$ and for its various extensions.  We set 
$$U(M_m(A^D\otimes_RS))=\{C\in GL_m(A^D\otimes_RS)\ \arrowvert\ C\sigma( C)=I_m\}. $$ 

Lemma \ref{aut} generalizes in higher dimensions and provides us with a group isomorphism 
\begin{equation}\label {h_0} \mathrm{Aut}((A\otimes_R S
,Tr_{A\otimes_R S}^{\prime })^{\perp{m}})\simeq U(M_m(A^D\otimes_RS)). \end{equation} We identify these groups. As for $m=1$, the functor $S\rightarrow U(M_m(A^D\otimes_RS))$ is the functor of points of a finitely presented affine group scheme  over $R$ that we denote by $U_{m, A^D}$. We let $U_{m,G}$ be its generic fiber. 

When the group $G$ is of odd order, it has no non-trivial absolutely  irreducible orthogonal or symplectic representations . Therefore,  the decomposition of   $K[G]$  into a product of simple algebras leads to the decomposition of $(K[G], \sigma)$  into a product of indecomposable algebras with involution
 which can be written  
\begin{equation}\label{c1}K[G]= K\times \prod_{i\in I}A_i\prod_{j\in J}(A_j\times   A_j^{\mathrm {op}})\end{equation}
where we denote by $\{A_i \}_{i\in I}$ the simple components of $K[G]$, different from $K$,  stable by $\sigma$ and by  $\{B_j:=A_j\times A_j^{\mathrm {op}} \}_{j\in J}$ the products of simple components of $K[G]$ interchanged by $\sigma$.  From any integer $m\geq 1$,  the decomposition  (\ref{c1}) yields to a similar decomposition of $M_m(K[G])$: 
\begin{equation}\label{c2}
M_m(K)\times \prod_{i\in I}M_m(A_i)\prod_{j\in J}(M_m(A_j)\times   M_m(A_j)^{\mathrm {op}})\end{equation}
(using  the isomorphism of algebras $M_m(A_j^{\mathrm {op}})\rightarrow M_m(A_j)^{\mathrm {op}}$ given by $X\rightarrow {^t}X$).
 The extension of the identity on $K$ induces the transposition $X\rightarrow {^t}X$ on $M_m(K)$. For $i\in I$ and $j\in J$ we set $A_{m,  i}=M_m(A_i), A_{m,j}=M_m(A_j)$  and $ B_{m,j}=A_{m,j} \times A_{m, j}^{op}$ and we  denote by $\sigma_i$ and $\sigma_j$ the extended involutions to $A_{m, i}$ and $B_{m, j}$. We note that the involutions $\sigma_i$  and $\sigma_j$ are of second kind (see \cite {KMRT} Section 2) and we refer to  
 $\{A_{m,i}, \sigma_i\}_{i\in I}$ (resp. $\{B_{m,j}, \sigma_{j}\}_{j\in J}$) as the  unitary (resp. split unitary) components of $(M_m(K[G]), \sigma)$. This  decomposition induces a decomposition of $U_{m,G}$ into a product of algebraic groups that we use in Section 6.1. 
\vskip 0.1 truecm

In order to attach to $(M, h)$  an element of $CU(A^D)$ we use the lemma:
\begin{lem}\label{detred} The following equalities hold:
\begin{enumerate}
\item  $\mathrm{Det}(U(M_m(A^D_{
\mathfrak{p}}))=\mathrm{Det}(U(A^D_{
\mathfrak{p}}))$.
\item $\mathrm{Det}(U(M_m(A^D_{
K}))=\mathrm{Det}(U(A^D_{K}))$.
\end{enumerate}
\end{lem}
\begin{proof}  We prove the first statement;  the proof of the second one is similar. One observes that the arguments  used 
in the proof of (\ref{eqn}) hold for any $R$-Hopf order in $K[G]$
 and so that we have 
 $$\mathrm{Det}(\mathrm{GL}_m(A^D_{\mathfrak{p}}))=\mathrm{Det}(A^{D \times}_{\mathfrak{p}}).$$ 
 Therefore  

\begin{equation}\label{h1}
\mathrm{Det}(U(M_m(A^D_{
\mathfrak{p}}))\subset \mathrm{Det}(\mathrm{GL}_m(A^D_{\mathfrak{p}}))_-= \mathrm{Det}(A^{D \times}_{\mathfrak{p}})_-.\end{equation}
Using Theorem \ref{basieg} we obtain the chain of inclusions 
\begin{equation}\label{h2}\mathrm{Det}(A^{D \times}_{\mathfrak{p}})_-=\mathrm{Det}(U(A^D_{\mathfrak{p}}))\subset \mathrm{Det}(U(M_m(A^D_{
\mathfrak{p}})).\end{equation}
The required equality follows from (\ref{h1}) and (\ref{h2}).

\end{proof}

It follows from  (\ref{h_0}) and Lemma \ref{detred}  that for any $\mathfrak{p}$, then $\mathrm{Det}(\theta_{\frak{p}})$ belongs to 
$$ \mathrm{Det}(U(M_m(K_{
\mathfrak{p}}[G])))= \mathrm{Det}(U(K_{\frak{p}}[G])$$ 
with almost all $\mathrm{Det}(\theta_{\frak{p}})\in  \mathrm{Det}(U(M_m(A^D_{
\mathfrak{p}})))=\mathrm{Det}(U(A^D_{\frak{p}}))$.

We define  the unitary 
class of $(M,h)$ as the element of  $\mathrm{CU}(A^{D}) $ represented 
 by 
\begin{equation*}
\prod_{\mathfrak{p}}\mathrm{Det}(\theta _{\mathfrak{p}
}) \in \mathrm{Det}(U(\mathbb{A}_K[G])) .
\end{equation*}

{\bf Orthogonal sums.}
 Suppose we are given $G$-forms 
 $(M_{1},h_1)$ and $(M_{2},h_{2})$ with local isomorphisms of $G$-quadratic spaces 
\begin{equation*}
j_{\mathfrak{p}}^{( 1) }:(A_{\mathfrak{p}%
},Tr'_{A_{\mathfrak{p}}})^{\perp{n_1}}\cong (M_{1,\mathfrak{p}%
},h_{1,\mathfrak{p}}) \ \text{and\ }j_{\mathfrak{p}}^{(2)
}:(A_{\mathfrak{p}%
},Tr'_{A_{\mathfrak{p}}})^{\perp{n_2}}\cong ( M_{2,\mathfrak{p}},h_{2,\mathfrak{p}})
\end{equation*}
with the unitary classes of $\left( M_{1},h_{1}\right) $ resp. $(
M_{2},h_{2}) $ represented by $\prod_{\mathfrak{p}}\mathrm{
Det}(\theta _{\mathfrak{p}}^{(1)}) $ resp. $\prod_{
\mathfrak{p}}\mathrm{Det}( \theta _{\mathfrak{p}}^{(2)})$.  Then we
have local isometries on the orthogonal sum $( M_{1},h_{1}) \bot
( M_{2},h_{2}) $ 
\begin{equation*}
j_{\mathfrak{p}}^{( 1) }\bot j_{\mathfrak{p}}^{( 2)
}:(A_{\mathfrak{p}%
},Tr'_{A_{\mathfrak{p}}})^{\perp{n_1}}\perp(A_{\mathfrak{p}%
},Tr'_{A_{\mathfrak{p}}})^{\perp{n_2}}\cong ( M_{1,\mathfrak{p}},h_{1,\mathfrak{p}}) \bot ( M_{2,
\mathfrak{p}},h_{2,\mathfrak{p}}) \ \ 
\end{equation*}%
and so the unitary class of $( M_{1},h_{1}) \bot (
M_{2},h_{2}) $ is seen to be represented by $\prod_{
\mathfrak{p}}\mathrm{Det}( \theta _{\mathfrak{p}}^{(1)}) .\mathrm{
Det}( \theta _{\mathfrak{p}}^{(2)}) .$

\section{The local case}
In this section our hypotheses are slightly more general than in the rest of the paper. We consider a  finite and flat group scheme $\mathcal{G}:=\mathrm{Spec}(A)$ over a ring $R$. We let $A^D$ be the $R$-linear dual of $A$,  endowed with its structure of $R$-Hopf algebra   
and  we denote by $S^D$ the antipode of $A^D$.   We let $U_{A^D}$ be the group scheme  over $R$,  given by its functor of points on commutative $R$-algebras $S$: 
$$S\rightarrow U_{A^D}(S)=\{x\in (S\otimes_R A^D)^\times\/ \mid xS^D(x)=1\}, $$ (see Section 2.3).  If $2$ is invertible in $R$ then this is the  {\it unitary  group scheme} associated with the $R$-algebra with involution $(A^D, S^D)$. By Lemma \ref{ind} we know that the natural morphism of group schemes $\mathcal{G}\rightarrow U_{A^D}$ induces a map of pointed sets  
$$u: H^1(R, \mathcal{G})\rightarrow H^1(R,  U_{A^D}). $$  When  $R$ is a field $K$ of characteristic different from $2$ and $\mathcal {G}$ is  the  constant group scheme attached to  a finite  group $G$, then the algebra $A^D$ is the group algebra $K[G]$ and $S^D$ is the involution of $K[G]$ defined by $g\rightarrow g^{-1}$   on  the elements of $G$.  Since $\mathcal{G}$ and $U_{A^D}$ are both smooth group schemes over $K$, the map $u$ can be understood as a  map 
$$H^1(G_K,  G)\rightarrow H^1(G_K, U_{K[G]}(K^s))$$
between non-abelian cohomology sets,  where $G_K$ acts trivially (resp. by Galois action) on $G$ (resp.  $U_{K[G]}(K^s)$). The elements of $H^1(G_K,  G)$ correspond to the isomorphism classes of $G$-Galois algebras  over $K$,  while those of $H^1(G_K, U_{K[G]}(K^s))$ correspond to the isomorphism classes of $G$-quadratic forms,  isomorphic to the unit $G$-form after scalar extension by $K^s$. The map $u$ associates to any $G$-Galois algebra $L$ over $K$  the quadratic space $(L, Tr_L)$,  where we let $Tr_L$ be  trace form on $Tr_{L/K}$.  Bayer and Lenstra have proved in \cite{Bayer90} that if the group $G$ is of odd order any $G$-Galois algebra has a self-dual normal basis. This result  is equivalent to the triviality of the map $u$ in this case. Our aim is to prove an "integral version" of this theorem in a local set-up where we replace $K$ by a local Henselian ring $R$ and $G$ by a finite and flat group scheme $\mathcal{G}$. 
\begin{thm}\label{main}  Let $R$ be a henselian local ring with perfect  residue field $k$  of  positive characteristic and let $\mathcal G:=\mathrm{Spec}(A)$ be a finite  and flat group scheme over $R$. We assume that $\mathcal{G}$ is of odd order. Then the natural map of pointed sets
 
$$H^1(R, \mathcal{G})\rightarrow H^1(R, U_{A^D})$$
is trivial in the following cases: 
\begin{enumerate}
\item The characteristic of $k$ is odd.
\item The characteristic of $k$ is $2$ , $R$ is an integral   domain and $\mathcal{G}$ is generically constant.
\end{enumerate}  

\end{thm}

\begin{cor}\label{locod}Let $R$ be a henselian local ring with perfect  residue field $k$  of  positive characteristic and let $\mathcal G:=\mathrm{Spec}(A)$ be a finite  and flat group scheme over $R$. We assume that $R$ is an integral domain and that $\mathcal{G}$ is of odd order and generically constant. Then,  for any principal homogeneous space $B$ for $A$, there exists an isomorphism of $G$-forms 
$$(B, Tr'_B)\simeq (A, Tr'_A).$$
\end{cor}
\begin{proof} The corollary follows from Theorem \ref{main}  and Proposition  \ref{local}
\end{proof}

\begin{rem} Under the hypotheses of the corollary we have seen that the $G$-forms $(A, Tr'_A)$ and the unit form $(A^D, \kappa_{A^D})$ are isomorphic (see Remark 1.12). Therefore,  for any principal homogeneous space for $B$ for $A$,   the $G$-form $(B, Tr'_B)$ is isomorphic to the unit form. 
\end{rem}

\begin{proof} [Proof of Theorem \ref{main}]
We let $\mathcal{G}_k:=\mathrm{Spec}(A_k) $ be the special fiber of $\mathcal{G}$.  We consider the commutative square of pointed sets. 
 \[\xymatrix{
H^1(R , \mathcal{G})\ar[r]^{u}\ar[d]_{t}&H^1(R, U_{A^D})\ar[d]_{s} \\
\textsc H^1(k, \mathcal{G}_k)\ar[r]^{u_k}&H^1(k, U_{A_k^D}). \\
}\] 

 We start by proving part (1) of Theorem  \ref{main} in the special case where $k$ is finite; we shall complete the proof for a perfect field $k$ after the proof of Lemma \ref{et0}. So now $2$ is invertible in both $k$ and  $R$. The group scheme  $U_{A^D}$ is smooth  and   of finite type  \cite{Bayer16} Appendix A and so, since $R$ is henselian,  the map $s$ is injective. Therefore in order to prove Theorem \ref{main} (1) it suffices to prove that $u_k$ is trivial. 

First we note that if  the group   $\mathcal{G}$ is commutative  and $k$ is finite then the result can  be easily deduced from \cite{Bayer94} Proposition 2.3.2.  Since  the groups involved are commutative,    the morphisms of pointed sets are group homomorphisms. If  $\mathcal{G}_k$ is a  commutative group scheme of order $n$,  the group   $H^1(k, \mathcal{G}_k)$ is annihilated by $n$. We let   $U^0_{A_k^D}$ be   the neutral component of   $U_{A_k^D}$. By  \cite{Bayer94} Proposition 2.3.2   we know that $U_{A_k^D}/U^0_{A_k^D} $ is an $2$-elementary abelian group and that $H^1(k, U_{A_k^D})$ injects into  $H^1(k, U_{A_k^D}/U^0_{A_k^D})$ if $k$ is a finite field. Therefore the group  $H^1(k, U_{A_k^D})$ injects into  a group annihilated by $2$.  Since  $n$ is odd,  we conclude that $u_k$ must be trivial.  

We now return to  the general case.  We consider the connected-\'etale exact sequence over $k$
\[\xymatrix{1\ar[r]^{}&\mathcal{G}_k^0\ar[r]^{a}&\mathcal{G}_k\ar[r]^{b}&\mathcal{G}_k^{et}\ar[r]^{}& 1.\\
}\]
To this sequence there is associated an exact sequence of pointed sets 
\[\xymatrix{\ar[r]^{}&H^1(k, \mathcal{G}_k^0)\ar[r]^{\tilde a}&H^1(k,\mathcal{G}_k)\ar[r]^{\tilde b}&H^1(k, \mathcal{G}_k^{et})\ar[r]^{}& \\
}\]

\begin{lem}\label{dijano} The map  of pointed sets  $\tilde b: H^1(k,\mathcal{G}_k)\rightarrow H^1(k, \mathcal{G}_k^{et})$  is a bijection. 
\end{lem}\begin{proof} Since $\mathcal{G}_k^0$ is a finite connected group scheme over a perfect field we know that $H^1(k,  \mathcal{G}_k^0)$ is trivial (see Appendix )  and so that the kernel of $\tilde b$ is trivial. 
Moreover, the morphism $b: \mathcal{G}_k\rightarrow \mathcal{G}_k^{et}$ has a section that we denote by $b'$  (see \cite{Tate} Section 3.7). Thus $b'$ induces a map of pointed sets 
\[\xymatrix{H^1(k, \mathcal{G}_k^{et})\ar[r]^{\tilde {b'}}&H^1(k,\mathcal{G}_k)\ar[r]^{\tilde b}&H^1(k, \mathcal{G}_k^{et})& \\
}\] such that ${\tilde b}\circ {\tilde {b'}}=\mathrm{id}$.  Hence we deduce that $\tilde b$ is a surjection. Our aim is now to prove that $\tilde b$ is an injection. Let  $x, y \in H^1(k, \mathcal{G}_k)$ such that $\tilde b(x)=\tilde b(y)$. We let $P$ be a $\mathcal{G}_k$-torsor which represents $x$. Following  \cite{Giraud} Chapter III we let 
\[\xymatrix{1\ar[r]^{}&{\mathcal{G}_{k, P}^0}\ar[r]^{c}&\mathcal{G}_{k,P}\ar[r]^{d}&\mathcal{G}_{k,P}^{et}\ar[r]^{}& 1\\ 
}\] be the exact sequence obtained by twisting the connected-\'etale sequence by the torsor $P$. We obtain from  \cite {Giraud}  Corollary 3.3.5. a commutative diagram

\[\xymatrix{H^1(k, \mathcal{G}_k^0)\ar[r]^{}&
H^1(k, \mathcal{G}_k)\ar[r]^{\tilde b}\ar[d]_{\theta_P}&H^1(k, \mathcal{G}_k^{et})
\ar[d]_{\theta}& \\
H^1(k, \mathcal{G}_{k,P}^0)\ar[r]^{}&H^1(k, \mathcal{G}_{k, P})\ar[r]^{\tilde d}&H^1(k, \mathcal{G}_{k,P}^{et})\\
}\]
where the horizontal rows are exact sequences of pointed sets and where $\theta_P$ and $\theta$ are bijections. From the very definition of  $\mathcal{G}_{k, P}^0$ we know that  there exists 
an algebraic variety $X$ over $k$ such that $\mathcal{G}_{k, P}^0\times_k X\simeq \mathcal{G}_k^0\times_k X$. By considering the residue field of a closed point of $X$ we obtain a finite extension 
$k'/k$ such that $\mathcal{G}_{k, P}^0\times_k k'\simeq \mathcal{G}_{k}^0\times_kk'$.  Since any connected group scheme is geometrically connected we deduce that $\mathcal{G}_{k, P}^0\times_k k'$  and thus $\mathcal{G}_{k, P}^0$ are both connected.  Because $\mathcal{G}_{k,P}^0$ is connected, the pointed set $H^1(k, \mathcal{G}_{k,P}^0)$ is trivial and so the kernel of $\tilde d$ is trivial. We can now complete the proof of the lemma. From the equality $\tilde b(x)=\tilde b(y)$ and the commutativity of the diagram we deduce that 
$(\tilde d\circ \theta_P)(x)=(\tilde d\circ \theta_P)(y)$. From the  definition of the twisted exact sequence we know that $\theta_P(x)=0$ since $P$ represents $x$  and thus,  since $\tilde d $ is a morphism of pointed sets,  we deduce  that $(\tilde d\circ \theta_P)(x)=0=(\tilde d\circ \theta_P)(y)$. Since  the kernel of $\tilde d$ is trivial, $\theta_P(x)=0=\theta_P(y)$ and because  $\theta_P$ is a bijection  we conclude that $x=y$,  
as required. 

\end{proof}
We now return to the morphism of group schemes $b': \mathcal{G}_k^{et}\rightarrow \mathcal{G}_k$ introduced earlier  and the map of pointed sets 
$\tilde {b'}: H^1(k,  \mathcal{G}_k^{et})\rightarrow H^1(k, \mathcal{G}_k)$ that it induces. Since we know from Lemma \ref{dijano}  that $\tilde b$ is a bijection  such that $\tilde b\circ \tilde{b'}=\mathrm{id}$, 
we deduce that $\tilde{b'}$ is a bijection.  The morphism  $b': \mathcal{G}_k^{et}\rightarrow \mathcal{G}_k$ is induced  by  a morphism of Hopf algebras $A_k\rightarrow A_{k,et}$ which itself induces a morphism of Hopf algebras $(A_{k,et})^D\rightarrow A_k^D$.  Therefore we get a commutative diagram 
 \[\xymatrix{
H^1(k ,  \mathcal{G}_k^{et})\ar[r]^{u_{k, et}}\ar[d]_{\tilde{b'}}&H^1(k, U_{(A_{k, et})^D})\ar[d]_{v} \\
\textsc H^1(k, \mathcal{G}_k)\ar[r]^{u_k}&H^1(k, U_{A_k^D}) \\
}\] 
where  $\tilde{b'}$ is a bijection. Therefore,  in order to show that $u_k$ is trivial,  it suffices to prove that $u_{k, et}$ is. We note that since   $\mathcal{G}$ is of odd order then both $\mathcal{G}_k$ and  $  \mathcal{G}_k^{et} $ are of odd order. We conclude that,  in order to prove the triviality of $u_k$,  it suffices to prove that this property holds  when the group $\mathcal{G}_k$ is \'etale. We  complete the proof Theorem \ref{main}, when $2$ is a unit in $R$,   by proving the following lemma:
\begin{lem}\label {et0} Let $k$ be a finite field of characteristic different from $2$ and let $\mathcal{G}_k:=\mathrm{Spec}(A_k)$ be a finite,  \'etale group scheme over $k$, of odd order,  then the natural map of pointed sets 
$$u_k: H^1(k, \mathcal{G}_k)\rightarrow H^1(k, U_{A^D_k})$$ is trivial. 
\end{lem}
\begin{proof} We consider the composition of morphisms of group schemes 
 \[\xymatrix{
\mathcal{G}_k\ar[r]^{f_k}&U_{A^D_k}\ar[r]^{g_k}
& U_{A^D_k}/U^0_{A^D_k}.\\
}\] We set $h_k:=g_k\circ f_k$. Since $\mathcal{G}_k$ is of odd order  and  since $U_{A^D_k}/U^0_{A^D_k}$ is a $2$-abelian elementary group,  we deduce that $h_k$ is trivial. Therefore this implies  that 
$f_k$ factorizes into a morphism  $\mathcal{G}_k\rightarrow U^0_{A^D_k}$ followed by a closed immersion $ U^0_{A^D_k}\rightarrow U_{A^D_k}$. Moreover, since $\mathcal{G}_k$ is \'etale it is reduced  and so  the morphism $\mathcal{G}_k\rightarrow U^0_{A^D_k}$ factors in a unique way into a morphism $\mathcal{G}_k\rightarrow {U^{'0}_{A^D_k}}$ followed by  $U^{'0}_{A^D_k}\rightarrow U^0_{A^D_k}$, where $U^{'0}_{A^D_k}$ is the reduced group scheme attached to ${U^{0}_{A^D_k}}$. Putting this together we deduce that
 the factorization  of $f_k$ shows  that  $u_k$ may be factorized as  a composition of maps of pointed sets
$$u_k: H^1(k, \mathcal{G}_k)\rightarrow H^1(k, U^{'0}_{A^D_k})\rightarrow  H^1(k, U_{A^D_k}).$$ We note that,   since $U^0_{{A^D_k}}$ is  connected,  it follows that  the group scheme $U^{'0}_{A^D_k}$ is also connected. Moreover  this group scheme is smooth since it is reduced by definition and so geometrically reduced since $k$ is perfect. It now follows from a theorem of Steinberg that $H^1(k, U^{'0}_{A^D_k})=1$ (see \cite{Serre94} Theorem 4).   Therefore we conclude that $u_k$ is trivial. This completes the proof of Theorem \ref{main} (1) when $k$ is finite.
\end{proof}  

We now prove part (2) of Theorem \ref{main}. So now $k$ is perfect and has characteristic $2$ and  we  consider  the group   $\mathcal{G}$ over $R$, of odd order $n$, which is generically constant. Since $n$ is a unit in $R$, the group  scheme $\mathcal{G}$ is  a  constant group scheme and so   there exists a finite group $G$ of order $n$ such that   $\mathcal{G}=\mathrm{Spec}(A)$ with $A=\mathrm{Map}(G,R)$ and $A^D=R[G]$; for reasons of simplicity we will write $G$ instead of $\mathcal{G}$ in the remainder of the proof. We note that,  since $G$ has odd order,  $R[G]$ is a maximal order in $K[G]$. We set  $U_{G, R}:=U_{A^D}$ and we  let $u: H^1(R, G)\rightarrow H^1(R,  U_{G, R})$ be the map of pointed sets induced by the  morphism  $s: G\rightarrow U_{G, R}$. 
 
 Since $2$ is not a unit in $R$,  the group scheme $U_{G,R}$ is not smooth. Following Serre,  we know that it decomposes into a product $U_{G, R}=U'_{G, R}\times \mu_2$ where $U'_{G, R}$ is a smooth  group scheme.  We consider the morphism of group schemes $p\circ s: G\rightarrow \mu_2$ where $p$ is the projection $U_{G, R}\rightarrow \mu_2$. Since $G$ is of odd order we deduce that $p\circ s$ is trivial. Therefore  $s$ factors through $G\rightarrow U'_{G, R}$ followed by the closed immersion  $U'_{G, R}\rightarrow U_{G, R}$. This factorisation induces a decomposition of  $u$. The  base change along $\mathrm{Spec}(k)\rightarrow \mathrm{Spec}(R)$ leads us to the following commutative diagram:
\[ \xymatrix{
H^1(R, G)\ar[r]^{u'}\ar[d]^{t} & H^1(R, U'_{G, R})\ar[r]^v\ar[d]^{s'}&H^1(R, U_{G,R})\ar[d]^{}  \\
H^1(k, G)\ar[r]^{u'_k}& H^1(k, U'_{G, k})\ar[r]^{v_k}& H^1(k, U_{G, k}).
} \]
Since $U'_{G, R}$ is smooth the map  $s'$ is an injection. Because $k$ is a perfect field  we know from \cite{Serre14} Theorem 5.1.2 that $U'_{G, k}$ coincides with the reduced algebraic group associated to $U_{G,k}$.  We let  $U'^{0}_{G, k}$ be the connected component of $U'_{G, k}$.  The group scheme $G$ being reduced,  we know that  the morphism  $G\rightarrow U_{G, k}$ factorises through $G\rightarrow U'_{G, k}$. Since $G$ is of odd order every element of $G$ is a square. Therefore it follows from Serre \cite{Serre14} Section 5.3 Theorem 5.3.1 that $G\rightarrow U'_{G, k}$ factorises through $G\rightarrow U'^{0}_{G, k}$. This implies  that $u'_k$ factorises  through  the map of pointed sets $H^1(k, G)\rightarrow H^1(k, U'^0_{G, k})$. Since by Serre \cite{Serre14} Section 4 Theorem C, we know that $H^1(k, U'^0_{G, k})=0$ and so we conclude that $u'_k$ is trivial. The map $s'$ being injective we know that $u'$,  and thus $u$,  are both trivial which completes the proof of Theorem \ref{main} in this case.
\end{proof}
We now extend the above to establish part (1) of Theorem \ref{main} in the more general setting where $k$ is perfect. So again $2$ is invertible in both $k$ and $R$.  It suffices  to  generalize  Lemma \ref{et0} as follows: 
\begin{lem}\label{etbis} Let $k$ be a field of characteristic different from $2$  and let $\mathcal{G}:=\mathrm{Spec}(A)$ be a finite,  etale group scheme over $k$, of odd order,  then the natural map of pointed sets 
$$u_k: H^1(k, \mathcal{G})\rightarrow H^1(k, U_{A^D})$$ is trivial. 
\end{lem}
\begin{proof} First  we note that if $\mathcal{G}$ is a constant group scheme, the result   follows from  the theorem of Bayer and  Lenstra \cite {Bayer90}  and \cite{Bayer94} Corollary 1.5.2. We now consider the general case. Since $\mathcal{G}$ is a not  constant group scheme we have to introduce a slight generalization of $G$-forms, where we adapt our previous notations from $A$ over $R$ to now $A$ over $k$. Following \cite{CCMT15} Section 2.2 we define an $A$-form on a finite and locally free $A^D$-module as a non-degenerate bilinear symmetric form $q: M\times M\rightarrow k$ such that $q(um, n)=q(m, S^D(u)n)\  \forall m, n\ \in M , u\in A^D$. For a $\mathcal{G}$-torsor $B$ we write $Tr_B$ for the trace form $Tr_{B/k}$. By \cite{CCMT15}  Proposition 5.1 and Lemma 3.1 we know that the unit form $(A^D, \kappa_{A^D})$ and   $(B, Tr_B)$  are  $A$-forms. Proposition 2.13 generalizes in this situation and $U_{A^D}$ can be identified with the group of automorphisms of the form $(A, Tr_A)$ . For any $\mathcal{G}$-torsor $B$ over $k$, there exists an isomorphism of $B$-algebras and $A^D$-modules 
$$\varphi: B\otimes_kB\simeq B\otimes A. $$ Hence, after scalar extension by $B$,  the $A$-forms $(B, Tr_B)$ and $(A, Tr_A)$ become isomorphic and so  $(B, Tr_B)$ defines an element of $H^1(k, U_{A^D})$. This class is precisely $u_k(B)$. Hence, in order to prove that the map $u_k$ is trivial,  it suffices to prove that for any $\mathcal{G}$-torsor $B$ the $A$-forms 
$(B, q_B)$ and $(A, q_A)$ are isomorphic. 

 Let $B$ be a  $\mathcal{G}$-torsor. Since $\mathcal{G}$ is a finite and  \'etale  then $B$ is a finite and \'etale $k$-algebra. Therefore there exists a  set of orthogonal idempotents  $\{\varepsilon_i, 1\leq i\leq m \}$  of $B$ with $1=\sum_{1\leq i\leq m}\varepsilon_i$ and such that $L_i=L\varepsilon_i$ is a finite separable extension of $k$ for each integer $i$.  By restriction  $\varphi$ induces  an isomorphism of $L_i$-algebras and $A^D$-modules $L_i\otimes_kB \simeq L_i\otimes_k A$ which leads us to an isomorphism of $A$-forms 
  $$L_i\otimes_k (B, Tr_B)\simeq L_i\otimes_k( A, Tr_A), \ 1\leq i\leq m.$$
 We conclude that  $u_k(B)\in \mathrm{Ker}(H^1(k, U_{A^D})\rightarrow H^1(L_i, U_{A^D}))$ for any integer $i, 1\leq i\leq m$.  Since the dimension of $B$ as a $k$-vector space  is  odd,   there exists at least one integer $i_0$ such that 
  $[L_{i_0}:k]$ is odd. By \cite{Bayer90} Theorem 2.1 we know that $H^1(k, U_{A^D})\rightarrow H^1(L_{i_0}, U_{A^D})$ is injective and so we may conclude that $u_k(B)=1$.  
\end{proof}

If $K$ is a number field then we denote  its ring of integers  by $R$  and $\Sigma$ denotes the set of its finite and infinite primes.  
\begin{cor} Let $K$ be a number field and let $\mathcal{G}:=\mathrm{Spec}(A)$ be a finite and flat group scheme over $R$. We assume that $\mathcal{G}$ is generically constant and of odd order. Then  the composition  of the morphisms 
 \[\xymatrix{
H^1(R, \mathcal{G})\ar[r]^{}&H^1(R,  U_{A^D})\ar[r]^{}
& \prod_{v\in \Sigma}H^1(R_v, U_{A_v^D}).\\
}\] is trivial. 
\end{cor}
\begin{proof} It suffices to prove that for each  prime $v\in \Sigma$ the map 
$$H^1(R_v, \mathcal{G}_v)\rightarrow H^1(R_v, U_{A_v^D})$$ is trivial. This follows from 
Theorem \ref{main} for the finite places and from \cite {Bayer90}  and \cite{Bayer94} Corollary 1.5.2 for the infinite primes.

\end{proof}

One  may  observe   that  Lemma \ref{etbis},  which holds for any   not necessarily perfect field $k$ of characteristic different from $2$, can be used   to show the triviality of the map  $u$ under various hypotheses. We give the following proposition as an example.
\begin{prop} Let $R$ be a semilocal P.I.D. with  field of fractions $K$. We assume that $2\in R^\times$ and $K$ is perfect. Let $\mathcal{G}:=\mathrm{Spec}(A)$ be a finite,  flat group scheme over $R$ of odd order.  Suppose that,  
\begin{enumerate}
\item $\mathcal{G}$ is generically \'etale. 
\item $A^D$ is a hereditary order.
\end{enumerate}
 Then the map

$$H^1(R, \mathcal{G})\rightarrow H^1(R, U_{A^D})$$ 
is trivial.
\end{prop}
\begin{proof} Let  $K$ be the field of fractions of $R$.  We consider the following commutative diagram of pointed sets as above
 \[\xymatrix{
H^1(R , \mathcal{G})\ar[r]^{u}\ar[d]_{r}&H^1(R, U_{A^D})\ar[d]_{v} \\
\textsc H^1(K, \mathcal{G}_K)\ar[r]^{u_K}&H^1(K, U_{A_K^D}). \\
}\] 
By Lemma \ref{etbis} we know that $u_K$ is trivial and so that $u_K\circ r$ is the trivial map. Moreover,  it follows from \cite{Bayer17} Theorem 5.3 that $v$ is injective. Therefore we conclude that $u$ is trivial.  
\end{proof}
\section{Grothendieck groups of local Hopf  algebras } 
Our aim is to transport the methods of Serre for group algebras,  developed in \cite{Serrerep} Part III,  to the more general setting of finite algebras and in particular Hopf algebras. 
\subsection{The Swan property}

Throughout this section $ R$ denotes  the valuation ring of a finite
extension $K$ of $\mathbb{Q}_{p}$.  Let $\mathfrak{p}=\pi R$ denote the
maximal ideal of $R$ and let $k=R/\mathfrak{p}$ be the residue field of $R$.
 We consider  a finite group $G$ and  an order $\Lambda$ in the
group algebra $K[ G]$. We let $G_{0}( \Lambda) $ denote the Grothendieck group of  
finitely generated $\Lambda $-modules and $G_{0}^{R}\left( \Lambda \right) $
denote the Grothendieck group of  finitely generated $\Lambda $-modules which
are free over $R$; henceforth we shall abusively refer to these as $\Lambda 
$-lattices. Then from page  $22$ of \cite{CR2} we know that, because $R$ is a
regular commutative ring, the natural map $G_{0}^{R}(\Lambda)
\rightarrow G_{0}( \Lambda) $ is an isomorphism. We recall also
from Ex.\ 3.1.21 in \cite{Ros} that $G_{0}( \Lambda) $ is natural with
respect to Morita equivalence and we have: 
\begin{equation}
G_{0}( M_{n}( \Lambda ) ) =G_{0}( \Lambda)
.
\end{equation}
From $38.43$ in \cite{CR2}  we know that for a homomorphism of rings $
R\rightarrow S$  the tensor product $\otimes _{R}S$ induces a homomorphism
of groups $G_{0}( \Lambda) \rightarrow G_{0}( \Lambda
\otimes _{R}S) $ via the composite
\begin{equation*}
G_{0}( \Lambda) \cong G_{0}^{R}( \Lambda) \overset{
\otimes _{R}S}{\rightarrow }G_{0}^{S}( \Lambda \otimes _{R}S)
\cong G_{0}( \Lambda \otimes _{R}S) 
\end{equation*}
since $\otimes _{R}S$ preserves exact sequences in $G_{0}^{R}( \Lambda
) $ because they are split over $R $.

We set $\widetilde{\Lambda }=\Lambda \otimes _{R}k$;  this is  a finite
dimensional $k$-algebra, and we denote by $\widetilde{\Lambda }^{ss}$ 
the maximal semisimple quotient of $\widetilde{\Lambda }$;  thus if $
J=J( \widetilde{\Lambda }) $ is the Jacobson radical of $
\widetilde{\Lambda }$,  then $\widetilde{\Lambda }^{ss}=\widetilde{\Lambda }
/J$.  From  \cite{CR2} Ex.6  p.42 the map $\widetilde{\Lambda }\rightarrow 
\widetilde{\Lambda }^{ss}$ induces an isomorphism
\begin{equation}\label{gss}
G_{0}( \widetilde{\Lambda }) \cong G_{0}( \widetilde{\Lambda 
}^{ss}) 
\end{equation}
and, for future reference, we recall that, $G_{0}( \widetilde{\Lambda }
^{ss}) $ is the free abelian group on the isomorphism classes of the
simple $\widetilde{\Lambda }^{ss}$-modules. We observe that because $
\Lambda $-modules which are $R$-free are $R$-flat, it follows that $
\otimes _{R}k$ takes exact sequences of $\Lambda $-lattices to exact
sequences of $\widetilde{\Lambda }$-modules; that is to say $\otimes _{R}k$ 
induces a homomorphism, denoted $\delta_{\Lambda} $ 
\begin{equation*}
\delta_{\Lambda}: G_{0}^{R}( \Lambda)
\rightarrow G_{0}( \widetilde{\Lambda }) .
\end{equation*}
Recall from \cite{CR2} (38.56) that we have the localization exact sequence 
\begin{equation}\label{loc}
G_{0}^{t}( \Lambda) \overset{\psi _{\Lambda }}{\rightarrow }
G_{0}^{R}( \Lambda) \overset{\varphi _{\Lambda }}{\rightarrow }
G_{0}(K[ G] ) \rightarrow 0
\end{equation}
where $G_{0}^{t}( \Lambda ) $ is the Grothendieck group of
finitely generated $\Lambda $-modules which are $R$-torsion modules  and that there exists a natural isomorphism 
\begin{equation}\label{G}
G_{0}^{t}( \Lambda) =G_{0}( \widetilde{\Lambda }) .
\end{equation}
In the case when $\Lambda =R[ G]$, from \cite{CR2} (39.10) we have Swan's
theorem which states that 
\begin{equation*}
\varphi _{G}=-\otimes _{R}K:G_{0}^{R}( R[ G]) \overset{%
\varphi _{G}}{\rightarrow }G_{0}( K[ G]) 
\end{equation*}
is an isomorphism of rings.  

\begin{definition} We shall say that the $R$-order  $\Lambda$  has the Swan property if the map $\varphi
_{\Lambda }$ in (\ref{loc}) is an isomorphism. 
\end{definition}
Recall from \cite{CR2} (39.15) : 
\begin{thm}\label{max} If $\Lambda $ is a
maximal $R$-order in $K[ G]$ which contains 
$R[ G]$, then $\Lambda$ has the Swan
property.
\end{thm}

\begin{prop}\label{se} 
Let $E_{1}, E_{2}$ be $\Lambda$-lattices in the
same $K$-vector space, with the property that $E_{1}\otimes
_{R}K=E_{2}\otimes _{R}K$. If we write $\widetilde{E}_{i}=E_{i}/
\mathfrak{p}E_{i}$,  then we have the equality in $G_{0}( 
\widetilde{\Lambda })$
$$[ \widetilde{E}_{1}] =[ \widetilde{E}_{2}] .$$
\end{prop}
\begin{proof} The proof of Theorem 32 on page 138 in \cite{Serrerep} is for $\Lambda=R[G]$ but it works perfectly well for general $R$-orders $\Lambda$ in $K[G]$.
\end{proof} 

Since $G_{0}^{t}( {\Lambda }) $ is generated by the
classes of modules $E_{1}/E_{2}$,  for such $E_{i}$,  we have shown:

\begin{cor}
The composite 

$$G_{0}^{t}(\Lambda ) \overset{\psi _{\Lambda }}{\rightarrow }G^R_{0}( \Lambda )\overset{\delta_\Lambda} {\rightarrow} G_{0}( \widetilde{\Lambda })$$
is zero.
\end{cor}
Since the composite $G_{0}^{t}( {\Lambda })\overset{\psi _{\Lambda }} {\rightarrow}
G^R_{0}( \Lambda ) \overset{\delta_\Lambda }{\rightarrow }G_{0}( 
\widetilde{\Lambda }) $ is zero, it follows from (\ref {loc}) that there is a natural ring
homomorphism, called the decomposition map, $d_{\Lambda }:G_{0}( K[
G]) \rightarrow G_{0}( \widetilde{\Lambda }) $ which
makes the following diagram commute:
\begin{equation*}
\begin{array}{ccc}
G^R_{0}( \Lambda) & \overset{\varphi _{\Lambda }}{\rightarrow } & 
G_{0}( K[ G] ) \\ 
& \delta_\Lambda \searrow & \downarrow d_{\Lambda } \\ 
&  & G_{0}( \widetilde{\Lambda }).
\end{array}
\end{equation*}
\subsection {Determinants}
In Section 2.5 we introduced the $\mathrm{Det}$ map and we recalled the equality 
$$\mathrm{Det}(K[G]^\times)=\mathrm{Hom}_{\Omega_K}(G_0(K^c[G]), K^{c\times}).$$
The counterpart of this in characteristic $p$ is again that we have two Wedderburn decompositions: 
$$\widetilde{\Lambda}^{ss}=\prod_{\chi} M_{{n_\chi}}(k_\chi), \ \ \widetilde{\Lambda}^{ss}\otimes_kk^c=\prod_{\chi'} M_{{n_\chi}}(k^c)$$ 
and the equality:  
$$\mathrm{Det}: \widetilde{\Lambda}^{ss \times}\simeq \mathrm{Hom}_{\Omega_k}(G_0(\widetilde{\Lambda}^{ss}\otimes_kk^c),  k^{c \times})$$ 
where now $\Omega_k=\mathrm{Gal}(k^c/k)$ and if $T_{\chi'}: \widetilde{\Lambda}^{ss}\rightarrow M_{n_\chi'}(k^c)$ is a representation of $\widetilde{\Lambda}^{ss}$, then for $z\in \widetilde{\Lambda}^{ss \times}$, 
$$\mathrm{Det}(z)(\chi')=\mathrm{det}(T_{\chi'}(z)).$$
\begin{prop}\label{codet} If the map $d_{\Lambda}: G_0(K^c[G])\rightarrow G_0(\widetilde \Lambda^{ss} \otimes_kk^c)$ has finite cokernel of $p$-power order,  then the canonical surjection $\beta:\Lambda^\times\rightarrow \widetilde{\Lambda}^{ss \times}$ induces a group homomorphism $\gamma: \mathrm{Det}(\Lambda^\times)
\rightarrow \mathrm{Det}(\widetilde{\Lambda}^{ss \times})$ such that the following diagram is commutative:
 \[\xymatrix{
\Lambda^\times\ar[r]^{\beta}\ar[d]_{Det}&\widetilde{\Lambda}^{ss \times}\ar[d]_{Det} \\
Det(\Lambda^{\times})\ar[r]^{\gamma}&Det(\widetilde{\Lambda}^{ss \times}). \\}\]
Moreover $\gamma$ is a surjective homomorphism whose kernel is $\mathrm{Det(} 1+J( \Lambda ))$.
\end{prop}
\begin{proof} We recall from Section 2.5 that we have a group homomorphism 
$$\mathrm{Det}:  \Lambda^\times\rightarrow \mathrm{Hom}_{\Omega_K} (G_0(K^c[G], O_K^{c\times})$$ and so a reduced map:
$$ \widetilde{\mathrm{Det}}:  \Lambda^\times\rightarrow \mathrm{Hom}_{\Omega_K} (G_0(K^c[G], k^{c\times}).$$ 
Moreover from \cite{F83} II Section 4 we know that 
$$\widetilde{\mathrm{Det}(\lambda)}(\chi)=\mathrm{Det}(\beta(\lambda))(d_{\Lambda}(\chi))\  \forall \lambda \in \Lambda^\times, \chi \in G_0(K^c[G]).$$

 Now   consider $\lambda\in\Lambda^\times$ such that $\mathrm{Det}(\lambda)=1$ and $\theta \in G_0(\widetilde{\Lambda}^{ss}\otimes k^c)$. Then there exists $\chi \in G_0(K^c[G]) $ such that $p^n\theta=d_{\Lambda}(\chi)$. Therefore we obtain 
$$\mathrm{Det}(\beta(\lambda))(\theta)^{p^n}=\mathrm{Det}(\beta(\lambda))(d_{\Lambda}(\chi))=\widetilde{\mathrm{Det}(\lambda)}(\chi)=1.$$

Since  $k$ is finite of characteristic $p$ the map $z\rightarrow z^{p^n}$ is an automorphism of $k^{c\times}$ we conclude that $\mathrm{Det}(\beta(\lambda))(\theta)=1$. Therefore we have proved that given 
$x\in \Lambda^\times$ such that $\mathrm{Det}(x)=1$, then $\mathrm{Det}(\beta(x))=1$ which indeed proves the existence of the  homomorphism $\gamma$,  which is surjective since $\mathrm{Det}\circ\beta$ is.

We now introduce a small amount of notation. 

\begin{definition}We set  
$$\ \ SL( \Lambda):=\mathrm{\ker }(\mathrm{Det}:\Lambda ^{\times}\rightarrow \mathrm{Hom}_{\Omega_K}( G_{0}( K^c[ G]), O_{K^c}^{\times })). $$
$$SL(\widetilde{ \Lambda}^{ss}):=\mathrm{\ker }(\mathrm{Det}: \widetilde{\Lambda} ^{ss \times}\rightarrow \mathrm{Hom}_{\Omega_k}( G_{0}(\widetilde{ \Lambda}^{ss}\otimes_kk^c), k^{c\times })). $$
\end{definition} 

For a field $F$ and an integer $n$ we denote by $\mathrm{GL}_n(F)$ the group of $n\times n$ invertible matrices over $F$ and by $SL_n(F)$ the group of those matrices whose determinant is equal to $1$. 
\begin{lem} \label{SLS}
Suppose that $R$ has residue characteristic different from $2$.  Then:
 \begin{enumerate}
 \item $SL_{n}( k) $ is a perfect group unless $
n=2$ and $k=\mathbb {F}_{3}$ in which case $SL_{2}( \mathbb{F}
_{3}) $ is contained in the commutator group $[
GL_{2}( \mathbb{F}_{3}) ,GL_{2}( \mathbb{F}_{3}) 
] $.Thus in all cases $ SL_{n}( k)$ is contained in $[
GL_{n}( k) ,GL_{n}( k)] $. 

\item The restriction of $\beta$ to $SL(\Lambda)$ induces a group homomorphism  $\alpha: SL( \Lambda) \rightarrow
SL( \widetilde{\Lambda }^{ss})$  which is  surjective.
\end{enumerate}
\end{lem}
\begin{proof} The proof of (1)  can be found in \cite{Lang}  Algebra XIII,  Theorem 8.3 and 9.2. In order to prove (2) 
we write $\widetilde{\Lambda }^{ss}=\oplus _{i}M_{n_{i}}( k_i)$,  so that 
$$\widetilde{\Lambda }^{ss\times }=\oplus _{i}GL_{n_{i}}( k_i)\ \mathrm{and}\ 
 SL( \widetilde{\Lambda }^{ss}) =\oplus
_{i}SL_{n_{i}}( k_i).$$

But $\Lambda ^{\times }$ maps onto $\widetilde{\Lambda }^{ss\times }$,  and
so $[ \Lambda ^{\times },\Lambda ^{\times }] \subset SL(
\Lambda) $ maps onto $[ \widetilde{\Lambda }^{ss\times },%
\widetilde{\Lambda }^{ss\times }]$ which contains $SL( \widetilde{\Lambda }
^{ss})$ as proved in (1). Therefore the homomorphism $\alpha$ is surjective. 
\end{proof}
\begin{rem}  We mention that there is
an alternative proof of the above using a theorem of Azumaya.  It follows from \cite{Azu} that $\Lambda$ contains a separable $R$-subalgebra $\Lambda'$ such that 
$\widetilde \Lambda^{' ss}\simeq \widetilde \Lambda^{ss}$. This implies that in order to prove Lemma \ref{SLS} we may assume that $\Lambda$ is separable.
In this case $\Lambda'$  is a maximal order of the form $\oplus _{i}M_{n_{i}}(R_i)$ where $R_i$ is a $p$-adic ring with residue field $k_i$. In order to prove the lemma is suffices to prove that the morphism ${SL}_{n_i}(R_i)\rightarrow {SL}(k_i)$ is surjective for every integer $i$ which is immediate. 
\end{rem}
In order to  complete the proof of the proposition we consider the commutative diagram with
exact rows:
\[\xymatrix{1\ar[r]^{}&
SL(\Lambda)\ar[r]^{}\ar[d]_{\alpha}&\Lambda^{\times}\ar[r]^{}
\ar[d]_{\beta}& \mathrm{Det}(\Lambda^{\times})\ar[r]^{}\ar[d]_{\gamma} &1\\
1\ar[r]^{}&SL(\widetilde{\Lambda}^{ss})\ar[r]^{}&\widetilde{\Lambda}^{ss}\ar[r]^{}&
\mathrm{Det}(\widetilde{\Lambda}^{ss})\ar[r]^{}&1.\\
}\]
By Lemma \ref{SLS} we know that  $\alpha$ is surjective; morover   $\mathrm{Ker}(\beta)=1+J(\Lambda)$. Therefore by the Snake lemma we conclude that  
$\mathrm{Ker}(\gamma)=\mathrm{Det}(1+J(\Lambda))$
\end{proof}
\subsection{Hopf orders} We keep the previous notations  but now in addition we assume that $\Lambda
=\Lambda _{G}$ is a Hopf order in $K[ G] $.  We start by observing
that for two $\Lambda $-modules $M, N$ the tensor product $M\otimes _{R}N$ is
a $\Lambda $-module via the comultiplication map of $\Lambda $: so for $m\in
M,\ n\in N,\ \lambda \in \Lambda $ with $\Delta ( \lambda )
=\sum \lambda _{( 1) }\otimes \lambda _{( 2) }$ we
set 
$$\lambda ( m\otimes n) =\sum \lambda _{( 1) }m\otimes
\lambda _{( 2) }n.$$
If $ M$  is $R$-free,  then $\otimes _{R}M$ preserves exact sequences of $
\Lambda $-lattices; so, using $G_{0}( \Lambda) =G_{0}^{R}(
\Lambda)$,  we see that $G_{0}( \Lambda) $ is a  ring.
Furthermore the twist isomorphism $N\otimes M\cong M\otimes N$ implies that $
G^R_{0}( \Lambda) $ is a commutative ring.

Because the map $\varphi _{\Lambda }=\otimes _{R}K:G^R_{0}( \Lambda)
\rightarrow G_{0}( K[ G] ) $ is a ring homomorphism, $
\mathrm{Im}\psi _{\Lambda }=\mathrm{Ker} \varphi _{\Lambda }$ is an ideal of $G^R_{0}(
\Lambda ) $.
\begin{thm} \label{ideal}The square of the ideal $\mathrm{Ker} \varphi _{\Lambda }$ is equal to $(0)$ in $G^R_{0}( \Lambda)$;  that is to say 
$$\mathrm{Ker} \varphi _{\Lambda }^2=(0).$$
\end{thm}
First we require\ some elementary properties of the tensor
product:
\begin{lem}\label{int}
\begin{enumerate}
\item  If $T_{i}$ are finitely generated $\widetilde{\Lambda }$-modules for $i=1,2$ then $T_{1} \otimes
_{R}T_{2}=T_{1}\otimes _{k}T_{2}$. 

\item  If $E$ is a $ \Lambda $-lattice and if $T$
 is a  finitely generated $\widetilde{\Lambda }$-module then writing  $\widetilde{E}=E/\mathfrak{p}E$
$$E\otimes _{R}T=\widetilde{E}\otimes _{R}T=\widetilde{E}\otimes _{k}T.$$

\end{enumerate}
\end{lem}
\begin{proof} For (1) see Ex. 1.12 in \cite{Ro}. For (2) we use
the exact sequence
$$\mathfrak{p}E\otimes _{R}T{\rightarrow }E\otimes _{R}T{\rightarrow }\widetilde{E}\otimes _{R}T\rightarrow 0$$
and by (1) we know $\widetilde{E}\otimes _{R}T=\widetilde{E}\otimes _{k}T$.

\end{proof}

\noindent{\bf Proof of Theorem \ref{ideal}}
\begin{proof} From (\ref{G}) we know that $G_{0}^{t}( \Lambda ) =G_{0}( \widetilde{\Lambda }) $
and  that $G_{0}( \widetilde{\Lambda })$ is generated by
terms $[ E_{1}/E_{2}] $ where $E_{1}$ and $E_{2}$ are $\Lambda $-lattices with $\pi
E_{1}\subset E_{2}\subset E_{1}$ which span the same vector space (see \cite{Se68}  Section 2.3, Corollary). It will therefore suffice to
show that for any such pairs of $\Lambda$-lattices, $( D_1, D_2)$ and $(E_1, E_2)$ one has 
$$\psi_{\Lambda}([ D_{1}/D_{2}]).\psi_{\Lambda}([ E_{1}/E_{2}]) =([ D_{1}] -[ D_{2}]) .( [ E_{1}] -[E_{2}])=0. $$
To see this, we write $F=D_{1}/D_{2}$;  then we have the exact sequence
\begin{equation}\label{D}  0\rightarrow D_{2}\rightarrow D_{1}\rightarrow F\rightarrow 0. \end{equation}
 From (\ref{D}),  applying $\otimes _{R}E_{i}$,  we have the exact sequences for $
i=1,2$
$$0\rightarrow D_{2}\otimes _{R}E_{i}\rightarrow D_{1}\otimes_{R}E_{i}\rightarrow F\otimes _{R}E_{i}\rightarrow 0.$$
But by  Lemma \ref{int}
$$F\otimes _{R}E_{i}=F\otimes _{R}\widetilde{E}_i=F\otimes _{k}\widetilde{E}_{i}.$$
So in summary we have shown
$$\psi_{\Lambda}([ D_{1}/D_{2}]).\psi_{\Lambda}([ E_{1}/E_{2}]) =[D_{1}\otimes
_{R}E_{1}]-[D_{2}\otimes _{R}E_{1}]-[D_{1}\otimes _{R}E_{2}]+
[D_{2}\otimes _{R}E_{2} ]$$
$$=F\otimes _{k}\widetilde{E}_{1}-F\otimes _{k}\widetilde{E}_{2}$$
and the latter term is zero by Proposition \ref{se}.
\end{proof}

\begin{rem} Following the argument  of \cite{CR2} Theorem (39.16) we can prove that 
$$\mathrm{Ker}(\varphi_{\Lambda})=\{x \in G_0(\Lambda)\ \mid
\ x^2=0\}.$$
Indeed we know from Theorem \ref{ideal}  that $x^2=0$ for any $x\in \mathrm{Ker}(\varphi_{\Lambda})$. If now $x^2=0$, then $\varphi_{\Lambda}(x)^2=0$ and since $G_0(K[G])$ has no non zero nilpotent elements we conclude that $x\in \mathrm{Ker}(\varphi_{\Lambda})$.
\end{rem}
For future reference we note the following important result of Jensen and
Larson \cite{JL}:
\begin{thm}\label{JL}
If $G$  is an abelian group, if $\Lambda $ is an 
$R$-Hopf order in $K[ G]$, and if $K$
is assez gros for $G$;  then $\mathrm{Ker }(\varphi _{\Lambda })=
(0)$;  that is to say $\Lambda$ has the Swan
property. 
\end{thm}
We \ have seen in Section 4.1  that group rings and maximal orders have the Swan
property. Moreover, any Hopf order $\Lambda $ which is connected has $
G_{0}^{t}( \Lambda )$ generated by the trivial simple module $
k$ (using augmentation and reduction) and $[ k] =[ R]
 -[ \pi R] =[ R] -[ R] =0$;  hence
such Hopf orders also have the {Swan property}. Presently we shall
also show that if $
K$ is assez gros for $G$, then $\Lambda _{G}$ has the Swan property whenever $G$ is an $l$-elementary group with $l\neq p$ and $l\neq 2$. 
\vskip 0.1 truecm
\noindent{\bf Conjecture.} This large number of examples leads us to ask
whether all Hopf orders have the Swan property when $R$ is local;  that is to say whether $G_0^R(\Lambda)$ is  a reduced ring for all Hopf orders $\Lambda$ in $K[G]$.
\vskip 0.1 truecm
\noindent{\bf Character action.} Our aim is to introduce various modules over $G_0(K[G])$. Since $\mathrm{Im}(\psi_{\Lambda})$ is an ideal of $G_0^R(\Lambda)$,  it has a structure of $G_0^R(\Lambda)$-module. By  Theorem \ref{ideal} we observe that $G_0^R(\Lambda)$ acts on $\mathrm{Im}(\psi_{\Lambda})$ via 
$G_0^R(\Lambda)/\mathrm{Ker}(\varphi_{\Lambda})=G_0(K[G])$.  The ring homomorphism $\delta_{\Lambda}$ induces an action of $G_0^R(\Lambda)$ on $G_0(\widetilde{\Lambda})$ and thus an action of $G_0(K[G])$, since $\delta_{\Lambda}\circ\psi_{\Lambda}=0$. Finally it follows from the very definition of this action that
$\mathrm{Im}(\delta_{\Lambda})=\mathrm{Im}(d_{\Lambda})$ is a submodule of $G_0(\widetilde{\Lambda})$ and so that $\mathrm{coker}(\delta_{\Lambda})$ is also a $G_0(K[G])$-module.

\subsection{ Change of groups and Frobenius structure} 
In this subsection $\Lambda =\Lambda _{G}$ is  again a Hopf order in $K
[ G]$. 
\begin{prop}\label{sous} If $H$ is a subgroup of $G$, then $\Lambda_H:=\Lambda \cap K[H]$ is an $R$-Hopf order in $K[H]$. Moreover, $\Lambda$ is a free $\Lambda_H$-module and 
$\Lambda_H$ is a direct summand of $\Lambda$.

\end{prop}
\begin{proof} Since $\Lambda$ is an $R$-Hopf order of $K[G]$, it follows that  $R[G]\subset \Lambda$, hence $R[H]\subset \Lambda_H$, and so $\Lambda_H$ is an $R$-order of $K[H]$. Indeed,  $\Lambda_H$ is stable under the action of the antipode, therefore in order to prove the proposition it suffices to prove that 
$$\Delta(\Lambda_H)\subset \Lambda_H\otimes \Lambda_H.$$ 
First we claim  that $\Lambda/\Lambda_H$ is a torsion free $R$-module. Let $a\in \Lambda$  such that $\bar a\in (\Lambda/\Lambda_H)_{tor}$. Then there exists $d\in R, d\neq 0$ such that 
$da\in \Lambda_H$.  Let $\{g_1,\cdots, g_q\}$ be a set of representatives of $G/H$, where $g_1$ is the unit element of $G$. One can write $a=\sum_{1\leq i\leq q}a_ig_i$, with $a_i\in K[H]$. Since $da\in \Lambda_H\subset K[H]$ we deduce that $da_i=0$ for $2\leq i$. We conclude that $a_i=0$, for $2\leq i$ and so that $a=a_1\in \Lambda\cap K[H]=\Lambda_H$ which proves the claim. 

Since $R$ is a principal ideal domain we can decompose $\Lambda$  into a direct sum of $R$-modules $\Lambda_H\oplus T$, where $\Lambda_H$ and $T$ are both finite and free. We let $\frak B_1:=\{e_1,\cdots, e_t\}$ (resp. $\frak B_2:=\{f_1, \cdots,  f_r\}$)  be an $R$-basis of $\Lambda_H$ (resp. $T$) and we  denote by $\frak B$ the $R$-basis $\frak B_1\cup \frak B_2$ of $\Lambda$. We note that $\frak B\otimes \frak B$ is a free basis  of $\Lambda\otimes \Lambda$  over $R$ and a free basis  of $K[G]\otimes K[G]$ over $K$. We decompose 
$\frak B\otimes \frak B$ into the disjoint union of the set $\frak B_1\otimes \frak B_1$, $\frak B_1\otimes \frak B_2$, $\frak B_2\otimes \frak B_1$ and $\frak B_2\otimes \frak B_2$. Let 
$x\in \Lambda_H$.  Since  $x\in K[H]$, then $\Delta(x)\in K[H]\otimes K[H]$ and so  there exist elements  $(\alpha_u)_{  u\in \frak B_1\otimes \frak B_1}$ in $K$  such that 
$$\Delta(x)=\sum_{u\in \frak B_1\otimes \frak B_1}\alpha_u u.$$ We now use the fact that $x\in \Lambda$ which is a Hopf order. Therefore $\Delta(x)\in \Lambda\otimes \Lambda$. Therefore there exist 
$(\beta_t)_{t \in \frak B_1\otimes \frak B_1}, (\gamma_v)_{v \in \frak B_1\otimes \frak B_2}, (\eta_ w)_{v \in \frak B_2\otimes \frak B_1}$ and $(\zeta_ y)_{y \in \frak B_2\otimes \frak B_2}$
elements of $R$ such that
 $$\Delta(x)= \sum_{t\in  \frak B_1\otimes \frak B_1}\beta_tt+ \sum_{v\in  \frak B_1\otimes \frak B_2}\gamma_vv+ \sum_{w\in  \frak B_2\otimes \frak B_1}\eta_ww+
  \sum_{y\in  \frak B_2\otimes \frak B_2}\zeta_yy.$$
  By comparing the equalities we conclude that $\gamma_v=0=\eta_w=\zeta_y=0, \forall v, w, y$  and that $\alpha_u=\beta_u\in R, \forall  u \in \frak B_1\otimes \frak B_1$. Since 
  $\frak B_1\otimes \frak B_1$ is a basis of $\Lambda_H\otimes \Lambda_H$ we conclude that $\Delta(\Lambda_H)\subset \Lambda_H\otimes \Lambda_H$ as required. 
  
  We now wish to prove that $\Lambda$ is a free $\Lambda_H$-module.  Since $\Lambda$ is a finite and free 
$R$-module and $\Lambda/\Lambda_H$ is a torsion free $R$-module we deduce that $\Lambda_H$ is an $R$-summand of $\Lambda$. We let $\pi$ be a uniformising  parameter of $R$. Since $\Lambda_H$ is an $R$-summand of $\Lambda$ then  $\Lambda_H\cap \pi \Lambda=\pi \Lambda_H$ and so  $\Lambda_H/\pi
\Lambda_H$   is $k$-Hopf subalgebra of $\Lambda/\pi \Lambda_G$. From the theorem of Nichols and Zoeller \cite{NZ} we deduce that $\Lambda/\pi \Lambda$ is a finite and free 
$\Lambda_H/\pi\Lambda_H$-module. Thus there exists  and isomorphism of left $\Lambda_H/\pi\Lambda_H$-modules
$f: (\Lambda_H/\pi\Lambda_H)^m\simeq \Lambda/\pi\Lambda$  for some integer $m$.  As a consequence of Nakayama's Lemma, the morphism $f$ can be lifted to a surjective homomorphism of left $\Lambda_H$-modules $\hat f: \Lambda_H^m\rightarrow \Lambda$. Since  $(\Lambda_H/\pi\Lambda_H)^m$ and $\Lambda/\pi\Lambda$ have the same dimension as $k$-vector spaces  we deduce that $\hat f$ is an isomorphism. 
  \end{proof}
\vskip 0.1 truecm
\noindent{\bf Frobenius structure.}  Given the Hopf order $\Lambda_G$ we have defined for any subgroup $H$ of $G$ a Hopf order $\Lambda_H$ of $K[H]$. Any $\Lambda _{G}$-module $M$ is a $\Lambda _{H}$-
module by restricting along the inclusion map $\Lambda _{H}\hookrightarrow
\Lambda_G$.  In the same way,  since $\Lambda _{H}\hookrightarrow
\Lambda_G$ induces an inclusion map $\widetilde{\Lambda} _{H}\hookrightarrow
\widetilde {\Lambda}_G$,  any $\widetilde {\Lambda}_G$-module is a $\widetilde {\Lambda}_H$-module. Restriction preserves both exact sequences and the ring
structure; that is to say restriction affords  ring homomorphisms of
Grothendieck rings

$$\mathrm{Res}_{\Lambda _{G}}^{\Lambda _{H}}: G_{0}^R( \Lambda_{G}) \rightarrow G_{0}^R( \Lambda _{H})\ \mathrm{and}\   \mathrm{Res}_{\widetilde{\Lambda} _{G}}^{\widetilde{\Lambda}_{H}}: G_{0}(\widetilde{\Lambda}_{G}) \rightarrow G_{0}(\widetilde{\Lambda} _{H}).$$

From Proposition \ref{sous} above we deduce that $\otimes_{\Lambda _{H}}\Lambda _{G}$ and $\otimes_{\widetilde{\Lambda} _{H}}\widetilde{\Lambda} _{G}$
preserve exact sequences and so induce  homomorphisms of abelian groups
$$\mathrm{Ind}_{\Lambda _{H}}^{\Lambda _{G}}: G_{0}^R( \Lambda_{H}) \rightarrow G_{0}^R( \Lambda _{G})\ \mathrm{and}\   \mathrm{Ind}_{\widetilde{\Lambda} _{H}}^{\widetilde{\Lambda}_{G}}: G_{0}(\widetilde{\Lambda}_{H}) \rightarrow G_{0}(\widetilde{\Lambda} _{G}).$$
If $\ M$ is a $\Lambda _{H}$-lattice and $N$ is a $\Lambda _{G}$-lattice then,
by the associativity property of the tensor product, we have the so-called
Frobenius identity 
\begin{equation*}
( \Lambda _{G}\otimes _{\Lambda _{H}}M) \otimes _{R}N\cong
\Lambda _{G}\otimes _{\Lambda _{H}}( M\otimes _{R}N) 
\end{equation*}
which we may write as
\begin{equation*}
\text{\textrm{Ind}}_{\Lambda _{H}}^{\Lambda _{G}}( M) .N\cong 
\text{\textrm{Ind}}_{\Lambda _{H}}^{\Lambda _{G}}( M.\text{\textrm{Res}}
_{\Lambda _{G}}^{\Lambda _{H}}( N)). 
\end{equation*}
Indeed,  a similar  identity holds when we replace $\Lambda_H$ (resp. $\Lambda_G$)  by $\widetilde{\Lambda}_H$ (resp. $\widetilde{\Lambda}_G$) and when $M$ and $N$ are respectively $\widetilde{\Lambda}_H$ and $\widetilde{\Lambda}_G$-modules.
Using the functorial properties of the maps $\psi_{\Lambda}$ and $\delta_{\Lambda}$   we check that these  restriction and  induction maps  allow us to define restriction and induction maps for the functors $H\rightarrow \mathrm{Im}(\psi_H)$ and $H\rightarrow \mathrm{coker}(\delta_H)$. 
Moreover,  for any such  $H$,  we have seen subsection 4.3 that $\mathrm{Im}(\psi_H)$ and $\mathrm{coker}(\delta_H)$ both have a structure of $G_0(K[H])$-module. According  to the terminology of \cite{CR2} Section 38  the functor which associates to any subgroup $H$ of $G$ the ring $G_0(K[H])$ is a  Frobenius functor,  where the restriction and the induction are the usual restriction and induction of characters.  Moreover the functors $H\rightarrow  \mathrm{Im}(\psi_H)$ and $H\rightarrow \mathrm{coker}(\delta_H)$ both are Frobenius modules for this Frobenius functor. 

\noindent{ \bf  Brauer induction.} In this subsection we assume that $K$ is assez gros for $G$. 
We now use Frobenius structure to apply Brauer induction to our study of the
representation theory of Hopf orders. 

Next, for a prime number $l$ different from $p$,  we recall the notion of an $l$-elementary group\ (see 10.1 in \cite{Serrerep}). A
subgroup $H$ of $G$ is called  $l$-elementary if we can write $H=C\times
C^{\prime }\times L'$ where $L'$ is an $l$-group,  $\ C^{^{\prime }}$ is
cyclic of order prime to $p$ and to $l$ and $C$ is cyclic of $p$-power order; so we
may write $ H=C\times L$ where the group $L$ has order
prime to $p$. We let $C( G) $ denote the set of all $l$-elementary subgroups of 
$G$ for all $l\neq p $. 

Let $\zeta$ denote a complex root of unity whose order is equal to the
exponent $e$ of $G$ and we may suppose that $e>2$ since $G$ has odd order.  We set $N=\mathbb{Q}
( \zeta) $. By Brauer's Induction Theorem (see Chapters 10 and 11 in \cite{Serrerep}), for 
some $m>0$ we can write 
\begin{equation}\label{br1}
p^{m}\varepsilon _{G}=\sum_{H\in C(G)}\mathrm{Ind}_{H}^{G}( \theta _{H}) 
\end{equation}
where $\theta _{H}\in G_{0}( N[ H]) $.  

In the  local situation that we consider  $R$ is a valuation ring of a
finite extension $K$ of $\mathbb{Q}_{p}$ and  $K$ is assez gros for $G$;  
 we may then choose field embeddings 
\begin{equation*}
h:N=\mathbb{Q}( \zeta) \hookrightarrow \mathbb{Q}_{p}(
\zeta) \hookrightarrow K. 
\end{equation*}
This induces an isomorphism, which henceforth we shall regard as an
identification, 
\begin{equation}
G_{0}( N[ G]) =G_{0}( K[ G] ).
\end{equation}
Therefore for any $H\in C(G)$ we can consider $\theta_H$ as an element of $G_0(K[H])$.
Using now the Frobenius module structure of $H\rightarrow \mathrm{Im}(\psi_{\Lambda_H})$  over the Frobenius functor $H\rightarrow G_0(K[H])$ and the Frobenius identity, we know that for $x \in \mathrm{Im}\psi _{\Lambda _{G}}$ we  similarly have 
\begin{equation}
p^{m}\varepsilon _{G}.x  =\sum_{H\in C(
G)}\mathrm{Ind}_{H}^{G}( \theta _{H}) .x
\end{equation}
and so 
\begin{equation} 
p^{m}x  =\sum_{H\in C(G) }\mathrm{
Ind}_{\Lambda _{H}}^{\Lambda _{G}}(\theta _{H}.\mathrm{Res}
_{\Lambda _{G}}^{\Lambda _{H}}x).
\end{equation}
Using now the Frobenius module structure of $H\rightarrow \mathrm{Im}(\delta_{\Lambda_H})$, then for any $y\in  \mathrm{Im}(\delta_{\Lambda_H})$ we have 
\begin{equation}
p^{m}y  =\sum_{H\in C(G) }\mathrm{
Ind}_{\widetilde{\Lambda} _{H}}^{\widetilde{\Lambda} _{G}}(\theta _{H}.\mathrm{Res}
_{\widetilde{\Lambda} _{G}}^{\widetilde{\Lambda} _{H}}y).
\end{equation}
Therefore Proposition \ref{inter} follows from equalities (41) and (42). 

\begin{prop} \label{inter} Let $C( G) $ denote the set of all $l$-elementary subgroups of $G$ for all $l\neq p $. There exists an integer $m>0$ such that :
\begin{enumerate} 
\item If $\mathrm{Im}(\psi _{\Lambda _{H}})=0$ for all $H\in C(G)$,  then we have 
$$p^{m}.\mathrm{Im}(\psi _{\Lambda _{G}})=0.$$
\item If $\delta_{\Lambda_H}$ is surjective  for all $H\in C(G)$,  then
 we have 
$$p^{m}.\mathrm{coker}(\delta_{\Lambda _{G}})=p^{m}.\mathrm{coker}(d_{\Lambda _{G}})=0.$$
\end{enumerate}
\end{prop}

\vskip 0.1 truecm
\noindent{\bf More on $l$-elementary groups.}
 We let $H$ be an $l$-elementary group with $l\neq p$. The group $H$ decomposes into a direct product $C\times L$ where $C$ is a cyclic $p$-group of order $p^n$ and $L$ a finite group of order $m$, coprime with $p$. We consider  a Hopf $R$-order $\Lambda_H$ of $K[H]$ and we define $\Lambda_C=\Lambda_H\cap K[C]$ and $\Lambda_L=\Lambda_H\cap K[L]$. By Proposition \ref{sous} we know that $\Lambda_C$ and $\Lambda_L$ are Hopf $R$-orders respectively of $K[C]$ and  $K[L]$ which implies,  since 
 $m$ is coprime with $p$,  that $\Lambda_L=R[L]$. The order $\Lambda_H$ can be described as follows: 
 \begin{prop}\label{clo} The Hopf $R$-orders $\Lambda_H$ and $\Lambda_C[L]$ are equal. 
\end{prop}
\begin{proof}  It follows from the definitions that $\Lambda_C[L]$ and $\Lambda_H$ are both Hopf $R$-orders of $K[H]$ such that $\Lambda_C[L]\subset\Lambda_H$. For any Hopf $R$-order $\Lambda$ of $K[H]$ we let  
$I(\Lambda)$ be the ideal of left integrals of $\Lambda$ (see Section 2.1). From \cite{L72} Corollary 5.2 we know that the equality of the orders $\Lambda_C[L]=\Lambda_H$ is equivalent to the equality of the $R$-ideals 
\begin{equation}\label{I_0}\varepsilon (I(\Lambda_C[L]))=\varepsilon (I(\Lambda_H)).\end{equation} Our aim is to prove this equality. 

For  a finite group $G$ of order $r$  we set $\sigma_G=\sum_{g\in G}g$. For any Hopf $R$-order $\Lambda$ of $K[G]$ one easily checks the following: 
\begin{equation}\label{I_1}I(K[G])=K\frac{\sigma_G}{r}, \ \ I(\Lambda)=\Lambda\cap I(K[G])= I_\Lambda\frac{\sigma_G}{r} \end{equation} where
$$I_\Lambda=\{\lambda\in R\\ \ \ \  \mathrm{with}\ \lambda \frac{\sigma_G}{r}\in \Lambda\}.$$
We conclude that $\varepsilon (I(\Lambda))=I_{\Lambda}$. 

Since $H=C\times L$, then $\sigma_H=\sigma_C.\sigma_L$ and thus  we have 
\begin{equation}\label{I_2}I(\Lambda_H)=I_{\Lambda_H} (\frac{\sigma_C}{p^n}. \frac{\sigma_L}{m}), \ I(\Lambda_C)=I_{\Lambda_C}\frac{\sigma_C}{p^n},  \ I(\Lambda_C[L])=I_{{\Lambda_C[L]}} (\frac{\sigma_C}{p^n}. \frac{\sigma_L}{m}).\end{equation}
For the sake of simplicity, for any subgroup $T$ of $G$ we write $I_T$ for $I_{\Lambda_T}$. Since $m$ is a unit of $R$ it is easy to check that $I_{\Lambda_C[L]}=I_C$. Therefore it follows from (\ref{I_1}) and (\ref{I_2}) that in order to prove (\ref{I_0}) it suffices to prove that 
$$I_H=\varepsilon (I(\Lambda_H))=I_C=\varepsilon (I(\Lambda_{C}[L])).$$ 

Indeed,  for $x\in I_C$ then $x\frac{\sigma_C}{p^n} \in \Lambda_C$ and,  since $m$ is a unit in $R$, we deduce that  $x\frac{\sigma_C}{p^n}\frac{\sigma_L}{m} \in \Lambda_C[L]\subset \Lambda_H$
 and so that $x \in I_H$.   Therefore we have proved that $I_C\subset I_H$.

In order to prove that $I_H\subset  I_C$ we start by observing that since $m$ is a unit of $R$, then $\frac{\sigma_L}{m}\in \Lambda_H$. We let $\{e_1,\cdots, e_m\}$ be a basis of $\Lambda_H$ over $\Lambda_C$. Thus there exists $\{x_1, \cdots, x_m\}$ of $\Lambda_C$ such that $\frac{\sigma_L}{m}=\sum_ix_ie_i$. Therefore $1=\sum_i\varepsilon(x_i)\varepsilon(e_i)$. Since for any $i$ we know that $\varepsilon(x_i)$ and $\varepsilon(e_i)$ belong to $R$, the equality implies that there exists at least $i_0$ such that $\varepsilon (x_{i_0})$ is a unit of $R$.  Let $x\in I_H$ then $x (\frac{\sigma_C}{p^n}. \frac{\sigma_L}{m}) \in \Lambda_H$. We set $y=x (\frac{\sigma_C}{p^n})$. Thus $y\in K[C]$ and we know that:
$$y\frac{\sigma_L}{m}=\sum_i(yx_i)e_i \in \Lambda_H.$$
Therefore there exist $\{y_1,; \cdots, y_m \in \Lambda_C\}$ such that 
\begin{equation}\label{I_3} \sum_i(yx_i)e_i =\sum_iy_ie_i.\end{equation}  Since $yx_i\in K[C]$  there exists non zero $d\in R$ such that  $dyx_i\in \Lambda_C$ for all $i$. It follows from (\ref{I_3}) that 
$dyx_i=dy_i$  and so that $yx_i=y_i$ for all $i$. We conclude that $yx_i\in \Lambda_C, \forall i$. We now have 
$$yx_i=x\frac{\sigma_C}{p^n}x_i=xx_i\frac{\sigma_C}{p^n}=x\varepsilon(x_i)\frac{\sigma_C}{p^n}\in \ \Lambda_C, \  \forall i.$$
In particular for $i=i_0$, then $x\varepsilon(x_{i_0})\frac{\sigma_C}{p^n}\in  \Lambda_C$ and so, since $\varepsilon (x_{i_0})$ is a unit of $R$, we deduce  that $x\frac{\sigma_C}{p^n}\in \Lambda_C$ and therefore that 
$x\in I_C$. Therefore we have proved that $I_H\subset I_C$. This completes the proof of the proposition.

\end{proof}
  
   From the proposition above  we have an isomorphism of Hopf $R$-algebras 
   \begin{equation}\label{tens}\Lambda_H=\Lambda_C[L]\simeq \Lambda_C\otimes_R\Lambda_L.\end{equation}
   
   By tensoring with $K$ we obtain the equalities: 
 $$\Lambda_H\otimes_RK=K[H],\  \Lambda_C\otimes_RK=K[C],\  \Lambda_L\otimes_RK=K[L]$$ and the isomorphism of Hopf $K$-algebras
 $$K[H]\simeq K[C]\otimes_KK[L].$$

 It follows from \cite{Serrerep} Theorem 10, that there exists an isomorphism of abelian groups:
 $$G_0(K[H])\simeq G_0(K[C]\otimes_{\mathbb Z}G_0(K[L]).$$
 We now set $$\widetilde{\Lambda}_H=\Lambda_H\otimes_Rk,\ \widetilde{\Lambda}_C=\Lambda_C\otimes_Rk, \ \widetilde{\Lambda}_L=\Lambda_L\otimes_Rk.$$
 By tensoring  (\ref{tens}) with $k$ we now obtain \begin{equation}\widetilde\Lambda_H=\widetilde\Lambda_C\otimes_k\widetilde \Lambda_{L}.\end{equation}
 We now assume that $K$ is assez gros for $H$.   
  \begin{thm}\label{el2} 
  For an elementary $l$-group $H$, with $l\neq p$,  there is an isomorphism of Grothendieck groups:
   \begin{equation}
 G_{0}(\widetilde{\Lambda}_{H}) \simeq G_{0}(\widetilde{\Lambda }_{C}) \otimes _{\mathbb Z}G_{0}( \widetilde{\Lambda }_L
).
\end{equation}
 Moreover,  we have:
\begin{enumerate}
\item The morphism $d_{\Lambda_H}$ is surjective.
\item $\mathrm{Im}(\psi _{\Lambda _{H}})=0$.
\end{enumerate}
\end{thm}
\begin{cor} Let $H$  be an  $l$-elementary group with $l\neq p$. Then, for any Hopf $R$-order of $K[H]$, 
the change of ground ring map induces an isomorphism of rings
$$\varphi_{\Lambda_H}: G_0^R(\Lambda_H)\simeq G_0(K[H]).$$
\begin{proof} The proof follows immediately from the theorem above and the localization exact sequence (\ref{loc}).
\end{proof}

\end{cor} 
\begin{proof} [Proof of Theorem \ref{el2}]
 We start by recalling that for any order $\Lambda$ then $\widetilde{\Lambda}\rightarrow \widetilde{\Lambda}^{ss}$ induces an isomorphism 
$G_0(\widetilde{\Lambda})\simeq G_0(\widetilde{\Lambda}^{ss})$ (see (\ref{gss})). We consider the   isomorphism of $k$-algebras
$$ \widetilde{\Lambda }_{C}\otimes _{k}\widetilde{\Lambda }_{L}\cong \widetilde{\Lambda }_{H}. $$

Since $k$ is a perfect field and since the algebras are finite over $k$ we know that  
$$J( \widetilde{\Lambda }_{C}\otimes _{k}\widetilde{\Lambda }_{L}) =J( \widetilde{\Lambda }_{C}) \otimes _{k}
\widetilde{\Lambda }_{L}+\widetilde{\Lambda }_{C}\otimes
_{k}J(\widetilde{\Lambda }_{L}). $$ Therefore  we have an  isomorphism  of $k$-algebras
\begin{equation}\label{dec1}
 \widetilde{\Lambda }_{C}^{ss}\otimes _{k}\widetilde{\Lambda }_{L}^{ss}=\widetilde{\Lambda }_{C}^{ss}\otimes _{k}\widetilde{\Lambda }_{L}\cong \widetilde{\Lambda }_{H}^{ss}.
\end{equation}
 
 Since $K$ is assez gros for $H$, we start by proving that  $\widetilde{\Lambda }_L$ and $\widetilde{\Lambda}_C^{ss}$ are split semisimple $k$-algebras.  Since $K$ is assez gros for $H$, there exists an isomorphism of $K$ algebras 
$$K[L]\simeq \prod_\beta M_{n_\beta}(K).$$
The order of $L$ is coprime to $p$ and thus $\Lambda_L=R[L]$ is a maximal order of $K[L]$. From \cite{R75} chapter 5   we deduce that there exists an isomorphism of $R$-algebras 
$$\Lambda_L\simeq  \prod_\beta M_{n_\beta}(R).$$
Therefore there exists an isomorphism of $k$-algebras 
\begin{equation}\label{pr1}\widetilde{\Lambda}_L\simeq \prod_\beta M_{n_\beta}(k)
\end{equation}
and so $\widetilde{\Lambda}_L$ is split. 

We now consider $\widetilde{\Lambda}_C^{ss}$. Since $\Lambda_C$ is commutative we can consider the affine group scheme  $\mathcal{G}=\mathrm{Spec}(\Lambda_C)$  and the connected  \'etale sequence attached to this group. We let $\Lambda_C^{et}$ be the \'etale subalgebra of $\Lambda_C$ of the global sections of 
$\mathcal{G}^{et}$. We have an isomorphism of $k$-algebras (see (\ref{utu}))
$$\widetilde{\Lambda}_C^{ss}\simeq \widetilde{\Lambda_C^{et}}.$$
Since $\Lambda_C^{et}\otimes_RK$ is a Hopf subalgebra of $K[C]$,  it is a group algebra $K[D]$ where $D$ is a subgroup of $C$. Therefore, since $K$ is assez gros for $H$,  it is assez gros for $C$ and $D$ and so there exists an isomorphism of $K$-algebras 
$$K[D]\simeq \prod_\alpha K. $$
Since $\Lambda_C^{et}$ is an \'etale Hopf $R$-order of $K[D]$ it is the maximal order and so we have an isomorphism of $R$-algebras 
$$\Lambda_C^{et}\simeq \prod_\alpha R$$ and  an isomorphism of $k$-algebras
\begin{equation}\label{pr2} \widetilde{\Lambda}_C^{ss}\simeq \prod_\alpha k.\end{equation}
In what follows we shall use the following Lemma:
\begin{lem}(see \cite{Ros})
Let $R$ and $S$ be commutative algebras on a field $k$. Then the following properties hold: 
\begin{enumerate}
\item $G_0(\mathrm{M}_n(R)\simeq G_0(R).$

\item $G_0(R\times S)\simeq G_0(R)\oplus G_0(S).$

\item $G_0(k)\simeq \mathbb Z.$

\end{enumerate}
\end{lem}
From  (\ref{pr1})  and (\ref {pr2}) we obtain  isomorphisms of $k$-algebras
\begin{equation}\label{pr3}
\widetilde{\Lambda}_H^{ss}\simeq (\prod_{\alpha}k)\otimes_k(\prod_{\beta}\mathrm{M}_{n_\beta}(k))\simeq \prod_{\alpha, \beta}\mathrm{M}_{n_\beta}(k).
\end{equation}
Using the lemma above we deduce from (\ref{pr3})  isomorphism of $\mathbb{Z}$-modules
$$G_0(\widetilde{\Lambda}_H)\simeq \oplus_{\alpha, \beta}G_0(\mathrm{M}_{n_\beta}(k))\simeq \oplus_{\alpha, \beta}G_0(k)\simeq \oplus_{\alpha, \beta}\mathbb Z\simeq G_{0}(\widetilde{\Lambda }_{C}) \otimes _{\mathbb Z}G_{0}( \widetilde{\Lambda }_L
).
$$
The first part of the theorem now follows.

 From now on,  for the sake of simplicity,  we write $\psi_H$ and $d_H$ for $\psi_{\Lambda_H}$ and  $d_{\Lambda_H}$  and $\psi_C$ and  $d_C$ for $\psi_{\Lambda_C}$ and $d_{\Lambda_C}$. 

 We start by proving that $d_C$ is surjective.  We set  $\mathcal{G}=\mathrm{Spec}(\widetilde \Lambda_C)$.  Then $\mathcal{G}$ is a  finite and commutative group scheme over $k$.  We consider the connected-\'etale sequence of $\mathcal{G}$ over $k$
$$0\rightarrow \mathcal{G}^0\rightarrow \mathcal{G}\rightarrow \mathcal{G}^{et}\rightarrow 0.$$
Since the  field  $k$ is  perfect, this sequence is split and $\mathcal{G}$ is the direct product   
\begin{equation}\label{C}\mathcal{G}=\mathcal{G}^0\times\mathcal{G}^{et}.\end{equation} Since $\mathcal{G}$ is a finite group, the groups $\mathcal{G}^0$ and $\mathcal{G}^{et}$ are finite, then  there exist  finite commutative Hopf $k$-algebras $\widetilde \Lambda_C^0$ and $\widetilde \Lambda_C^{et}$ such that $\mathcal{G}^0=\mathrm{Spec}(\widetilde \Lambda_C^0)$ and $\mathcal{G}^{et}=\mathrm{Spec}(\widetilde \Lambda_C^{et})$ and the equality (\ref{C}) can be translated into an equality of Hopf algebras
\begin{equation}\label{CC} \widetilde \Lambda_C=\widetilde \Lambda_C^0\otimes_k\widetilde \Lambda_C^{et}.
\end{equation} 
Since the algebras above are finite over a perfect field, we deduce from (\ref{CC}):
$$ \widetilde \Lambda_C^{ss}=\widetilde \Lambda_C^{0,ss}\otimes_k\widetilde \Lambda_C^{et,ss}$$
and so that 
\begin{equation}\label{decet}\widetilde \Lambda_C^{ss}=k\otimes_k\widetilde \Lambda_C^{et,ss}\simeq \widetilde \Lambda_C^{et}.\end{equation}
More precisely, let $i_C: \widetilde \Lambda_C^{et}\hookrightarrow \widetilde \Lambda_C$ be the canonical injection and $\alpha_C$ be the canonical surjection $\widetilde \Lambda_C\rightarrow \widetilde \Lambda_C^{ss}$, then 
$\beta_C:=\alpha_C\circ i_C$ is an isomorphism of $k$-algebras. Therefore $\beta_C$ induces a group isomorphism 
$\hat\beta_C: G_0(\widetilde \Lambda_C^{ss})\rightarrow G_0(\widetilde \Lambda_C^{et})$ which can be decomposed into $\hat\beta_C=\hat i_C\circ\hat \alpha_C$  with 
$\hat i_C: G_0(\widetilde \Lambda_C)\rightarrow G_0(\widetilde \Lambda_C^{et})$   and $\hat \alpha_C: G_0(\widetilde \Lambda_C^{ss})\rightarrow G_0(\widetilde \Lambda_C)$. Since $\hat \alpha_C$ is a group isomorphism we deduce that the restriction homomorphism $\hat i_C$ is also a group isomorphism.

We recall that  since $R$ is Henselian,  then
$\widetilde{ \Lambda}_{C}^{et}=\widetilde {\Lambda_{C}^{et}}$ and therefore,  as claimed before,  we have isomorphisms of $k$-algebras 
\begin{equation}\label{utu} \widetilde \Lambda_C^{ss}\simeq \widetilde \Lambda_C^{et}\simeq \widetilde {\Lambda_{C}^{et}}.\end{equation}
Moreover we know that  $\Lambda_C^{et}\otimes_RK =K[D]$ where $D$ is a subgroup of $C$. 

We consider  the decomposition homomorphism   $d_C:G_0(K[C])\rightarrow G_0(\widetilde \Lambda_C)$. Since $\Lambda_C^{et}\otimes_RK =K[D]$ we also have a decomposition homomorphism 
$d_D: G_0(K[D])\rightarrow G_0(\widetilde{\Lambda}_C^{et})$. We let $\hat{j_C}: G_0(K[C])\rightarrow G_0(K[D])$ be the homomorphism of groups induced by restriction. 
The following diagram is commutative 
 \[\xymatrix{
G_0(K[C])\ar[r]^{d_C}\ar[d]_{\hat j_C}&G_0(\widetilde{\Lambda_C})
\ar[d]_{\hat i_C} \\
\textsc G_0(K[D]))\ar[r]^{d_D}&G_0(\widetilde{\Lambda_C}^{et}) \\
}\] 
One knows that $\hat j_C$ is surjective. Moreover since $\widetilde{\Lambda_C}^{et}$ is a semisimple algebra over $k$, then the decomposition map $d_D$ is a group isomorphism and so 
$d_D\circ \hat j_C=\hat i_C\circ d_C$ is surjective. Therefore, since $\hat i_C$ is a group isomorphism,  we deduce that $d_C$ is surjective. 
 
We now want to prove that $d_H$ is surjective. Since $d_C$ is surjective, we obtain a set of generators of $G_0(\widetilde{\Lambda_C})$ by considering the classes $\{[\frac{M_1}{\pi M_1}], \cdots, [\frac{M_r}{\pi M_r}]\}$ where $\{M_1,\cdots, M_r\} $ are $\Lambda_C$-lattices such that $\{[M_1\otimes K], \cdots, [M_t\otimes K]\}$ is a $\mathbb Z$-basis of $G_0(K[C])$.  Since the order of $L$ is coprime to  $p$,  it follows from \cite{Serrerep} Proposition 4.3,  that   a $\mathbb Z$-basis of  $G_0(\widetilde\Lambda_L)$ is given by $\{ [\frac{N_1}{\pi N_1}], \cdots, [\frac{N_t}{\pi N_t}]\}$, where $\{N_1,\cdots, N_t\} $ are $\Lambda_L$-lattices such that $\{[N_1\otimes K], \cdots, [N_t\otimes K]\}$ is a $\mathbb Z$-basis of $G_0(K[L])$. 
 We conclude that  \begin{equation}\label{S} S=\{[\frac{M_i}{\pi M_i}\otimes_k\frac{N_j}{\pi N_j}],\ 1\leq i\leq r,\ 1\leq j\leq t\}\end{equation}
is a set of generators of $G_0(\widetilde{\Lambda}_H)$. We set $U_i=M_i\otimes_RK$ and $V_j=N_j\otimes_RK$. Then $M_i\otimes_RN_j$ is a $\Lambda_H$ lattice of $U_i\otimes_R  V_i$ and so 
\begin{equation}d_H( [U_i\otimes_K V_j])=[\frac{M_i\otimes_RN_j}{\pi (M_i\otimes_RN_j)}]. \end{equation}
\begin{lem}\label{ut} Let $ L_1$, $L_2$, $L'_1$ and $L'_2$ be finite and free $R$-modules with $L'_1\subset L_1$ and $L'_2\subset L_2$. We assume that $\pi L_1\subset L'_1$ and $\pi L_2\subset L'_2$. Then the morphism $\theta$ of $R$-modules 
$L_1\otimes_RL_2\rightarrow \frac{L_1}{L'_1}\otimes_k \frac{L_2}{L'_2}$ induces an isomorphism of $k$-vector spaces
$$\frac{L_1\otimes_RL_2}{L'_1\otimes_R L_2+L_1\otimes_RL'_2}\simeq \frac{L_1}{L'_1}\otimes_k \frac{L_2}{L'_2}.$$
Moreover,  if $L_1$ and $L'_1$ (resp. $L_2$ and $L'_2$ ) are $\Lambda_C$ (resp. $\Lambda_L$)-lattices, then $\theta$ is an isomorphism of $\widetilde\Lambda_C\otimes_k\widetilde\Lambda_L$-modules.
\end{lem}
\begin{proof}  Indeed $\theta$ is surjective; therefore  it suffices to prove that it is injective.  We know that $L_1$ and $ L'_1$ are both $R$-free modules. Since $\pi L_1\subset L'_1 $ the classical structure theory of finite modules over a principal ring implies that there exists a basis $\{e_1,\cdots, e_n\}$  of $L_1$ such that $\{e_1, \cdots e_r, \pi e_{r+1}, \cdots, \pi e_n\}$ is a basis of $L'_1$. We easily check that $\{\bar e_{r+1},\cdots, \bar e_n\}$ is a basis of $\frac{L_1}{L'_1}$ as a $k$-vector space. Thus every element $y\in \frac{L_1}{L'_1}\otimes_k \frac{L_2}{L'_2}$ can be written in a unique way $y=\bar e_{r+1}\otimes x_{r+1}+\cdots  +\bar e_n\otimes x_n$ with 
$x_{r+1}, \cdots, x_n \in  \frac{L_2}{L'_2}$. Let $x$ be an element of $L_1\otimes_RL_2$. There exist $b_1, \cdots, b_n \in L_2$ such that 
$x=\sum_{1\leq i\leq n}e_i\otimes b_i$. Therefore
$$\theta(x)=\sum_{1\leq i\leq n}\bar e_i\otimes \bar b_i=\sum_{r+1\leq i\leq n}\bar e_i\otimes \bar b_i .$$
We conclude that $x\in \mathrm{Ker}(\theta)$ if and only if $b_i\in L'_2, \ r +1\leq i\leq n$ and so 
$$x=\sum_{1\leq i\leq r}e_i\otimes b_i+\sum_{i\leq r+1\leq n}e_i\otimes b_j \in L'_1\otimes_R L_2+L_1\otimes_RL'_2.$$ Therefore $\theta $ induces  the required  isomorphism of $k$-vector spaces,  still denoted by $\theta$.  
When $L_1$ and $L'_1$ (resp. $L_2$ and $L'_2)$ are $\Lambda_C$ (resp. $\Lambda_L$) lattices, then $\Lambda_C\otimes_R\Lambda_L$ acts on $L_1\otimes_RL_2$ via the action of $\Lambda_C$ on $L_1$ and the action on $\Lambda_L$ of $L_2$ and $\theta$ commutes with the action of $\widetilde\Lambda_C\otimes_k\widetilde\Lambda_L$ induced on $\frac{L_1\otimes_RL_2}{L'_1\otimes_R L_2+L_1\otimes_RL'_2}$ and $\frac{L_1}{L'_1}\otimes_k \frac{L_2}{L'_2}$. 
\end{proof}
Choosing  $L_1=M_i$, $L_2=N_j$, $L'_1=\pi M_i$ and $L'_2=\pi N_j$ we note the equality 
$$L_1\otimes_RL'_2+L'_1\otimes_RL_2=M_i\otimes_R\pi N_j+\pi M_i\otimes_R N_j=\pi(M_i\otimes_RN_j)$$ and we deduce from  Lemma \ref{ut} the isomorphism of $\widetilde{\Lambda}_H$-modules: 

$$\frac{M_i\otimes_RN_j}{\pi (M_i\otimes_RN_j)}\simeq \frac{M_i}{\pi M_i}\otimes_k\frac{N_j}{\pi N_j}.$$
We conclude that each element of $S$ in (\ref{S}) belongs to $\mathrm{Im}(d_H)$ and so that $d_H$ is surjective.

In order to complete the proof of Theorem \ref{el2} it suffices to prove that (1) implies (2) which is what we do in the next lemma.
\begin{lem}
Suppose that the decomposition homomorphism $d_H$ is surjective then $\psi_H$ is the trivial map.
\end{lem}
\begin{proof} We consider the group homorphism 
$$\psi_H\circ d_H: G_0(K[H])\rightarrow G_0(\Lambda_H).$$
Let $V$ be  a $K[H]$-module and let $M$ be  a $\Lambda_H$-lattice of $V$. We then have $d_H([V])=[\frac{M}{\pi M}]$. From the exact sequence of $\Lambda_H$-modules 
$$0\rightarrow \pi M\rightarrow M\rightarrow \frac{M}{\pi M}\rightarrow 0$$ we deduce that $\psi_H( [\frac{M}{\pi M}])=[M]-[\pi M]=0$, since $M$ and $\pi M$ are isomorphic $\Lambda_H$-lattices.  Therefore we have proved that 
$\psi_H\circ d_H([V])=0$ for every $K[H]$-module $V$. Since $d_H$ is surjective,  we conclude that $\psi_H$ is the trivial map. 
\end{proof}

\end{proof}

As a consequence of Proposition \ref{inter} and Theorem \ref{el2} we have now shown: 
\begin{thm}\label{majeur} If $G$  is a finite group, if $\Lambda_G $ is an 
$R$-Hopf order in $K[ G]$ , and if $K$
is assez gros for $G$;  then,  for some $m>0$ we have: 
\begin{enumerate}
 \item $p^{m}.\mathrm{Im}(\psi _{\Lambda
_{G}})=0$
\item $p^{m}\mathrm{coker}(\delta_{\Lambda_G})=p^m\mathrm{coker}(d_{\Lambda_G})=0$.
\end{enumerate}
\end{thm}

\subsection {Duality.} 

 Here we first consider duality for representations of a finite group $G$ over a
field $K$ of characteristic zero (see 13.2 in \cite{Serrerep}).

Let $V$ denote a $K[G] $-representation; that is to say, $ V$
is a left $K[ G]$-module of finite $K$-dimension. The dual of $V$, denoted $V^{D}$,  is the left $K[G]$-representation 
$\mathrm{Hom}_{K}(V, K) $ where for $g\in G,\ f:V\rightarrow K,\ \
^{g}f( v) =f( g^{-1}v)$.  Thus, if $V$ has
character $\chi$,  then $V^{D}$ has character $\overline{\chi }$ where $
\overline{\chi }( g) =\chi( g^{-1})$. 
\begin{definition}
 We say that $V$ is a self-dual representation of $G$ if $V\cong V^{D}$ as left $K[
G]$-modules. We write $G_{0}^{+}( K[ G]
)$ for the subgroup of $G_{0}(K[ G]) $ of characters of self-dual representations. We note that,
if $K\subset \mathbb{R}$, then every $K[ G]$-representation is self-dual.
\end{definition}
\begin{example} For any $K[ G] $-representation $V$ the
representation $V\oplus V^{D}$ is self-dual: indeed, identifying $V=(
V^{D}) ^{D}$ we have
\begin{equation}
( V\oplus V^{D})^{D}=V^{D}\oplus V^{DD}=V^{D}\oplus V\cong
V\oplus V^{D}.
\end{equation}
In particular observe that $(V\oplus V^{D})$ supports the $G$-invariant non-degenerate symmetric form $q$  
\begin{equation}
q( v\oplus f,v^{\prime }\oplus f^{\prime }) =f( v^{\prime
}) +f^{\prime }( v) 
\end{equation}
and it also supports the $G$-invariant non-degenerate alternating form $a$  
\begin{equation}
a( v\oplus f,v^{\prime }\oplus f^{\prime }) =f( v^{\prime
}) -f^{\prime }( v) .
\end{equation}
\end{example}

\noindent{\bf Brauer induction of self-dual representations.} We now assume that $K$ is assez gros for $G$. 
We keep the notations of subsection 4.4. We recall that  $\zeta$ denotes a complex root of unity whose order is equal to the
exponent $e$ of $G$,  with $e>2$,  and  that  $N=\mathbb{Q}
( \zeta)$. We let  $N^{+}=\mathbb{Q}(\zeta +\zeta ^{-1}) $ and 
we put  $\Gamma =\Gamma _{N^{+}}=\mathrm{Gal}( N/N^{+}) $ so that 
$\Gamma=\left\langle \gamma \right\rangle$ where $\gamma$ is complex
conjugation. The  action of $\Gamma$ on the $e$-th roots of unity induces a group isomorphism  $\Gamma \simeq \pm 1$; we identify $\Gamma$ and its image in  
$\mathbb Z/2\mathbb Z$. From the above (see also \cite{Serrerep} Section 13.2) we know that an
irreducible character $\chi \in G_{0}( N[G]
) $ is self-dual iff $\chi $ is real valued and  $\chi $
is the character of an orthogonal representation of $G$ iff $\chi \in
G_{0}( N^{+}[ G]) $. 

Recall from  12.6 of \cite{Serrerep} that a subgroup $H$ of $G$ is called $\Gamma$-$l$-elementary if we can write $\ H=C\rtimes L$ where $C$ is cyclic of order
prime to $l$ and $L$ is an $l$-group such that, for any $\lambda \in L$, there exists a unique $t \in \Gamma$ such that 
$$\lambda x\lambda^{-1}=x^t, \forall x\in C.$$
We note that since we  take $H$ to be a subgroup of the group $G$ of
odd order and since $\Gamma $ has order $2$,  in our situation,  a subgroup $\Gamma$-$l$-elementary is  always  $l$
-elementary; that is to say $H$ is the direct product of $C$ and $L$. 

By Brauer's Theorem, as described above, for 
some $m>0$ we can write 
\begin{equation}\label{br1}
p^{m}\varepsilon _{G}=\sum_{H\in C(G) }\mathrm{Ind}_{H}^{G}( \theta _{H}) 
\end{equation}
for virtual characters $\theta _{H}\in G_{0}( N^{+}[H])=G_0^+(N^+[H])$ (see \cite{Serrerep} Theorem 28). 

We fix an  embedding 
\begin{equation*}
t: N=\mathbb{Q}( \zeta) \hookrightarrow \mathbb{Q}_{p}(
\zeta) \hookrightarrow K
\end{equation*}
and we identify 

\begin{equation}
G_{0}( N[ G]) =G_{0}( K[ G] ).
\end{equation}
 Under this identification $G^+_{0}( N[ G]) $
identifies as the group of virtual self-dual characters $G_{0}^{+}( K
[ G])$ of $G$  and so we may view $\varepsilon _{G}$
(resp. $\theta _{H}$)  in (\ref{br1}) above as self-dual characters of $G$ (resp. $
H$). 

\medskip

\noindent{\bf Duality for Grothendieck groups of Hopf orders.}
\vskip 0.1truecm
\noindent {\bf Duality for $G_{0}( \Lambda)$.} Here, as previously,
we assume that $R$ is the valuation ring of a finite extension $K$ of $
\mathbb{Q}_{p}$, that the group $G$ has odd order  and that $K$ is assez gros for $G$. Let $\Lambda =\Lambda
_{G}$ denote a Hopf $R$-order in $K[G]$. We wish to define a
duality map on $G_{0}( \Lambda)$.  For a $\Lambda$-lattice $M$
we let $M^{D}=\mathrm{Hom}_{R}( M,R)$, endowed with a $\Lambda$-module structure  given by 
$$(\lambda f):m\rightarrow f(S(\lambda)m\ \forall m\in M, \lambda \in \Lambda$$
where $S$ is the antipode of $\Lambda$.  We say that such a module is self-dual if
there is an isomorphism of $\Lambda $-modules $M\cong M^{D}$. Since $
\mathrm{Hom}_R( -,R)$ preserves exact sequences of free $R$-modules, the
map $M\mapsto M^{D}$ induces an involution $G_{0}^{R}(\Lambda)$. We let $
G_{0}^{R +}( \Lambda)$ denote the subgroup of $G_0^R(\Lambda)$ generated by classes of
self-dual finitely generated $\Lambda $-lattices. This is a subring of $G_0^R(\Lambda)$.

\noindent{\bf Duality for $G_{0}( \widetilde{\Lambda })$.}  Recall
that the reduction map $\Lambda \rightarrow \widetilde{\Lambda }$ induces a
ring homomorphism $\delta_{\Lambda} :G^R_{0}(\Lambda) \rightarrow
G_{0}( \widetilde{\Lambda })$. For a finitely generated $\widetilde{\Lambda }$-module $N$ we set $
N^{D}=\mathrm{Hom}_{k}( N,k)$.  The functor $\mathrm{ Hom}_{k}(
-,k)$ preserves exact sequences of finitely generated $\widetilde{
\Lambda }$-modules and so, as previously, the map $N\rightarrow N^D$  induces an involution on $
G_{0}( \widetilde{\Lambda })$.  Note that the reduction map is
natural for duality in the sense that for a finitely generated $\Lambda $-lattice $M$
\begin{equation*}
\delta_{\Lambda}( M^{D}) =\delta_{\Lambda} ( M) ^{D}
\end{equation*}
by virtue of the natural isomorphism 
\begin{equation*}
\mathrm{Hom}_{R}( M,R) \otimes _{R}k\cong \mathrm{Hom}_{k}( M\otimes
_{R}k,k) .
\end{equation*}

\begin{prop}\label{triv1}
 Let $H$ be an $l$-elementary group  of odd order and let $\Lambda:=\Lambda_H$ be a Hopf order of $K[H]$. We assume that $p>2$, $l\neq p$ and that $K$ is assez gros for $H$. Then 
$ \widetilde{\Lambda}_{H}^{ss}$ has no non-trivial simple self-dual modules. 
\end{prop}
\begin{proof} We recall that   we may write $ H=C\times L$ where the group $C$ is a cyclic $p$-group and $L$ has order
prime to $p$. Under these hypotheses    we know from (\ref{dec1}) that there exists   a natural isomorphism of $k$-algebras with involutions:
$$\widetilde{ \Lambda}^{ss} _{C}\otimes _{k} \widetilde{\Lambda}^{ss} _{L}\simeq \widetilde{\Lambda}^{ss} _{H}.$$ Therefore,  in order to prove  Proposition \ref{triv1},  it suffices to prove that the result  holds for  
 $\widetilde{\Lambda}^{ss} _{C}$ and   $\widetilde{\Lambda}^{ss} _{L}$. Since  $L$ is a group of order coprime to $p$, then $\widetilde{\Lambda}^{ss} _{L}=\widetilde{\Lambda} _{L}=k[L]$. Since the order of $L$ is odd then  $k[L]$ has no non-trivial simple self-dual module (see \cite{Serrerep} Section 15.5, Proposition 43). 

 We now consider the semisimple $k$-algebra $\widetilde{\Lambda}_C^{ss}$. By (\ref{decet}) we have  isomorphisms  
 $$\widetilde{\Lambda}_C^{ss} \simeq \widetilde{\Lambda}_C^{et}\simeq \widetilde{\Lambda_C^{et}}.$$
 We know that  $\Lambda_{C}^{et}$ is a Hopf $R$-order of $K[D]$, where $D$ is a subgroup of $C$ (see Proof of Theorem 4.16). Since $\Lambda_{C}^{et}$ is separable and $K$ is assez gros for $H$,  it follows from \cite{L72} that this order  is the unique maximal order of $K[D]$. The group $D$ is a $p$-group,  we denote by $r$ its order. If $\{\chi_1, \cdots, \chi_r\}$  are the irreducible $K$-characters of $D$, then 
\begin{equation}\label{dec3}\Lambda_{C}^{et}=\oplus_{1\leq i\leq r}\Lambda_{C}^{et}e_i=\oplus_{1\leq i\leq r}Re_i, \end{equation}
with $e_i=\frac{1}{r}\sum_{x\in D}\chi(x^{-1})x$ and thus 
\begin{equation}\label{find}\widetilde{ \Lambda}_{C}^{et}=\oplus_{1\leq i\leq r}k\tilde e_i.\end{equation} Since $1=\sum_{1\leq i\leq r}e_i$ is a decomposition of $1$ into  primitive orthogonal idempotents, then (\ref{find}) is a decomposition of $\widetilde{ \Lambda}_{C}^{et}$ into minimal left ideals.  Moreover,  since  $\widetilde{ \Lambda}_{C}^{et}$ is a split commutative semi-simple $k$-algebra, then $\{k\tilde e_1,\cdots, k\tilde e_r\}$ are non-isomorphic $\widetilde{ \Lambda}_{C}^{et}$-left ideals (see \cite{R75}  Section 6.C). In particular we deduce  that $k\tilde e_i=k\tilde e_j$ implies  $\Lambda_{C}^{et}e_i\simeq \Lambda_{C}^{et}e_j$ and so  $e_i=e_j$.  We let $S$ denote the antipode of $\Lambda_{C}^{et}$. Suppose that $k\tilde e_i$ is a self-dual simple module, then we have:
$$\tilde S(k\tilde e_i)=k\tilde S(\tilde e_i)=k\widetilde{ (S(e_i)})=k\tilde e_i. $$ This implies that  ${\Lambda_C^{et}}S(e_i)\simeq {\Lambda_C^{et}}e_i$ and so that $S(e_i)=e_i$. If $e_i$ is attached to the character $\chi_i $, then $S(e_i)$ is attached to the character $\chi_i^{-1}$. Therefore if $S(e_i)=e_i$  then  $\chi_i=\chi_i^{-1}$ and thus   $\chi_i$ is the trivial character since $r$ is odd. We conclude that  $\widetilde{\Lambda_C^{et}}$ has no non-trivial self-dual modules. This  completes the proof of the proposition. 
\end{proof}

\noindent {\bf Character action for self-dual classes.}
 We note that the reduction map $d_{{\Lambda}_G}: G_0(K[G])\rightarrow G_0(\widetilde{\Lambda}_G)$ induces by restriction a morphism of rings 
 $d^+_{{\Lambda}_G}: G^+_0(K[G])\rightarrow G^+_0(\widetilde{\Lambda}_G)$. We  consider $G_0^+{(\widetilde{\Lambda}_G}) $ endowed with a structure of $G^+_0(K[G])$-module  via $d^+_{{\Lambda}_G}$.  We want to check that the $H\rightarrow G^+_0(\widetilde{\Lambda}_H)$ is a Frobenius module over the Frobenius functor $H\rightarrow G^+_0(K[H])$. Indeed, restriction provides  homomorphisms of Grothendieck rings 
 $$\mathrm{Res}_{\Lambda _{G}}^{\Lambda _{H}}: G_{0}^{R +}( \Lambda_{G}) \rightarrow G_{0}^{R +}( \Lambda _{H})\ \mathrm{and}\   \mathrm{Res}_{\widetilde{\Lambda} _{G}}^{\widetilde{\Lambda}_{H}}: G^+_{0}(\widetilde{\Lambda}_{G}) \rightarrow G^+_{0}(\widetilde{\Lambda} _{H}).$$ A key point is  to show that induction induces group homomorphisms 
$$\mathrm{Ind}_{\Lambda _{H}}^{\Lambda _{G}}: G_{0}^{R +}( \Lambda_{H}) \rightarrow G_{0}^{R +}( \Lambda _{G})\ \mathrm{and}\   \mathrm{Ind}_{\widetilde{\Lambda} _{H}}^{\widetilde{\Lambda}_{G}}: G^+_{0}(\widetilde{\Lambda}_{H}) \rightarrow G^+_{0}(\widetilde{\Lambda} _{G}).$$
This will follow from Theorem \ref{indu} and Proposition \ref{symp}. 

We let $H$ be a subgroup of the finite group $G$, we consider  a Hopf $R$-order  $\Lambda_G$ of $K[G]$ and we set $\Lambda_H=\Lambda_G\cap K[H]$; this is a Hopf $R$-order of $K[H]$.  Indeed $\Lambda_G$ is a $\Lambda_H$-bimodule by left and right mutiplication by $\Lambda_H$. For any left module $M$ we consider  the following $R$-modules: 
$$(\Lambda_G\otimes_{\Lambda_H}M)^D:=\mathrm{Hom}_R(\Lambda_G\otimes_{\Lambda_H}M, R)\ \ \mathrm{and} \ \Lambda_G\otimes_{\Lambda_H}M^D:= \Lambda_G\otimes_{\Lambda_H}\mathrm{Hom}_R(M, R).$$ Both modules can be endowed with a $\Lambda_G$-module structure. 
  The $\Lambda_G$-module structure of $(\Lambda_G\otimes_{\Lambda_H}M)^D$ is given by 
\begin{equation}\label{acl}(\lambda.f)(x\otimes_{\Lambda_H}m)=f(S(\lambda)x\otimes_{\Lambda_H}m)\ \ \forall  \lambda, x \in \Lambda_G, m\in M. \end{equation}
This is the usual $\Lambda_G$-module structure for the dual of the  $\Lambda_G$-module $\Lambda_G\otimes_{\Lambda_H} M$.  
  The $\Lambda_G$-module structure of $\Lambda_G\otimes_{\Lambda_H}M^D$ is defined  by
\begin{equation}\lambda.(x\otimes_{\Lambda_H}g)=(\lambda x)\otimes_{\Lambda_H}g. \ \forall  \lambda, x \in \Lambda_G, g\in M^D .\end{equation}
This is the $\Lambda_G$-module structure obtained by left multiplication.

For the sake of simplicity for $f: X\rightarrow Y$ we often write $<f, x>$ for $f(x)$.  For an $R$-module $M$ we write $M_K=K\otimes_RM$ and for a morphism of $R$-modules $f:M\rightarrow P$ we write $f_K= M_K\rightarrow P_K$ for the morphism of $K$-vector spaces induced by $f$. During the proof of Theorem \ref{indu} we often use the following Lemma. 
\begin{lem}\label{simp} Let $N$ be a finite group, $\Lambda$ be an $R$-order of $K[N]$ and let $M$ and $P$ be finitely generated $\Lambda$-lattices so that  
$\mathrm{Hom}_{\Lambda}(M, P)$ is a $R$-lattice of $K\otimes_R\mathrm{Hom}_{\Lambda}(M, P)$. Tensoring by $K$ yields an isomorphism of $K$-vector spaces 
$$K\otimes_R\mathrm{Hom}_{\Lambda}(M, P)\simeq {Hom}_{K[N]}(M_K, P_K).$$
 
Then $f\in \mathrm{Hom}_{R}(M, P)$ lies in $\mathrm{Hom}_{\Lambda}(M, P)$  iff $f_K$  lies  in $\mathrm{Hom}_{K[N]}(M_K, P_K) $.
\end{lem} 
\begin{proof} From \cite{R} Theorem 3.84 we obtain the isomorphism of $K$-vector spaces above.  Therefore, for $\lambda\in \Lambda$ and  $m\in M$,  we have $f(\lambda m)=f_K(\lambda m)=\lambda f_K(m)=\lambda f(m)$
\end{proof}
\begin{thm} \label{indu} For any left $\Lambda_H$-module $M$, there exists an isomorphism of $\Lambda_G$-modules
$$(\Lambda_G\otimes_{\Lambda_H}M)^D\simeq \Lambda_G\otimes_{\Lambda_H}M^D.$$
\end{thm} 
\begin{proof}
 We proceed  in three steps. We prove successively the existence of isomorphisms of $\Lambda_G$-modules
$$ (\Lambda_G\otimes_{\Lambda_H}M)^D\simeq \mathrm{Hom}_{\Lambda_H}(M, \Lambda_G) $$ 
$$\mathrm{Hom}_{\Lambda_H}(M, \Lambda_G)\simeq \Lambda_G\otimes_{\Lambda_H}\mathrm{Hom}_{\Lambda_H}(M, \Lambda_H)$$
and finally the isomorphism of $\Lambda_H$-modules 
 $$ \mathrm{Hom}_{\Lambda_H}(M,\Lambda_H)\simeq M^D.$$ 
\noindent{\bf  Step 1.} We endow $\Lambda_G^D$ with a structure of left $\Lambda_H$-module structure  by setting
\begin{equation}<\alpha*u, x>=<u, x\alpha >\   \forall \alpha \in \Lambda_H, u\in \Lambda_G^D, x\in \Lambda_G \end{equation}
and we consider the $R$-module:  
$$\mathrm{Hom}_{\Lambda_H}(M, \Lambda_G^D)=\{f: M\rightarrow \Lambda_G^D\ \mid f(\alpha.m)=\alpha*f(m), \forall \alpha \in \Lambda_H, m\in M \}.$$

Since $\Lambda_G$ is a left $\Lambda_G$-module by left mutiplication, then  its dual $\Lambda_G^D$ is a left $\Lambda_G$-module  with 
\begin{equation}<\lambda.v, x>=<v, S(\lambda)x>, \forall \lambda \in \Lambda_G, v\in \Lambda_G^D, x\in \Lambda_G. \end{equation}
We claim that we can endow $ \mathrm{Hom}_{\Lambda_H}(M, \Lambda_G^D)$ with a structure of left $\Lambda_G$-module by defining    $\lambda.f$, for $f\in \mathrm{Hom}_{\Lambda_H}(M, \Lambda_G^D)$ and $\lambda \in \Lambda_G$,   as the map:  
\begin{equation}\lambda.f: m\rightarrow \lambda.f(m)\end{equation}
i.e \begin{equation}<(\lambda.f)(m), x>=<\lambda.f(m), x>=<f(m), S(\lambda)x>, \forall x \in \Lambda_G.\end{equation}
  We have to check that this morphism of $R$-modules is a morphism of $\Lambda_H$-modules  and so   that 
$$(\lambda.f)(\alpha m)=\alpha*((\lambda.f)(m))\ \forall \alpha\in \Lambda_H, m\in M . $$ This follows from the equalities, for all $x\in \Lambda_G$
$$<(\lambda.f)(\alpha m), x>=<f(\alpha m),S(\lambda)x>=<f(m),S(\lambda)x\alpha>=<\alpha*(\lambda.f)(m),x>.$$

We now  introduce the map 
$F:(\Lambda_G\otimes_{\Lambda_H}M)^D\rightarrow \mathrm{Hom}_R (M, \Lambda_G^D)$   given by 

$$F(f): m\rightarrow (x\rightarrow f(x\otimes_{\Lambda_H} m)) \ \forall m\in M, x \in \Lambda_G.$$
\begin{prop}\label{st1} The map $F$ is an isomorphism of $\Lambda_G$-modules. 
\end{prop}
\begin{proof} 

By  Theorem 2.11 of \cite{Ro}  we know that $F$ is the adjoint  isomorphism of $R$-modules. Therefore it suffices to prove that $F$ is an isomorphism of $\Lambda_G$-modules. This follows from the equalities:
$$<F(\lambda.f)(m), x>=(\lambda.f)(x\otimes_{\Lambda_H}m)=f(S(\lambda)x\otimes_{\Lambda_H}m)$$ (see (\ref{acl}))

$$<(\lambda.F(f))(m), x>=<F(f)(m), S(\lambda)x>=f(S(\lambda)x\otimes_{\Lambda_H}m).$$

\end{proof}

We now consider $$\mathrm{Hom}_{\Lambda_H}(M, \Lambda_G)=\{f: M\rightarrow \Lambda_G)\ \mid f(\alpha m)=\alpha f(m) \ \forall m\in M, \alpha \in \Lambda_H\}$$
 endowed with the left $\Lambda_G$-module structure given by:
\begin{equation}\label{acc}(\lambda.f)(m)=f(m)S(\lambda), \ \forall m\in M, \lambda\in \Lambda_{G}.\end{equation}
We recall that  we have fixed a basis $\theta$ of $\Lambda^D_G$ as a left  $\Lambda_G$-module. Hence for every $u\in \Lambda^D_G$ there exists $i(u)\in \Lambda_G$ such that $u=i(u)\theta$. By Proposition 2.3 the  map 
$i: \Lambda^D_G\rightarrow \Lambda_G$ is given 
 by \begin{equation}i: u\rightarrow \lambda^{-1}\sum_{g\in G}u(g^{-1})g.\end{equation} We let 
 $\hat i: \mathrm{Hom}_{\Lambda_H}(M, \Lambda^D_G)\rightarrow \mathrm{Hom}_{\Lambda_H}(M, \Lambda_G)$ be the map defined by 
 $$\hat i(f)(m)=i(f(m)).$$
 \begin{lem}\label{imp} The map $$\hat i: \mathrm{Hom}_{\Lambda_H}(M, \Lambda^D_G)\rightarrow \mathrm{Hom}_{\Lambda_H}(M, \Lambda_G)$$ is an isomorphism of $\Lambda_G$-modules. 
 \end{lem}
\begin{proof}Indeed  $\hat i$ is  a bijection. 
We start by proving that  for $f\in \mathrm{Hom}_{\Lambda_H}(M, \Lambda^D_G)$ then $\hat i(f)$ is a $\Lambda_H$-morphism.  By Lemma \ref{simp} it suffices to show that for $f$ in 
$\mathrm{Hom}_{K[H]}(M_K, \mathrm{Hom}_K(K[G], K))$ we have: 
\begin{equation} \hat i(f)(hm)=h\hat i(f)(m) \ \forall h\in H, m\in M_K.\end{equation}  By definition of $\hat i$  we have:  
\begin{equation}\hat i(f)(hm)= i(f(hm))=i(h*f(m))=\lambda^{-1}\sum_{t\in H}f(m)(t^{-1}h)t.    
 \end{equation}                                                                                                                                                                            
By  writing $t^{-1}h=v^{-1}$  we obtain from the (73) 
\begin{equation}\lambda^{-1}\sum_{t\in H}f(m)(t^{-1}h)t=\lambda^{-1}\sum_{v\in H}f(m)(v{^{-1}})hv=h(\lambda^{-1}\sum_{v\in H}f(m)(v^{-1})v).\end{equation}
Thus (72) follows from (73) and (74).  

It remains to prove that $f \rightarrow \hat i(f)$  is a morphism of $\Lambda_G$-modules and so to check that for $f\in \mathrm{Hom}_{\Lambda_H}(M, \Lambda^D_G)$
$$\hat i(gf)=g\hat i(f)\  \forall g \in G$$  and the proof is entirely similar. 
 \end{proof}
As a consequence of Proposition \ref{st1} and Lemma \ref{imp} we obtain  
\begin{cor}\label{cst1} The map $F_1=\hat i \circ F$
 is an isomorphism of $\Lambda_G$-modules 
$$(\Lambda_G\otimes_{\Lambda_H}M)^D\rightarrow \mathrm{Hom}_{\Lambda_H}(M, \Lambda_G).$$

\end{cor}
\vskip 0.1 truecm
\noindent{\bf Step 2.} We consider 
 $$\mathrm{Hom}_{\Lambda_H}(M, \Lambda_H)=\{f: M\rightarrow \Lambda_H\ \mid\ f(\alpha.m)=\alpha.f(m)\ \forall m\in M, \alpha\in \Lambda_H\} $$ 
 endowed with the left $\Lambda_H$-module structure given by  
\begin{equation}\lambda.f: m\rightarrow f(m)S(\lambda). \end{equation}  Since $\Lambda_G$ is a right $\Lambda_H$-module we can consider 
$$\Lambda_G\otimes_{\Lambda_H}\mathrm{Hom}_{\Lambda_H}(M, \Lambda_H).$$
This a left $\Lambda_G$-module by left mutiplication. 

By Proposition 4.23 we know  that $\Lambda_G$ is a free left $\Lambda_H$-module.  We fix a basis  $\{e_1, \cdots, e_n\}$ of $\Lambda_G$ as a right $\Lambda_H$-module  and we write $\Lambda_G=\oplus_i e_i\Lambda_H$.  Therefore for every element 
$u\in \Lambda_G\otimes_{\Lambda_H}\mathrm{Hom}_{\Lambda_H}(M, \Lambda_H)$  there exists a unique family $\{f_1,\cdots, f_n\}$ of $\mathrm{Hom}_{\Lambda_H}(M, \Lambda_H)$ such that $u=\sum_ie_i\otimes_{\Lambda_H}f_i$. 

  Using the antipode $S$ of $\Lambda_G$ we note that $\Lambda_G$ is free left  $\Lambda_H$-module with basis
$\{S(e_1), \cdots, S(e_n\}$. Therefore for any $g\in \mathrm{Hom}_{\Lambda_H}(M, \Lambda_G)$ there exist a unique family $\{g_1, \cdots, g_n\}$ 
of $\mathrm{Hom}_{\Lambda_H}(M, \Lambda_H)$ 
 such that 
 $$g(m)=g_1(m)S(e_1)+\cdots +g_n(m)S(e_n), \forall m\in M.$$
 We denote by   $g_iS(e_i): M\rightarrow \Lambda_G$ the map given by $ m\rightarrow g_i(m)S(e_i)$. Indeed,  
 each $g_iS(e_i) \in \mathrm{Hom}_{\Lambda_H}(M, \Lambda_G)$ and we can write  $g=\sum_i g_iS(e_i)$.
 \begin{prop}\label{st2} The map 
 $$F_2: \Lambda_G\otimes_{\Lambda_H}\mathrm{Hom}_{\Lambda_H}(M, \Lambda_H)\rightarrow \mathrm{Hom}_{\Lambda_H}(M, \Lambda_G)$$ defined by the rule
 $$\sum_ie_i\otimes_{\Lambda_H}f_i\rightarrow \sum_i f_iS(e_i)$$
 is an isomorphism of $\Lambda_G$-modules.  
 \end{prop} 
 \begin{proof} It  follows from the definition that $F_2$ is a bijection. Therefore it remains to show that $F_2$ commutes with the action of $\Lambda_G$ and so  to check that for every $ i , 1\leq i\leq n$ we have the equality
$$F_2(\alpha.(e_i\otimes_{\Lambda_H}f_i))=\alpha.F_2(e_i\otimes _{\Lambda_H}f_i)\ \forall \alpha\in \Lambda_G.$$
For $\alpha \in \Lambda_G$ there exist $\{\alpha_{i,j}\in \Lambda_H,  1\leq j\leq n\}$ such that 
\begin{equation}\alpha e_i=\sum_je_j\alpha_{i,j}.\end{equation} 
It follows from (75) and (76) that 
$$\alpha(e_i\otimes f_i)=\alpha e_i\otimes_{\Lambda_H} f_i=\sum_je_j\alpha_{i,j}\otimes_{\Lambda_H} f_i=\sum_je_j\otimes_{\Lambda_H} \alpha_{i,j}.f_i=\sum_je_j\otimes_{\Lambda_H} f_iS(\alpha_{i,j}).$$
Therefore  
\begin{equation}F_2(\alpha.(e_i\otimes_{\Lambda_H}f_i))=\sum_jf_iS(\alpha_{ij})S(e_j)=f_iS(\sum_j e_j\alpha_{ij})=f_iS(\alpha e_i)=(f_iS(e_i))S(\alpha).\end{equation} 
From the definition  of the action of $\Lambda_G$ on  $\mathrm{Hom}_{\Lambda_H}(M, \Lambda_G)$ (see  (70))  we can write  (77) as follows
$$F_2(\alpha.(e_i\otimes_{\Lambda_H}f_i))=(f_iS(e_i))S(\alpha)=\alpha.F_2(e_i\otimes_{\Lambda_H}f_i)\ \forall \alpha \in \Lambda_G.$$
 We conclude that $F_2$ is a morphism of $\Lambda_G$-modules.
 \end{proof}

\vskip  0.2 truecm

\noindent{\bf Step 3.}

For a left $\Lambda_H$-module $M$ we consider $\mathrm{Hom}_{\Lambda_H}(M, \Lambda_H^D)$ with 
\begin{equation}\mathrm{Hom}_{\Lambda_H}(M, \Lambda_H^D)=\{ f:M \rightarrow \Lambda_H^D  \mid  f(\alpha.m)= \alpha*f(m)\}.\end{equation}
 For 
$f\in \mathrm{Hom}_{\Lambda_H}(M, \Lambda_H^D)$ and $\lambda \in \Lambda_H$  we set 
\begin{equation}< (\lambda.f)(m), x>=<f(m), S(\lambda)x>\ \forall x\in \Lambda_H. \end{equation}
Indeed, $\lambda.f \in \mathrm{Hom}_R(M, \Lambda_H^D)$.  We check by hand that  
$$(\lambda.f)(\alpha m)=\alpha*(\lambda.f)(m)\ \forall \alpha \in \Lambda_H$$ and so $\lambda.f\in \mathrm{Hom}_{\Lambda_H}(M, \Lambda_H^D)$. Therefore (79) provides us with a structure of left $\Lambda_H$-module on  $\mathrm{Hom}_{\Lambda_H}(M, \Lambda_H^D)$.  

We consider the map 
$$\fonc{\varphi}{\mathrm{Hom}_{\Lambda_H}(M, \Lambda_H^D)}{\mathrm{Hom}_R(M, R)}
{f}{m\rightarrow \varepsilon_D(f(m))}.$$ Our aim is to prove that $\varphi$ is an isomorphism of $\Lambda_H$-modules. This will follow from Proposition \ref{st3}.

Since $M$ is a $\Lambda_H$-module it is a $\Lambda_H^D$-comodule. We denote by $\rho: M\rightarrow M\otimes_R\Lambda_H^D$ the comodule map. 
We  introduce  the map 
$$
\fonc{\psi}{\mathrm{Hom}_R(M, R)}{\mathrm{Hom}_{R}(M, \Lambda_H^D)}
{f}{(m\rightarrow \sum_{(m)}f(m_{(0)})u_{(1)})}
$$ 
\begin{prop}\label{st3} The following properties hold:
\begin{enumerate}
\item $\varphi$ is a morphism of $\Lambda_{H}$-modules.
\item For $f\in \mathrm{Hom}_R(M, R))$, then  $\psi(f)  \in \mathrm{Hom}_{\Lambda_H}(M, \Lambda_H^D)$.
\item  The maps $\varphi$ and $\psi$ are   isomorphisms of $\Lambda_H$-modules such that   $$\varphi\circ\psi=\psi\circ \varphi= id.$$
 \end{enumerate}

\end{prop}
\begin{proof}   We recall that 
$$K\otimes_R\mathrm{Hom}_{\Lambda_H}(M, \Lambda_H^D)=\mathrm{Hom}_{K[H]}(M_K, K[H]^D).$$
We start by proving (1). By Lemma \ref{imp},  it suffices to show that for 
 $f\in \mathrm{Hom}_{K[H]}(M_K, K[H]^D)$, $h\in H$ and $m\in M_K$ we have $\varphi_K(hf)(m)=h\varphi_K(f)(m)$. This follows from the equalities 
 \begin{equation}\varphi_K(h.f)(m)=\varepsilon_D((h.f)(m))=<(h.f)(m), 1>=<f(m), h^{-1}>\end{equation}
 \begin{equation}(h.\varphi_K (f))(m)=\varphi_K(f)(h^{-1}m)=\varepsilon(f(h^{-1} m))=<f(h^{-1} m), 1>\end{equation}
and, since $f$ is a $K[H]$-morphism:
\begin{equation}<f(h^{-1} m), 1>=<h^{-1}*f(m), 1>=<f(m), h^{-1}>. \end{equation}

We now want to prove (2). Again by Lemma \ref{imp} we are reduced to  showing that for $f\in \mathrm{Hom}_{K}(M_K, K)$, $m\in M_K$ and $h\in H$, then 
\begin{equation}\label{k1}\psi_K(f)(hm)=h*\psi_K(f)(m).\end{equation} We recall that $\{l_t, t\in H\}$ is a basis of the $K$-vector space $K[H]^D=\mathrm{Map}(H, K)$ with $l_t(h)=\delta_{t, h}$. Moreover,  for $m\in M_K$,   the term $\rho(m)$ simplifies  to  
\begin{equation}\label{k2}\rho(m)=\sum_{t\in H}(tm)\otimes_Kl_t.\end{equation} 
Then we have 
\begin{equation}\label{k3}\psi_K(f)(m)=\sum_{t\in H} f(tm)l_t.\end{equation}Therefore 
\begin{equation}\label{k33}\psi_K(f)(hm)=\sum_{t\in H}f((th)m)l_t=\sum_{v\in H}f(vm)l_{vh^{-1}}.\end{equation}
We  check  that $h*l_t=l_{th^{-1}}$ and so   
\begin{equation}\label{k4} h*\psi_K(f)(m)=\sum_{t\in H} f(tm)(h*l_t)=\sum_{t \in H}f(tm)l_{th^{-1}}.\end{equation}
We deduce (\ref{k1}) from (\ref{k33}) and (\ref{k4}) and so $\psi(f)$ is a $\Lambda_H$-morphism. 

 We  now start by observing that $\varphi$ is clearly injective since $\varphi_K$ is injective. Therefore, in order to prove (3),  it suffices to show that $\varphi\circ \psi=id$. Indeed,  this   will show that $\varphi$ is surjective and hence is an isomorphism. Since $\varphi\circ\psi=id$,  we will conclude  that $\psi $ is the inverse isomorphism  of $\varphi$  and so are both isomorphisms of $\Lambda_H$-modules. 
 
 To show that 
 $\varphi\circ \psi=id$ by Lemma \ref{imp} it suffices to check that $\varphi_K\circ\psi_K=id$.
 We consider  $f\in \mathrm{Hom}_K(M_K, K)$ and $m\in M_K$.  From the above we have  as required:  
$$(\varphi_K \circ\psi_K)(f)(m)=\varphi_K(\psi_K(f))(m)=\varepsilon_D(\psi_K(f)(m))=\sum_tf(tm)\varepsilon_D(l_t)=f(m). $$
\end{proof}

We now follow   the lines of Lemma \ref{imp}, using $\theta_H$ instead of $\theta$ and $i_H$ instead of $i $ to  prove: 
\begin{lem}\label{imp3}  The map $$\hat i_H: \mathrm{Hom}_{\Lambda_H}(M, \Lambda^D_H)\rightarrow \mathrm{Hom}_{\Lambda_H}(M, \Lambda_H)$$ is an isomorphism of $\Lambda_H$-modules. 
 \end{lem}
 Therefore we deduce from Proposition \ref{st3} and Lemma \ref{imp3} 
 \begin{cor}\label{cst3}
 There exist isomorphisms of $\Lambda_H$-modules
 $${\mathrm{Hom}_{\Lambda_H}(M, \Lambda_H)}\simeq {\mathrm{Hom}_{\Lambda_H}(M, \Lambda_H^D)}\simeq M^D$$
 \end{cor}

Using now Corollary 4.27, Proposition \ref{st2} and Corollary 4.32  we finally obtain that as required  there exists an isomorphism of $\Lambda_G$-modules
$$(\Lambda_G\otimes_{\Lambda_H}M)^D\simeq \Lambda_G\otimes_{\Lambda_H}M^D.$$ This completes the proof of Theorem \ref{indu}.
\end{proof}
 It follows from Theorem \ref{indu} that $\mathrm{Ind}_{\Lambda_H}^{\Lambda_G}$ induces a group homomorphism 
 $$G_0^{R +}(\Lambda_H)\rightarrow G_0^{R +}(\Lambda_G).$$ We want to show that these  results remain true in positive odd characteristic. This  follows from the next proposition.  
 As before we consider  the Hopf  $R$-orders $\Lambda_H\subset \Lambda_G$ and,  by tensoring by the finite field $k$,  the Hopf $k$-algebras $\widetilde \Lambda_H\subset \widetilde \Lambda_G$.  We have defined by scalar extension  group homomorphisms $ \mathrm{Ind}_{\Lambda_H}^{\Lambda_G}$ 
and $\mathrm{Ind}_{\widetilde\Lambda_H}^{\widetilde\Lambda_G}$ that for the sake of simplicity we simply denote by $\mathrm{ind}$
$$\mathrm{Ind}: G_0^R(\Lambda_H)\rightarrow G_0^R(\Lambda_G)\ \ \mathrm{and}\ \ \mathrm{Ind}: G_0^R(\widetilde\Lambda_H)\rightarrow G_0^R(\widetilde\Lambda_G)
.$$ We recall that the map which associates  to a module its dual induces an involution on the Grothendieck groups involved, denoted by $x\rightarrow x^D$.
\begin{prop}\label{symp}For every $y\in G_0(\widetilde{\Lambda_H})$ the following equality holds 
$$\mathrm{Ind}(y^D)=\mathrm{Ind}(y)^D$$
\end{prop}
\begin{proof} We have a commutative diagram 
 \[\xymatrix{
G_0^R(\Lambda_H)\ar[r]^{\mathrm{Ind}}\ar[d]_{\delta_H}&G_0^R(\Lambda_G)\ar[d]_{\delta_G} \\
G_0(\widetilde{\Lambda_H})\ar[r]^{\mathrm{Ind}}&G_0(\widetilde{\Lambda_G}). \\}\]Let $y\in G_0(\widetilde{\Lambda_H})$. It follows from Theorem \ref{majeur} that  there exists $x\in G_0^R(\Lambda_H)$ such that $p^my=\delta_H(x)$ and so 
$p^my^D=\delta_H(x^D)$. Then we deduce from the commutativity of the diagram that  
$$p^m\mathrm{Ind}(y^D)=\mathrm{Ind}(p^my^D)=\mathrm{Ind}\circ \delta_H(x^D)=\delta_G\circ \mathrm{Ind}_H(x^D).$$
We know from Theorem \ref{indu} that 
$$ \mathrm{Ind}(x^D)=(\mathrm{Ind}(x)^D)$$
Therefore we obtain  
$$p^m\mathrm{Ind}(y^D)=(\delta_G(\mathrm{Ind}(x)))^D=(\mathrm{Ind}(\delta_H(x)))^D=p^m(\mathrm{Ind}(y))^D.$$
Since $G_0(\widetilde\Lambda_G)$ is a free $\mathbb{Z}$-module we conclude that 
$$\mathrm{Ind}(y^D)=\mathrm{ind}(y)^D.$$
\end{proof}
It follows from the properties of the restriction and  induction maps, in particular  Theorem 4.26,   that $H\rightarrow G_0^{+}(\widetilde {\Lambda_H})$ has a Frobenius module structure over 
$H\rightarrow G_0^+(K[H])$. We use this structure to study the self-dual representations of $\widetilde{\Lambda_G}$. 

By  the Brauer induction formula of (\ref{br1}) we have: 
\begin{equation}
 p^{m}\varepsilon _{G}=\sum_{H\in C(G) }\mathrm{Ind}_{H}^{G}( \theta _{H})
\end{equation}
for virtual characters $\theta _{H}\in G^+_{0}(K[H])$. Therefore,   for any $x \in G^+_0( \widetilde{
\Lambda }_H) $,  we obtain
\begin{equation}\label {gindu}
p^{m}x  =\sum_{H\in C(G) }\mathrm{
Ind}_{\widetilde{\Lambda} _{H}}^{\widetilde{\Lambda} _{G}}(\theta _{H}.\mathrm{Res}
_{\widetilde{\Lambda} _{H}}^{\widetilde{\Lambda} _{G}}x).
\end{equation}

Recall that a finitely generated self-dual $\widetilde{\Lambda }_{G}$-module $ M$ is called symplectic if there is a duality isomorphism $
f:M\cong M^{D}$ such that the corresponding form $M\times M\rightarrow k$ is
an alternating form. Recall by standard algebra that $\dim _{k}(
M)$ is necessarily even for such symplectic $M$. More generally we  call a class symplectic if it lies  in  the additive subgroup of $ G_0(\widetilde{\Lambda}_G)$ generated by the classes of symplectic $\widetilde{\Lambda}_G$-modules.

We now can use Proposition \ref{triv1} and the induction formula (\ref{gindu}) to show: 
\begin{thm}\label{nosimp}
If $p>2$, if $G$ has odd order and if $K$ is 
 assez gros for $G$,  then $\widetilde{\Lambda}_G^{ss}$  has  no self-dual and simple   module which is of even dimension on $k$.

\end{thm}
\begin{proof} Suppose for contradiction that  $M$  is a simple module which is self-dual.  Then, continuing with the notation of (\ref{gindu}),  we set  $$x_H:=\theta _{H}.\mathrm{Res}
_{\widetilde{\Lambda} _{H}}^{\widetilde{\Lambda} _{G}}[ M]\in G^+_0( \widetilde{
\Lambda }_H) $$   for any $H\in C(G)$.  Since by   Proposition \ref{triv1} we know that $\widetilde{\Lambda}_H$ has no non-trivial, self-dual, simple modules then  there exist   non self-dual  $\widetilde{\Lambda}_H$-modules
$\{S_1,\cdots, S_t\}$ such that $\{[k], [S_i], [S_i^D], 1\leq i\leq t\}$ is a basis of $G_0(\widetilde{\Lambda}_H)$ over $\mathbb{Z}$ . Therefore  we can write
\begin{equation}\label{el}
x_H=x_0[k]+\sum_{1\leq i\leq r}x_i([S_i]+S_i^D]) 
\end{equation}
with $x_0\cdots, x_r \in \mathbb{Z}$.  Moreover,  since  $\mathrm{dim}_k(M)$ is even then $x_0$  is even. After applying induction $\mathrm{Ind}_{\widetilde{\Lambda}_H}^{\widetilde{\Lambda}_G}$,  taking into consideration  the signs of the $x_i$'s, and using Theorem \ref{indu},   we deduce from (\ref{el})  that  for any $H\in C(G)$  there exist  an integer $a_0(H)$ and $\widetilde{\Lambda}_G$-modules $N_H$ and $L_H$ such that 
\begin{equation}\label{sum1}\mathrm{Ind}_{\widetilde{\Lambda}_H}^{\widetilde{\Lambda}_G}(x_H)={\varepsilon}_H([N_H+N_H^D])-{\varepsilon}'_H([L_H+L_H^D])+2a_0(H)[k]  \end{equation}
with $\varepsilon_H, \varepsilon_H' \in \{0, 1\}$. By summing over $H \in C(G)$  it follows from (\ref{gindu}) and (\ref{sum1}) that one can find $\widetilde{\Lambda}_G$-modules $U$ and $V$ such that 
\begin{equation}\label{dc} p^m[M]=\varepsilon ([U]+[U^D])-\varepsilon' ([V]+[V^D])+2\sum_{H\in C(G)} a_0(H)\mathrm{
Ind}_{\widetilde{\Lambda} _{H}}^{\widetilde{\Lambda} _{G}}([k]) \end{equation} 
with $\varepsilon, \varepsilon' \in \{0, 1\}$. We now consider $$\{ [k], [X_1], \cdots [X_r], [Y_1], \cdots ,[Y_s], [Y_1^D], \cdots ,[Y_s^D]\}$$ a basis of $G_0(\widetilde{\Lambda})$ where each $X_i$ (resp. $Y_j$) is a simple non-trivial self-dual (resp. simple non self-dual) module of $G_0(\widetilde{\Lambda})$.  By decomposing with respect to  this basis the various classes of modules appearing in the right hand side of (\ref{dc}) we conclude that there exist integers such that 
$$p^m[M]=2b_0[k]+ \sum_{1 \leq i \leq r}2b_i[X_i]+\sum_{1 \leq j\leq s}d_j([Y_j]+[Y_j^D]).$$
But, since $p$ is odd,  this is impossible since  $M$ is a simple module.
\end{proof}
We can deduce from  Theorem \ref{nosimp}:

\begin{cor}\label{nonosimp} We assume the hypotheses of the Theorem. Suppose that  $\widetilde{\Lambda}_G^{ss}$ is a split semi-simple algebra,  
 then it  has no simple symplectic component. 
 \end{cor} 
\begin{proof}
Since the radical of $\widetilde{\Lambda}_G$ is stable under the antipode $S^D$, then $(\widetilde{\Lambda}_G^{ss}, S^D)$ is an algebra with involution over $k$. Moreover, since 
 $\widetilde{\Lambda}_G^{ss}$ is split,  we can write 
 $$ \widetilde{\Lambda}_G^{ss}=\prod_{1\leq i\leq d}\mathrm{M}_{n_i}(k).$$
 Suppose that there exists $i, 1\leq i\leq d$,  such that 
$\mathrm{M}_{n_i}(k)$ is a symplectic component of  $ \widetilde{\Lambda}_G^{ss}$. Then by \cite{KMRT} Proposition 2.6  we know that $n_i$ is even. Each simple factor of 
$ \widetilde{\Lambda}_G^{ss}$ is attached to a $\widetilde{\Lambda}_G^{ss}$-simple module,  unique up to isomorphism. Let $V_i$ be the simple $ \widetilde{\Lambda}_G^{ss}$-module attached to  $\mathrm{M}_{n_i}(k)$. We know that $n_i$ is equal to dimension of $V_i$ as a $k$-vector space (see \cite{R75} Theorem 7.4). Moreover, since $\mathrm{M}_{n_i}(k)$ is a simple component of $\widetilde{\Lambda}_G^{ss}$, stable under $S^D$, it follows from  \cite{LM} Lemma 2.1  that $V_i$ is isomorphic to $V_i^D$ as a $\widetilde{\Lambda}_G^{ss}$-module. Therefore we conclude that if  $\mathrm{M}_{n_i}(k)$ is a symplectic component of $ \widetilde{\Lambda}_G^{ss}$, then  $V_i$ is a simple and self-dual $\widetilde{\Lambda}_G^{ss}$-module of even dimension over $k$ and by  Theorem \ref{nosimp}  it is impossible. 
\end{proof}

\section {Unitary Determinants}

\subsection{Algebras with involution}
In this subsection we suppose that $K$ is a field of characteristic zero and
 we suppose  that $G$ is a finite group of odd order. Recall that the
involution which is $K$-linear and induced by
inversion on the group elements of $G$ is usually written as $x\mapsto 
\overline{x}$; however, for typographical reasons, sometimes we shall
denote the involution by $c$ and write $\overline{x}=c(x)$.

 The group algebra $K[ G] $ is a semi-simple $K$-algebra with Wedderburn decomposition as a product of simple $K$-algebras
whose simple components,   indexed by $\chi $,   are matrix rings
over division algebras $D_{\chi }$ with centers denoted $Z_{\chi }$
\begin{equation}
K[ G] =\prod_{\chi }M_{n_{\chi }}( D_{\chi
}) .
\end{equation}
 We now  recall from (29) Section 2-7  the decomposition of $K[ G] $ into $c$-stable simple algebras:
\begin{equation}\label{devis}K[G]=K\times\prod_{i\in I^{u}}A_{i}\times \prod_{j\in J}(A_{j}\times A^{\mathrm{op}}_{j})\end{equation}
where, for $j\in J$,  the involution on $( x,y) \in A_{j}\times
A^{\mathrm{op}}_{j}$ is $c (( x,y^{\mathrm{op}}))=( y,x^{\mathrm{op}})$,  where, for $i\in
I^{u},$ then $ c$ acts as a unitary involution on the simple algebra $A_{i}$ 
and so  induces a nontrivial automorphism of order $2$ on the center of $
A_{i}$. where the copy of $K$ on the right hand-side corresponds to the trivial
representation of $G$.
 
\begin{definition}The group of minus determinants of $K[ G] ^{\times } $ is defined as
\begin{equation}\label{det-}
\mathrm{Det}(K[ G]^\times ) _{-}=\{\mathrm{Det}(
z) \in \mathrm{Det}(K[G]^\times)\  \text{ s.t. }\ 
\mathrm{Det}( z\overline{z}) =1\text{ }\}.
\end{equation}
\end{definition}

 We recall from Section 2.7 that $U(K[G])=\{x\in K[G]^\times \mid\ x\bar x=1\}$. Given a $c$-stable subgroup $\Lambda$ of  $K[ G] ^{\times
}$, we set $U(\Lambda) =\Lambda \cap U( K[ G]
) $ and we set

\begin{equation*}
\mathrm{Det}( \Lambda^\times) _{-}=\mathrm{Det}( \Lambda^\times)
\cap \mathrm{Det}( K[ G]^\times) _{-}.
\end{equation*}
In many situations we shall need to know whether we have the equality 
\begin{equation*}
\mathrm{Det}( U( \Lambda)) =\mathrm{Det}(
\Lambda) _{-}\ .
\end{equation*}

The following result comes from Theorem 7 on page 104 of \cite{F84}:

\begin{thm}\label{fglob}
If $K$ has characteristic different from $2$ and if $G$ has odd order, then 
$$\mathrm{Det}(U( K[ G]) ) =\mathrm{Det}(K[G]^\times) _{-}.$$
\end{thm}
\subsection{Reduction of determinants}
In this subsection we adopt the notations of subsection 4.1. More precisely $R$ is a local $p$-adic ring integers with valuation
ideal $\mathfrak{p}$  and the residue field,  denoted by $k$,  has
characteristic $p$.  We consider a Hopf  $R$-order $\Lambda:=\Lambda_G $ in $K[G]$. We recall that $G_{0}( \Lambda)$ is the Grothendieck group of finitely
generated left $\Lambda $-modules and $G_{0}^{R}( \Lambda) $ is the
Grothendieck group of  left $\Lambda$-lattices. We write  $\widetilde{\Lambda }=\Lambda\  \mathrm{mod}\ \mathfrak{p\ }$ and we denote by $\widetilde{
\Lambda }^{ss}$  the semi-simplification of $\widetilde{\Lambda}$. 
 If $\mathcal{J}$ and   $\widetilde{\mathcal{J}}$ denote respectively the radicals of $\Lambda $ and  $\widetilde{\Lambda }$,  then
\begin{equation*}
\frac{\Lambda }{\mathcal{J}}=\frac{\widetilde{\Lambda }}{\widetilde{\mathcal{%
J}}}\cong \widetilde{\Lambda }^{ss}.
\end{equation*}
    The involution on $\Lambda$ given by the antipode induces an involution on the $k$-semisimple algebra $\widetilde{\Lambda }^{ss}$  that we denote by $c$. 
    \begin{thm}\label{ss} If $p>2$ and  $G$ has odd order  then 
\begin{equation*}
\mathrm{Det}( \widetilde{\Lambda }^{ss \times}) _{-}=\mathrm{Det}(
U( \widetilde{\Lambda }^{ss}) ) .
\end{equation*}
\end{thm}
\begin{proof}  We decompose   $\widetilde{\Lambda }^{ss}$ into a product of indecomposable simple algebras with involution (see 1.D in [7]): 

\begin{equation}\label{dss}
\widetilde{\Lambda }^{ss}= \prod_{i\in I^{u}}a_{i}\times \prod_{j\in J}(a_{j}\times
a_{j}^{op})\times \prod_{h\in I^{o}}a_{h}\times\prod_{f\in I^{s}}a_{f}
\end{equation} where $a_l=M_{n_l}(k_l)$ for $l\in I= J\cup I^u \cup I^o\cup I^s$. For any $l\in I^u\cup I^o\cup I^s$ (resp. $J$) we denote by $\sigma_l$ the restriction of $c$ to $a_l$ 
(resp. $a_j\times a_j^{op})$. For $l\in I^u$ (resp. $I^o$,  resp. $I^s$) then $\sigma_l$ is unitary (resp. orthogonal,  resp. symplectic) and for $j \in J$ the involution $\sigma_l$ is given by 
$(x, y^{op})\rightarrow (y,x^{op})$. We set $b_j:=a_j\times a_j^{op}$. 

We start by proving that $\widetilde{\Lambda}^{ss}$ has no indecomposable symplectic components. 

\begin{lem}If $p>2$ and $G$ has  of odd order then $I^s$ is empty. 
\end{lem}
\begin{proof}  For  a finite extension $K'/K$  with ring of integers $R'$ and residual field $k'$   we set $\Lambda'=\Lambda\otimes_RR'$; indeed $\Lambda'$ is an $R'$-Hopf order of $K'[G]$. We observe the equalities
$$\widetilde {\Lambda}'=(\Lambda\otimes_{O_K}O_{K'})\otimes_{O_{K'}}k' =\Lambda\otimes_{O_K}k'=(\Lambda\otimes_{O_K}k)\otimes_kk'=\widetilde \Lambda\otimes_kk'.$$
Since  $\widetilde {\Lambda}$ and $k'$ are finite $k$-algebras  we know that
$$\mathrm{rad}(\widetilde {\Lambda}')=\mathrm{rad}(\widetilde \Lambda\otimes_kk')=\mathrm{rad}(\widetilde \Lambda)\otimes_kk'+\widetilde \Lambda\otimes_k\mathrm{rad}(k')=\mathrm{rad}(\widetilde \Lambda)\otimes_kk'$$ and so 
\begin{equation}\label{decsymp}\widetilde {\Lambda}^{'ss}=\widetilde {\Lambda}'/\mathrm{rad}(\widetilde {\Lambda}')=(\widetilde {\Lambda}\otimes_kk')/(\mathrm{rad}(\widetilde {\Lambda})\otimes_k k')=\widetilde {\Lambda}^{ss}\otimes_kk'.\end{equation}
 We now choose $K'/K$ such that $K'[G]$ and $\widetilde {\Lambda}^{'ss}$ are both split semi-simple algebras respectively on $K'$ and $k'$. Our aim is to show   that $\widetilde \Lambda^{ss}$ has  no simple symplectic component. We suppose otherwise and we let  $A_1={M}_{n_1}(k_1)$ be such a component. Using  
\cite{KMRT} Proposition 2.19 we know that  the restriction of $c$ to $A_1$ is given by 
$$c(x)=u^txu^{-1},\ \forall x\in {M}_{n_1}(k_1)$$ where $u\in Gl_{n_1}(k_1)$ and  $^tu=-u$.  

It  follows from (\ref{decsymp})  that each simple factor of  $A_1\otimes_kk'$ is a simple factor of $\widetilde \Lambda^{'ss}$. 
We  have an isomorphism of $k'$-algebras 
$$\varphi: A_1\otimes_kk'={M}_{n_1}(k_1)\otimes_kk'\rightarrow \oplus_{\sigma\in S}{M}_{n_1}(k')$$ 
defined by $x\otimes\lambda \rightarrow \oplus_{\sigma\in S}y_{\sigma}$, where  $S=\{\sigma_1, \cdots, \sigma_m\}$ is the set of $k$-embeddings of $k_1$ into $k'$ and  
$y_{\sigma}=\sigma(x)\lambda$.  Therefore the simple factors of $A_1\otimes_kk'$ consist of $m=[k_1:k]$ copies of ${M}_{n_1}(k')$

The involution $c$ of $A_1$ extends to an involution $c_1$ of $ A_1\otimes_kk'$ with $c_1(x\otimes\lambda)=c(x)\otimes \lambda$. Therefore   
$$\varphi( c_1(x\otimes\lambda))=\varphi (c(x)\otimes\lambda)=\sum_{\sigma \in S}z_{\sigma}$$ with 
$$z_{\sigma}=\sigma(c(x))\lambda=\sigma(u)\sigma(^tx)\sigma(u)^{-1}\lambda=\sigma(u)^t\sigma(x)\sigma(u)^{-1}\lambda$$
$$=\sigma(u)^t(\sigma(x)\lambda)\sigma(u)^{-1}=\sigma(u)^t y_{\sigma}\sigma(u)^{-1}.$$
Therefore each simple factor ${M}_{n_1}(k')$ of $A_1\otimes_kk'$ is  $c_1$-stable and   the restriction  of $c_1$ to the $\sigma$-factor  is given by $c_1(y_{\sigma})=\sigma(u)y_{\sigma}\sigma(u)^{-1}$. Since $^t(\sigma(u))=\sigma(^t u)=-\sigma (u)$ and $p>2$, then this involution is symplectic and so each simple factor of  $A_1\otimes_kk'$ is symplectic. We conclude that $\widetilde{\Lambda}'^{ss}$ would have symplectic components which is excluded by Corollary \ref{nonosimp} and so $\widetilde{\Lambda}^{ss}$ has no  symplectic factor which is simple,  as required.
\end{proof}
We consider now the indecomposable  unitary or orthogonal components of $\widetilde \Lambda^{ss}$.  If $F$ is  a field, $A$  the matrix algebra $M_n(F)$ and $\sigma$ an involution of $A$, then    $\sigma$  induces by restriction an involution on $F$ that we denote by $\tau$ ; 
such a $\tau$ can be extended to an involution of $A$ by setting: 
$$\tau((a_{i, j}))=^t(\tau(a_{i, j})). $$
For the sake of simplicity we say that $\sigma$ is orthogonal  if the restriction of $\sigma$ to $F$  is trivial and unitary when the restriction of $\sigma$ to $F$ is  non trivial.   We set 
$$\mathrm{Det}(A^\times)_-=\{\mathrm{det}(X),  X\in A\  \mathrm{and}\ \mathrm{det}(X\sigma(X))=1\}$$
$$U(A)=\{X\in A\ \mid X\sigma(X)=I_n\}.$$ Our aim is to prove 
\begin{lem}\label{right} Assume that $\sigma$ is orthogonal or unitary, then one has 
$$\mathrm{Det}(U(A))=\mathrm{Det}(A^\times)_-.$$

\end{lem}
\begin{proof} The inclusion  $\mathrm{Det}(U(A))\subset \mathrm{Det}(A^\times)_-$ is clear.  
It follows from \cite{KMRT} Proposition 2.20 that for any involution $\sigma$ of $A$, such that the restriction of $\sigma$ to $F$ is $\tau$,  there exists $S\in A^{\times}$ such that 
$$\sigma (X)=S^{-1}\tau(X)S\ \  \forall X\in A.$$  First we assume that  $\sigma$ is   orthogonal and the characteristic of $F$ is $2$. If $X\in \mathrm{Det}(A^\times)_-$ then 
$$\mathrm{det}(X\sigma(X))=\mathrm{det}(X)\mathrm{det}(^tX)=\mathrm{det}(X)^2=1$$
and so $\mathrm{det}(X)=1$. We conclude that 
$$\mathrm{Det}(U(A))=\mathrm{Det}(A^\times)_-=1$$ in this case. 

We now assume that $\sigma$ is unitary or orthogonal with the characteristic of $F$ different from $ 2$. It follows once again from \cite{KMRT} Proposition 2.20 that $\tau(S)=S$. Therefore in order to complete the proof of Lemma \ref{right} it suffices to prove that 
$$ \mathrm{Det}(A^\times)_-\subset \mathrm{Det}(U(A)).$$
 Let $x\in \mathrm{Det}(A^\times)_-$ given by $x=\mathrm{det}(M)$ and $\mathrm{det}(M\sigma(M))=1$. Since $\tau(S)=S$ we know by \cite{F84} chap.III Lemma 3.4 p.97 that there exists $T\in GL_n(F)$ and $D=\mathrm{diag}(d_1, \cdots, d_n)$,  with $d_i \in F^\times$,  such that $S=\tau(T)DT$. We 
 let $\Delta: F^\times\rightarrow GL_n(F)$ be the standard embedding given by $t\rightarrow \Delta(t)=\mathrm{diag}(t, 1,\cdots, 1)$ and we  set $\Delta_T(t)=T^{-1}\Delta(t)T, \forall t \in F^\times$. By an easy computation we check that  

 \begin{equation}\label{1}T\sigma(X)T^{-1}=D^{-1}\tau(TXT^{-1})D\ \ \forall X\in GL_n(F)\end{equation}
 and so  
 \begin{equation}\label{2}T\sigma(\Delta_T(t))T^{-1}=D^{-1}\tau(\Delta(t))D=\Delta(d_1^{-1}\tau(t)d_1)=\Delta(\tau(t))\ \forall t\in F^\times\end{equation}
 where the first equality follows from (\ref{1}) and the second from a simple computation. We observe that (\ref{2}) can be written
 \begin{equation}\label{3}\sigma(\Delta_T(t))=\Delta_T(\tau(t))\ \forall t\in F^\times. \end{equation}
We set  $N:=\Delta_T(\mathrm{det}(M))$.   It follows from (\ref{3}) that 
 $$N\sigma(N)=\Delta_T(\mathrm{det}(M))\Delta_T(\tau(\mathrm{det}(M)))=\Delta_T(\mathrm{det}(M)\tau(\mathrm{det}(M)).$$
 Since we have 
 $$\tau(\mathrm{det}(M))=\mathrm{det}(\tau(M)=\mathrm{det}(\sigma(M))$$
 we conclude that $N\sigma(N)=\Delta_T(1)=1$ hence that $N\in U(A)$ and so  that 
 $$x=\mathrm{det}(M)=\mathrm{det}(N)\in \mathrm{Det}(U(A)).$$
 This completes the proof of the lemma.
 \end{proof}

We now return to the proof of Theorem \ref{ss}. 
The isomorphism of algebras (\ref{dss}) induces  isomorphisms of groups  
\begin{equation}
\mathrm{Det}(\widetilde{\Lambda}^{ss})=\prod_{i\in I^u}k_i^{\times}\prod_{j\in J}(k_j^{\times}\times k_j^{\times})\prod_{h\in I^o}k_h^{\times}
\end{equation}
\begin{equation}
\mathrm{Det}(\widetilde{\Lambda}^{ss})_{-}=\prod_{i\in I^{u}}\mathrm{Det}(a_i^{\times})_{-}\times \prod_{j\in J}\mathrm{Det}(b_j^{\times})_{-}\times \prod_{h\in I^{o}}\mathrm{Det}(k_h^{\times})_{-} 
\end{equation}
and 
\begin{equation}\mathrm{Det}(U(\widetilde{\Lambda}^{ss}))=\prod_{i\in I^{u}}\mathrm{Det}(U(a_i))\times \prod_{j\in J}\mathrm{Det}(U(b_j))\times \prod_{h\in I^{o}}\mathrm{Det}U((k_h)) 
\end{equation}

One easily checks that:
$$\mathrm{Det}(U(b_j))=\mathrm{Det}( b_{j}^{\times })_-=\{(x, x^{-1}), x\in k_j^{\times}\}. $$
Moreover,  the equalities  $$\mathrm{Det}(U(a_i))=\mathrm{Det}( a_{i}^{\times})_-=\{x\in k_i^{\times}\ \mid x\tau_i(x)=1\}$$ and 
$$\mathrm{Det}(U(a_h))=\mathrm{Det}( a_{h}^{\times})_-=\pm 1$$
follow from Lemma \ref{right}. 
\end{proof}

\subsection{The structure of unitary determinants. }
As above, $K$ is a $p$-adic field, $R$ is its ring of integers and $G$ is a finite group of odd order. We consider a Hopf $R$-order $A^D$ of $K[G]$ and we let $\mathcal {J}$ denote its radical. The main goal of this section is to show:
\begin{thm}\label{fund}
For an arbitrary prime $p$ suppose that $R$  is a $p$-adic ring of
integers and for  $G$ of odd order,  then we have
\begin{equation*}
\mathrm{Det}( U( A^{D}) =\mathrm{Det}(
A^{D \times}) _{-}.
\end{equation*}

\end{thm}

The key to the proof of this theorem lies in the following three exact
sequences: 

\begin{thm}\label{cle}
For $p\neq 2$, the following sequences are exact :
\begin{eqnarray}\label{exact}
1 &\rightarrow &\mathrm{Det}( 1+\mathcal{J}) \rightarrow \mathrm{%
Det}( A^{D \times}) \rightarrow \mathrm{Det}( \widetilde{A}^{D 
ss \times }) \rightarrow 1 \\
1 &\rightarrow &\mathrm{Det}( 1+\mathcal{J}) _{-}\rightarrow 
\mathrm{Det}( A^{D \times}) _{-}\rightarrow \mathrm{Det}( 
\widetilde{A}^{D  ss \times}) _{-}\rightarrow 1  \notag \\
1 &\rightarrow &\mathrm{Det}( U( 1+\mathcal{J}))
\rightarrow \mathrm{Det}( U( A^{D}) ) \rightarrow 
\mathrm{Det}( U( \widetilde{A}^{D ss})) \rightarrow
1.  \notag
\end{eqnarray}
\end{thm}
We proceed in three steps. Firstly we show Theorem \ref{fund} when $p=2.$ Next we
show Theorem \ref{cle} and then finally we  prove Theorem \ref{fund} for $p>2$.
\linebreak

\textbf{Step 1.} So first we suppose that $R$ has residue characteristic
$2$. Then, since $G$ has odd order, $A^{D}=R[G]$ is a maximal $R$-order in $K[ G]$ and is isomorphic to a direct sum of matrix rings over 
local rings of integers which are non-ramified over $R$. The required equality in this case is given by \cite {T89} Theorem 1 (b)  (see the proof of this theorem under our hypotheses in \cite {T89} Proof of (3.3.3) 
Step 2).

\textbf{Step 2.}\ We now suppose that $p\neq 2$ and we consider the filtration
\begin{equation*}
...1+\mathcal{J}^{n}\subset\cdots \subset 1+\mathcal{J}^{2}\subset 1+\mathcal{J}
\end{equation*}
and for each $n>0$ we note that $1+\mathcal{J}^{n}/1+\mathcal{J}^{n+1}$ is
abelian and finite of odd order. 

For a multiplicative group $\mathcal{H}$, endowed with an involution $c$ which fixes the neutral element,  we define the set
$$\mathcal{H}_+=\{x \in \mathcal{H}\  \mid  \ x=c(x)\}.$$ We note that if $\mathcal{H}$ is commutative then $\mathcal{H}_+$ is a subgroup of $\mathcal{H}$. 
\begin{prop}\label{C1} The following equalities hold: 
\begin{enumerate}
\item  $( 1+\mathcal{J}) _{+}=\{( 1+j) ( 1+
\overline{j}) \ \mid j\in \mathcal{J}\}$. 
\item $SL(1+\mathcal{J})_+=\{(1+j)(1+\bar j)\ \mid\ 1+j\in SL(1+\mathcal J) \}$.
\item $\mathrm{Det}( 1+\mathcal{J}) _{+}=\{\mathrm{Det}((
1+j)( 1+\overline{j})) \ \mid j\in \mathcal{J}\}$.

\end{enumerate}
\end{prop}
\begin{proof} (1)  Consider $1+j\in ( 1+\mathcal{J}) _{+}$. We show that
for each $n>0$ we can write 
\begin{equation*}
1+j=x_{1}...x_{n}.x_{n}...x_{1}\  \mathrm{mod}\ (1+\mathcal{J}^{n+1})
\end{equation*}
with $x_{i}\in ( 1+\mathcal{J}^{i}) _{+}$. The result then
follows on taking limits. To show that we can write $1+j$  in this manner
we\ proceed by induction.

Since $p>2$, for each $n>0$ we know that the group $[1+\mathcal{J}^n/1+\mathcal{J}^{n+1}]_+$  is an abelian group 
 of odd exponent and so each element in this group is a square of some power of itself. So when $
n=1$ we can write $$1+j \equiv ((1+j)^m)^2=x_{1}^{2}\ \mathrm{mod}\ (1+\mathcal{J}^2) $$ with $x_{1}\in ( 1+
\mathcal{J})_{+}$.

 Inductively we assume that we may write 
\begin{equation*}
1+j=x_{1}\cdots x_{n-1}.x_{n-1}\cdots x_{1}\ \mathrm{mod}\ (1+\mathcal{J}^{n})
\end{equation*}
with each $x_{i}\in ( 1+\mathcal{J}^{i}) _{+}.$ Then,  since  $[1+\mathcal{J}^n/1+\mathcal{J}^{n+1}]_+$  is an abelian group 
 of odd exponent, we can write as above:
\begin{equation*}
( x_{1}...x_{n-1}) ^{-1}( 1+j) (
x_{n-1}..x_{1}) ^{-1}\equiv {x}_{n}^{2}\ \mathrm{mod} ( 1+\mathcal{J}^{n+1})
\end{equation*}
with  ${x}_{n}$ in $
( 1+\mathcal{J}^{n}) _{+}$. We therefore have   
\begin{equation*}
1+j=x_{1}\cdots x_{n-1}x_{n}.x_{n}x_{n-1}\cdots x_{1}\ \mathrm{mod}\  (1+\mathcal{J}^{n+1})
\end{equation*}
with   $x_i\in (1+\mathcal{J}^i)_+$ for $1\leq i\leq n$ which completes the inductive step. 
\vskip 0.1 truecm
(2) We start by  observing  that $[SL(1+\mathcal{J}^{n})/SL(1+\mathcal {J}^{n+1})]_+$ is a subgroup of the group  $[1+\mathcal{J}^n/1+\mathcal{J}^{n+1}]_+$ for each  integer $n>0$. Therefore the proof of (2) 
is similar to the proof of (1) when replacing in each step $(1+\mathcal{J}^{n})$ by $SL(1+\mathcal{J}^{n})$. 
 
\vskip 0.1 truecm

(3) For any $n\geq 1$  the group $\mathrm{Det}(1+\mathcal{J}^n)/\mathrm{Det}(1+\mathcal{J}^{n+1})$ is abelian and of odd exponent. 
Moreover,  one  easily checks that 

$$\mathrm{Det}(1+\mathcal{J}^n)_+\cap \mathrm{Det}(1+\mathcal{J}^{n+1})=\mathrm{Det}(1+\mathcal{J}^{n+1})_+.$$  It follows that   
 $\mathrm{Det}(1+\mathcal{J}^n)_+/\mathrm{Det}(1+\mathcal{J}^{n+1})_+$ is a subgroup of $\mathrm{Det}(1+\mathcal{J}^n)/\mathrm{Det}(1+\mathcal{J}^{n+1})$ for any integer $n\geq 1$, and so    is abelian 
  of odd exponent.   Thus  each element of this group is the square of some power of itself. 
 
 Let $u$ be an element of $\mathrm{Det}(1+\mathcal{J})_+$, then there exists $x_1\in 1+\mathcal{J}$ and $z_1\in 1+\mathcal{J}^2$ with $\mathrm{Det}(x_1)\in \mathrm{Det}(1+\mathcal{J})_+$ and 
 $\mathrm{Det}(z_1)\in \mathrm{Det}(1+\mathcal{J}^2)_+$ such that 
  $$u=\mathrm{Det}(x_1)^2\mathrm{Det}(z_1)=\mathrm{Det}(x_1 \bar x_1)\mathrm{Det}(z_1).$$
  Inductively we assume that there exist $\{x_i\}_{1\leq i\leq n-1}$ and  $\{z_i\}_{1\leq i\leq n-1}$ with  $ x_i \in 1+\mathcal{J}, x_i-x_{i-1}\in \mathcal{J}^i, z_i\in 1+\mathcal{J}^{i+1},  \mathrm{Det}(z_i) \in \mathrm{Det}(1+\mathcal{J}^{i+1})_+$ such that 
  $$u=\mathrm{Det}(x_{i} \bar x_{i})\mathrm{Det}(z_i)),\ \forall i\leq n-1 .$$ We want now to construct $x_n$ and $z_n$. Since  $\mathrm{Det}(z_{n-1})\in  \mathrm{Det}(1+\mathcal{J}^{n})_+$,  
  then there exists $y_n \in 1+\mathcal{J}^n$ 
  and $z_n\in 1+\mathcal{J}^{n+1}$ with $\mathrm{Det}(y_n)\in \mathrm{Det}(1+\mathcal{J}^n)_+$,  
 $\mathrm{Det}(z_n)\in \mathrm{Det}(1+\mathcal{J}^{n+1})_+$ and such that 
  $$\mathrm{Det}(z_{n-1})=\mathrm{Det}(y_{n})^2\mathrm{Det}(z_n).$$
  We set $x_n=x_{n-1}y_n$, thus $x_n-x_{n-1}\in \mathcal{J}^n$. Moreover,  since $\mathrm{Det}(y_n)\in \mathrm{Det}(1+\mathcal{J}^n)_+$, then $\mathrm{Det}(y_n)^2=\mathrm{Det}(y_n)\mathrm{Det}(\bar y_n)$. 
  Therefore we obtain 
  $$u=\mathrm{Det}(x_{n-1} \bar x_{n-1})\mathrm{Det}(z_{n-1}))=\mathrm{Det}(x_{n-1} \bar x_{n-1})\mathrm{Det}(y_{n})^2\mathrm{Det}(z_n)=$$
  $$\mathrm{Det}(x_{n-1} \bar x_{n-1})\mathrm{Det}(y_n)\mathrm{Det}(\bar y_n)\mathrm{Det}(z_n)=\mathrm{Det}(x_{n} \bar x_{n})\mathrm{Det}(z_n)$$
  and the result follows on passing to the limit.
  \end{proof}

\begin{cor}\label{C2}
Given any element $\mathrm{Det}( z) \in \mathrm{Det}( 1+
\mathcal{J}) $ with $\mathrm{Det}( z.\overline{z}) =1$, 
we can find \ $\widetilde{z}\in U( 1+\mathcal{J}) $ with $
\mathrm{Det}( \widetilde{z}) =\mathrm{Det}( z)$; 
that is to say
\begin{equation*}
\mathrm{Det}( U( 1+\mathcal{J})) =\mathrm{Det}(
1+\mathcal{J}) _{-}.
\end{equation*}
\end{cor}
\begin{proof} \ Consider $\mathrm{Det}( z) \in \mathrm{Det}( 1+
\mathcal{J})$ with $ \mathrm{Det}( z.\overline{z}) =1 $. Then $z.\overline{z}\in SL(1+\mathcal{J})_+$ and by 
 (2) above we can write $z.\overline{z}=s.\overline{s}$ with $s\in SL(1+\mathcal{J})$. We therefore have 
 $\mathrm{Det}(z)=\mathrm{Det}(s^{-1}z)$ and 
 $$s^{-1}z.\overline{s^{-1}z}=s^{-1}z.\overline{z}.\overline{s^{-1}}=s^{-1} z.\overline{z}.\overline{s}^{-1}=1.$$
 
\end{proof}

We are now in a position to show that the three rows in Theorem \ref{cle}  are
exact.

We first show that the top row in (\ref{exact}) is exact. To this end we consider the
commutative diagram where by Lemma  \ref{SLS} we know that the top row is also
exact and where  we know that the bottom row is exact (see Section 4.2 ): 
\begin{equation*}
\begin{array}{ccccccccc}
1 & \rightarrow & SL( 1+\mathcal{J}) & \rightarrow & SL(
A^{D}) & \rightarrow & SL( \widetilde{A}^{D ss}) & 
\rightarrow & 1 \\ 
&  & \downarrow &  & \downarrow &  & \downarrow &  &  \\ 
1 & \rightarrow & 1+\mathcal{J} & \rightarrow & A^{D\times } & \rightarrow & 
\widetilde{A}^{D ss^{\times }} & \rightarrow & 1.
\end{array}
\end{equation*}
Using the Snake Lemma we obtain an exact sequence of cokernels, and this
gives the desired exact sequence.

Next we show the exactness of the middle row:
\begin{equation*}
1\rightarrow \mathrm{Det}( 1+\mathcal{J}) _{-}\overset{a}{
\rightarrow }\mathrm{Det}( A^{D \times}) _{-}\overset{b}{\rightarrow }
\mathrm{Det}( \widetilde{A}^{D ss \times}) _{-}\rightarrow 1.
\end{equation*}
The map $a$ is an inclusion and hence is injective. Next we show that $b $
is surjective: given $\mathrm{Det}( d) \in \mathrm{Det}( 
\widetilde{A}^{D ss \times}) _{-},$ by Theorem \ref{ss} we know that   $\mathrm{Det}( 
\widetilde{A}^{D ss \times}) _{-}=\mathrm{Det}( U( \widetilde{A}
^{D ss}))$,  and so we may assume that in fact $d\in U( 
\widetilde{A}^{D ss})$.  Choose a lift $\ d^{\prime }\in A^{D \times}$ of $d$.  Then $d^{\prime }.\overline{d^{\prime }}\in ( 1+\mathcal{J}
) _{+}$ and by Proposition \ref{C1} we can write $d^{\prime }.\overline{
d^{\prime }}=( 1+j) \overline{( 1+j) }$ and therefore $
( 1+j) ^{-1}$\bigskip $d^{\prime }.\overline{( 1+j)
^{-1}\bigskip d^{\prime }}=1$ and thus in fact  
\begin{equation*}
\mathrm{Det}(( 1+j) ^{-1}\bigskip d^{\prime }) \in 
\mathrm{Det}( U( A^{D})) \subset \mathrm{Det}(
A^{D \times}) _{-}
\end{equation*}
and is a lift of $\mathrm{Det}(d)$. 

To conclude this part of the proof we must show that $\mathrm{ker}( b)\subset \mathrm{Im}
 ( a)$. Indeed, given $\mathrm{Det}( d) \in \mathrm{ker} (b)$,  then by the
exactness of the top row in (\ref{exact})
\begin{equation*}
\mathrm{Det}( d) \in \mathrm{Det}( 1+\mathcal{J}) \cap 
\mathrm{Det}( A^{D \times}) _{-}\subset \mathrm{Det}( 1+\mathcal{J}) _{-}
\end{equation*}
as required.

To conclude we establish the exactness of the bottom row of (\ref{exact}). For this
we use the exactness of the middle row of (\ref{exact}). To this end we consider the
diagram 
\begin{equation*}
\begin{array}{ccccccccc}
1 & \rightarrow & \mathrm{Det}( U( 1+\mathcal{J})) & 
\overset{\alpha }{\rightarrow } & \mathrm{Det}( U( A^{D})
) & \overset{\beta }{\rightarrow } & \mathrm{Det}( U( 
\widetilde{A}^{D ss}) ) & \rightarrow & 1 \\ 
&  & \downarrow &  & \downarrow &  & \downarrow &  &  \\ 
1 & \rightarrow & \mathrm{Det}( 1+\mathcal{J}) _{\_} & \rightarrow
& \mathrm{Det}( A^{D \times}) _{-} & \rightarrow & \mathrm{Det}( 
\widetilde{A}^{D ss \times}) _{-} & \rightarrow & 1
\end{array}
\end{equation*}
where all three vertical maps are inclusions. The map $\alpha $ is also an
inclusion and therefore injective. Moreover the map  $\beta $ is surjective, since as seen above,  every element of $ \mathrm{Det}( U( 
\widetilde{A}^{D ss}) ) $ has a lift in $\mathrm{Det}(U(A^D))$. To conclude, we show that $\mathrm{ker}( \beta) \subset \mathrm{Im}(\alpha) $. Indeed,
given $d\in \mathrm{ker }(\beta) $, we know that $d\in \mathrm{Det}( A^{D \times})
_{-}$ has trivial image in $\mathrm{Det}( \widetilde{A}^{D ss \times})
_{-}$ and so by the exactness of the lower row $d$ lies in $\mathrm{Det}
( 1+\mathcal{J}) _{\_}$ and by Corollary \ref{C2},  $\mathrm{Det}( 1+
\mathcal{J}) _{-}=\mathrm{Det}( U( 1+\mathcal{J})) $.

\medskip

\textbf{Step 3. } We now prove Theorem \ref{fund} using Theorem \ref{cle} for $p>2$. To this end we
again consider the commutative diagram with exact rows:
\begin{equation*}
\begin{array}{ccccccccc}
1 & \rightarrow & \mathrm{Det}( U( 1+\mathcal{J})) & 
\rightarrow & \mathrm{Det}( U( A^{D})) & \rightarrow
&  & \mathrm{Det}( U( \widetilde{A}^{D ss \times }))
\rightarrow & 1 \\ 
&  & \downarrow &  & \downarrow &  &  & \downarrow &  \\ 
1 & \rightarrow & \mathrm{Det}( 1+\mathcal{J}) _{-} & \rightarrow
& \mathrm{Det}( A^{D \times}) _{-} & \rightarrow &  & \rightarrow 
\mathrm{Det}( \widetilde{A}^{D ss \times}) _{-} & 1.
\end{array}
\end{equation*}
By Corollary \ref{C2} the left hand downward vertical arrow is an equality and by
Theorem \ref{ss} the right hand downward vertical arrow is an equality; hence the
central downward vertical arrow is also an equality, as required. 
\subsection{ The group $\mathrm{ker}(\xi)$}

Again in this subsection we assume $G$ to have odd order. 

We recall from (24) that 
\begin{equation}
\mathrm{Cl}( A^{D}) =\frac{\mathrm{Det}( \mathbb{A}_{K}[
G] ^{\times }) }{\mathrm{Det}( K[ G] ^{\times
}) \cdot \prod_{\mathfrak{p}}\mathrm{Det}( A_{\mathfrak{p}
}^{D\times }) }
\end{equation}
and from (26) we have 
\begin{equation*}
\mathrm{CU}( A^{D}) =\frac{\mathrm{Det}( U(\mathbb {A}_{K}[G])
 }{\mathrm{Det}(
U( K[G] ) .\prod_{\mathfrak{p}}%
\mathrm{Det}( U( A_{\mathfrak{p}}^{D}) ) }.
\end{equation*}%
It therefore follows that the kernel of the natural map $\xi :$ $\mathrm{
CU}\left( A^{D}\right) \rightarrow \mathrm{Cl}\left( A^{D}\right) $,  denoted 
$\mathrm{ker}( \xi ) $,   is a subgroup of:
\begin{equation}
\Omega ( G): =\frac{[\mathrm{Det}( K[ G]^\times)
.\prod_{\mathfrak{p}}\mathrm{Det}( A_{\mathfrak{p}}^{D \times})
]_{-}}{\mathrm{Det}( U(K[ G]))\prod_{
\mathfrak{p}}\mathrm{Det}( U(A_{\mathfrak{p}}^{D})) }.
\end{equation}

However, from Theorems \ref{fglob} and \ref {fund}  we know that 
\begin{equation*}
\mathrm{Det}( U( K[ G])) \prod_{
\mathfrak{p}}\mathrm{Det}( U( A_{\mathfrak{p}}^{D})) =
\mathrm{Det}( K[ G]^\times) _{-}\prod_{\mathfrak{p}}
\mathrm{Det}( A_{\mathfrak{p}}^{D \times}) _{-}
\end{equation*}
and so we may write 
\begin{equation}
\Omega ( G) =\frac{[\mathrm{Det}( K[ G]^\times)
.\prod_{\mathfrak{p}}\mathrm{Det}( A_{\mathfrak{p}}^{D \times})
]_{-}}{\mathrm{Det}( K[ G]^\times) _{-}\prod_{%
\mathfrak{p}}\mathrm{Det}( A_{\mathfrak{p}}^{D \times}) _{-}}.
\end{equation}

\begin{prop}\label{exp}
\bigskip The group $\Omega \left( G\right) ,\ $and hence the subgroup$\ \ker
\xi$, is an elementary 2-group.
\end{prop}
\begin{proof} Consider 
\begin{equation*}
x\in \mathrm{Det}( a) \mathrm{Det}( u) \in [ 
\mathrm{Det}( K[ G]^\times) .\mathrm{Det}(
\prod_{\mathfrak{p}}A_{\mathfrak{p}}^{D \times})] _{-}
\end{equation*}
with $\mathrm{Det}( a) \in \mathrm{Det}( K[ G]^\times
) $ and $\mathrm{Det}( u) \in \mathrm{Det}(
\prod_{\mathfrak{p}}A_{\mathfrak{p}}^{D \times})$.  Then for any
character $\chi $ of $G$ we have 
\begin{equation*}
1=\mathrm{Det}( au)( \chi +\overline{\chi }) =\mathrm{
Det}( au)( \chi) \mathrm{Det}( au)( 
\overline{\chi }) =\mathrm{Det}(a.u.\overline{a}.\overline{u})(
\chi) .
\end{equation*}
We may therefore deduce that $\mathrm{Det}(a.u)=\mathrm{Det}(\overline{a}%
^{-1}.\overline{u}^{-1})$,  so that 
\begin{equation*}
x^{2}=\mathrm{Det}(a.u)^{2}=\mathrm{Det}(a.u).\mathrm{Det}(\overline{a}^{-1}.
\overline{u}^{-1})=\mathrm{Det}(a.\overline{a}^{-1}).\mathrm{Det}(u.
\overline{u}^{-1})
\end{equation*}
and so 
\begin{equation*}
x^{2}\in \mathrm{Det}( K[ G]^\times) _{-}\mathrm{Det}(
\prod_{\mathfrak{p}}A_{\mathfrak{p}}^{D \times}) _{-}.
\end{equation*}

\end{proof}

\section {Proof of Theorems}
We recall the diagram introduced in Section 2.7 
\begin{equation}\label{didi}
\begin{array}{ccc}
\mathrm{PH}^{\prime }(A) & \overset{\phi }{\rightarrow } & 
\mathrm{CU}(A^{D}) \\ 
& \searrow \psi & \downarrow \xi  \\ 
&  & \mathrm{Cl}(A^{D}).
\end{array}
\end{equation}
This section contains the  proofs of our main theorems. We begin by describing two key tools that we shall use for these proofs.

\subsection{Kneser Strong Approximation}.

\vskip 0.1 truecm  
 We let $K$ be a number field, $R$ its ring of integers and we denote by $\mathbb A_K$ the ring of  finite ad\`eles.    For an algebraic group $\mathcal{H}$ over $K$ we denote by $\mathcal{H}(\mathbb A_K)$ the group of its finite ad\`elic points,  endowed with the usual ad\`elic topology. More precisely:

$$\mathcal{H}(\mathbb A_K)=\{ h=(h_{\frak{p}}) \in \prod_{\frak{p} }\mathcal{H}(K_\frak{p})\  \vert \ \mathrm{for\ almost\ all}\   h_{\frak{p}}\  \in \mathcal{H}(R_{\frak{p}}) \}$$ where $\mathfrak{p}$ goes through the set of maximal ideals of $R$. We recall that  the open subgroups of $\prod_{\frak{p}}\mathcal{H}(R_{\frak{p}})$ are taken as  a fundamental system of neighborhoods of the identity. 

\begin{definition} The algebraic group $\mathcal{H}$ is said to have  Kneser's Strong Approximation if $\mathcal{H}(K)$ is dense in $\mathcal{H} (\mathbb A_K)$
\end{definition}
In \cite{K1}, Main Theorem, Kneser proved: 
\begin{thm} Suppose that $\mathcal{H}$ is a connected, absolutely almost simple algebraic group over $K$. Then $\mathcal{H} $ has Strong Approximation if and only if the following two properties hold:
\begin{enumerate}
\item $\mathcal{H}$ is simply connected.
\item $\prod_{v\in S}\mathcal{H}(K_v)$ is not compact, where $S$ is the set of infinite primes of $K$.
\end{enumerate}  
\end{thm}  
We now come back to the notation of Section 2.7. If $G$ is a finite group of odd order, we recall that for any $m\geq 1$ we have considered  the unitary $R$-group scheme $U_{m, A^D}$  and denoted by $U_{m, G}$ its generic fiber. This group can be identified  with  the group of automorphisms of the $G$-form $(A, Tr'_A)^{\perp m}$.   In Section 2.7 (30), we have introduced,  for $i\in I$ and $j\in J$,  the $K$-algebras with involution $(A_{m, i}, \sigma_i)$ and $(B_{m,j}, \sigma_j)$. We can attach to these algebras their unitary group  schemes over $K$ that we denote by $U_{A_{m, i}, K}$ and $U_{B_{n,j}, K }$. For $i\in I$ and $j\in J$  we denote   by $E_i$ and $E_j$ the subfields of elements of  the centers of    $A_{m,i}$    and $B_{m,j}$,  stables respectivement by $\sigma_i$ and $\sigma_j$. Indeed, for every $i$ and $j$, then  $(A_{m,i}, \sigma_i)$    and $(B_{m,j}, \sigma_j)$  are algebras with involution on $E_i$ and $E_j$ and so we can consider the unitary group schemes   $U_{A_{m, i}, E_i}$ and  $U_{B_{m,j}, E_j}$     over $E_i$ and $E_j$. 
We have the equalities 
$$U_{A_{m, i}, K}=\mathrm{R}_{E_i/K}( U_{A_{m, i}, E_i})\ \ \mathrm{and} \ \  U_{B_{m,j}, K }=  \mathrm{R}_{E_j/K}( U_{B_{m, j}, E_j})     $$
where $\mathrm{R}$ is the Weil's restriction. 

The decomposition of the algebra with involution $(M_m(K[G]), \sigma) $ into a product of indecomposable algebras  (\ref {c2})  yields  the decomposition of $U_{m, G}$ into a product of algebraic groups:
\begin{equation}\label{decu}U_{m,G}=O_{m, K}\times \prod_{i\in I}U_{A_{m,i}}\times  \prod_{j\in J}U_{B_{m,j}}\end{equation}
$$=O_{m, K}\times \prod_{i\in I}\mathrm{Res}_{E_{i/K}}(U_{A_{m,i}, E_i})\times  \prod_{j\in J}\mathrm{Res}_{E_j/K}(U_{B_{m,j}, E_j}).$$
We denote by    $SU_{m, G}$  the subgroup of $U_{m, G}$ defined as the kernel of the morphism of algebraic groups induced by the reduced norm. It may  be described as the product 
\begin{equation}\label{decdec}SU_{m,G}:=SO_{m,K}\times \prod_{i\in I}\mathrm{Res}_{E_{i/K}}(SU_{A_{m,  i}, E_i})\times \prod_{j\in J}\mathrm{Res}_{E_j/K}(SU_{B_{m,j}, E_j})\end{equation}
where for $i\in I$ and $j\in J$ the groups  $SU_{A_{m,i}, E_i}$ and $SU_{B_{m,j}, E_j}$ are respectively defined as  the kernel of the morphism $n_{rd}: U_{A_{m,i}, E_i}\rightarrow U_{F_i, E_i}$ 
and $U_{B_{m,j}, E_j}\rightarrow U_{B_j, E_j}$ induced by the reduced norm.  

We define   $U^{\varepsilon}_{m, G}$ as  the kernel of  the morphism of algebraic groups $\varepsilon : U_{m, G}\rightarrow O_{m, K}$ induced by projection to the trivial character of $G$ and we denote by $SU^{\varepsilon}_{m, G}$ the kernel of the restriction of this morphism to  $SU_{m, G}$. Indeed  
$SU^{\varepsilon}_{m, G}=SU_{m, G}$ when $m=1$
\begin{thm} \label{Kimp}Suppose that $G$ is a group of odd order and that $K$ is a non-totally real number field. Then the algebraic group  $SU^{\varepsilon}_{m, G}$ has Kneser's Strong Approximation. 

\end{thm} 
\begin{proof} 
We start by proving a proposition.

 Let $E$ be a number field and  let $(C, \sigma)$ be a finite dimensional, non commutative  $E$-algebra,  endowed with an  involution of the second kind  whose center $F$ is a quadratic \'etale extension of $E$  (see  \cite{ KMRT} Section 2.B). More precisely we assume that either $C$ is a simple algebra and   $F$ is a field or a direct product of two simple algebras and   $F=E\times E$.    We let $SU_{C, E}$ be the kernel of the morphism of algebraic groups over $E$ induced by the reduced norm $n_{red}: U_{C,E}\rightarrow U_{F,E}$. 

\begin{prop}\label{K1} We assume that one of the following conditions holds:
\begin{enumerate}
\item $C$ is the product of two simple algebras interchanged by $\sigma$ which satisfy Eichler's condition. 
\item $C$ is a simple algebra and $E$ is a non-totally real number field. 
\end{enumerate} Then  $SU_{C,E}$  is an algebraic group over $E$ which has  Kneser's Strong Approximation. 
\end{prop} 
\begin{proof} We write $U_C$ for $U_{C, E}$ and  $SU_{C}$ for $SU_{C, E}$. First we assume that $C$ is a product of two simple algebras interchanged by $\sigma$. Under a slight abuse of notation we write $C=D\times D^{op}$ and we let $\sigma$ be defined by $(x, y^{op})=(y, x^{op})$. For any commutative $E$-algebra  $T$, then the morphism of groups induced by the reduced norm $U_{C}(T)\rightarrow U_{F}(T)$ is given by $(x, x^{-1})\rightarrow (n_{red}(x), n_{red}(x)^{-1})$. Therefore the group $SU_{C}$ is isomorphic to the reduced norm one subgroup  defined  for any $T$ as the kernel of $(D\otimes_ET)^\times\rightarrow (F\otimes_ET)^\times$. Since $D/E$ satisfies Eichler's condition,  then the norm one subgroup satisfies 
Kneser's Strong Approximation (see \cite {CFY} Theorem 12 for instance). 

We now suppose that $C$ is a simple $E$-algebra.  Since  $\mathrm{dim}_{E }(C)>4$, it follows from  \cite{BP} Section 1.2 that $SU_{C}$ is a semisimple and simply connected algebraic group over $E$. Therefore, it follows from  \cite{K1} Main Theorem that to complete the proof of  Proposition \ref {K1}      it remains to check that  $\prod_{v}SU_{C} (E_{v}) $ is non compact, when $v$ goes through the set  of non archimedean places  of $E$. We know that   $E$ has at least 
 one complex prime $v$.  Our aim is to prove that the group $SU_{C} (E_{v})$ is not compact for such a prime. The canonical embedding $E\hookrightarrow E_{v}$ factorizes through
  an embedding $F\hookrightarrow  E_{ v}$. For the sake of simplicity we write $E\subset F \subset E_{ v}$. From \cite{KMRT} Proposition 2.15 we know that there exists an isomorphism of $F$-algebras  with involution 
 $$(C, \sigma)\otimes_{E}F\simeq (C\times C ,\eta)$$
 where $\eta$ is defined by $\eta (x, y^{op})=(y, x^{op})$. Therefore we obtain the following isomorphisms of  $E_{ v}$-algebras with involution: 
 $$(C,  \sigma)\otimes_{E}E_{v}\simeq((C, \sigma)\otimes_{E}F)\otimes_{F}E_{v}\simeq(C\otimes_{F} E_{ v}\times C\otimes_{F} E_{ v}, \eta).$$
 We are now back  to the previous case and thus, as seen before,  we have an isomorphism of topological groups 
 $$SU_{C} (E_{v})\simeq SGL^1_C(E_{ v})$$
 where $SGL^1_C(E_{v})$ is the kernel of the homomorphism of groups 
 $$(C\otimes_{F}E_{ v})^\times\rightarrow E_{ v}^\times$$ induced by the reduced norm. This group will be equal to $SL_m(\mathbb C)$ for some integer $m>1$ and  so is not compact. 
\end{proof}
We can now complete the proof of Theorem 6.3.  Using  the decomposition (\ref{decdec}),  it follows from \cite{PR} Proposition 7.1  that $SU^{\varepsilon}_{m, G}$ has Kneser   Strong Approximation since by Proposition \ref{K1}  we know that for  any $i\in I$ and $j\in J$ the groups  $SU_{A_{m,  i, E_i}}$ and $SU_{B_{m,  j, E_j}}$  have this property .
\end{proof}

\subsection {A double coset description of isometry classes of $G$-forms}
We  let $\mathrm{LI}((A, Tr_A')^{\perp m})$ be the set of isometry classes of $G$-forms over $R$ which are locally isometric to $(A, Tr')^{\perp m}$ at all prime ideals of $R$ (this includes the generic prime $0$). 

We set \begin{equation}c(R, U_{m, A^D})= U_{m, A^D}(K)\backslash U_{m, A^D}(\mathbb A_K)/ \prod_{\mathfrak{p}}U_{m, A^D}(R_{\mathfrak{p}})
\end{equation}
 $$=U(M_m(K[G])\backslash U(M_m(\mathbb A_K[G])/ \prod_{\mathfrak{p}}U(M_m(A^D_{\frak{p}})).$$

 If  $(M, q)$ is such a form,  for each  prime ideal of $R$,   we fix isomorphisms of $G$-forms  
\begin{equation*}
\varphi _{\mathfrak{p}}:(A_{\frak{p}}, Tr'_{A_{\frak{p}}})^{\perp m}\simeq ( M_{\mathfrak{p}}, q_{\mathfrak{p}}) \ \mathrm{and} \ \varphi_{0}: (A_K, Tr_{A_K}')^{\perp m}\simeq (M_K, q_K). 
\end{equation*}

Then  the composite map 
\begin{equation*}
\varphi _{0}^{-1}\circ \prod_{\mathfrak{p}}\varphi _{\mathfrak{p}}\in 
\mathrm{Aut}(( A, Tr_A')^{\perp m}\otimes_R\mathbb{A}_K)=U(M_m(\mathbb A_K[G]) 
\end{equation*}
 determines a well-defined double coset $$U(M_m(K[G])(\varphi _{0}^{-1}\circ \prod_{\mathfrak{p}}\varphi _{\mathfrak{p}}) \prod_{\mathfrak{p}}U(M_m(A^D_{\frak{p}})) \ \in  c(R, U_{m, A^D})$$ which is independant of  choices. The map 
 $$(M, q)\rightarrow U(M_m(K[G])(\varphi _{0}^{-1}\circ \prod_{\mathfrak{p}}\varphi _{\mathfrak{p}}) \prod_{\mathfrak{p}}U(M_m(A^D_{\frak{p}})$$ induces a map $\alpha_m: \mathrm{LI}((A, Tr_A')^{\perp m})\rightarrow c(R, U_{m, A^D})$ which is a bijection of  pointed sets, pointed respectively by the isomorphism class of $(A, Tr_A')^{\perp m}$ and the trivial  double coset  of  $c(R, U_{m, A^D})$. By taking determinants and using Lemma \ref{detred} we obtain a map 
 $$c(R, U_{m, A^D})\rightarrow CU(A^D).$$
 We recall that the unitary class of $(M,q)$  is represented in $CU(A^D)$ by $\mathrm{Det} (\varphi _{0}^{-1}\circ \prod_{\mathfrak{p}}\varphi _{\mathfrak{p}})$
 (see Section 2.7).

We   attach to any $B\in PH(A)$ the $G$-form $(B, Tr'_B)$. Since $G$ is of odd order we know by Corollary 3.5 that the $G$-forms  
  $(B, Tr'_B)$ and $(A, Tr'_A)$  are locally isomorphic.  Therefore  we can consider   for each prime ideal of $R$ isomorphisms
 
\begin{equation}\label{r1}
\phi _{\mathfrak{p}}:( A_{\mathfrak{p}}, Tr'_{A_{\mathfrak{p}}})
 \simeq ( B_{\mathfrak{p}},Tr_{B_{\mathfrak{p}}}^{\prime
}) \ \mathrm{and}\ \phi _{0}:( A_{K}, Tr'_{A_K}) \simeq (
B_K,Tr_{B_K}^{\prime }).
\end{equation}
We set $\theta_{\frak{p}}=\phi _{0}^{-1}\circ \phi _{\mathfrak{p}}$  and $\theta=\prod_{\frak{p}}\theta_{\frak{p}}$. We note that $\theta_{\frak{p}} \in U(K_{\frak{p}}[G])$ and so that $\varepsilon (\theta_{\frak{p}})=\pm 1$.  By possibly modifying the generator of $B_{\frak{p}}$ as an $A^D_{\frak{p}}$-module we can assume that $\varepsilon (\theta_{\frak{p}})=1$ then  we say  in this case that $\theta_{\frak{p}}$ is normalized. We know that $\mathrm{Det}(\theta)$ represents $\phi(B)$.  
More generally, for any integer $m$ then $(B, Tr'_B)^{\perp m}$ defines a class in $\mathrm{LI}((A, Tr_A')^{\perp m})$. The double coset class 
\begin{equation}\label{rep}U(M_m(K[G])(\theta^{\perp m}) \prod_{\mathfrak{p}}U(M_m(A^D_{\frak{p}})) \ \in  c(R, U_{m, A^D})\end{equation}  represents  $\alpha_m(B):=\alpha_m((B, Tr'_B)^{\perp m})$ and $\mathrm{Det}(\theta^{\perp m})$ represents the unitary class of $(B, Tr'_B)^{\perp m}$ in $CU(A^D)$.

\subsection {Proofs of the main theorems} 
 For the reader's ease we recall  the statements of these theorems. 
 
 \begin{thm}Suppose that $G$ is an abelian group of odd order and $K$ is a number field. Then $(B, Tr'_B)$ and $(A, Tr'_A)$ are isomorphic $G$-quadratic spaces if and only if $B$ is a free $A^D$-module. 
\end{thm}

\begin{thm} \label{choc} Let   $G$ be a finite group of odd order and  $K$ is  a non-totally real number field. Suppose that $\phi(B)^n=1$. Then 
$(B, Tr'_B)^{\perp{m_n}}$ and  $(A, Tr'_B)^{\perp{m_n}}$ are isomorphic $G$-quadratic spaces,   where $m_n=1$ (resp. $2n$) if $n=1$ (resp. $n>1$).
\end{thm}

\begin{thm}\label{LA}
 Suppose that  $G$ is an abelian  group of  odd order and  $K$  a non-totally real number field. If $e=e\left( G\right) $ is the exponent of $G,$
then $\left( B,Tr_{B}^{\prime }\right) ^{\perp 2e}$ and $\left(
A,Tr_{A}^{\prime }\right) ^{\perp 2e}\ $are\ isomorphic $G$-quadratic spaces.
\end{thm}

\begin{thm}\label{gen}
Suppose that $G$ is a group of  odd order and $K$  is a  non-totally real number field.   If  $\psi(B)^m=1$,   then 
$(B,Tr_{B}^{\prime }) ^{\perp 4m}$ and $( A,Tr_{A}^{\prime })
^{\perp 4m} $ are isomorphic  $G$-quadratic spaces. 
\end{thm}

\begin{thm}\label{const} Suppose that$\ G$ has odd order and that  $ \mathrm{Spec}( A) $ is a constant group
scheme. If $K/\mathbb{Q\,}\ $ is a non-totally real number field,  unramified at the primes dividing the order
of $G$
 then $( B,Tr_B^{\prime }) ^{\perp 2e^{ab}}$ and $
( A,Tr_{A}^{\prime }) ^{\perp 2e^{ab}} $ are isomorphic $G$-quadratic spaces, where $e^{ab}=e( G^{ab}) $
denotes the exponent of $G^{ab}$. 
\end{thm}
 
 In the above  theorems,  since $G$ is odd, it follows from  Corollary \ref{locod}   that  $( B_{\mathfrak{p}},Tr_{B_{\mathfrak{p}
}}^{\prime }) $ and $( A_{\mathfrak{p}},Tr_{A_{\mathfrak{p}
}}^{\prime })$ are locally isomorphic $G$-quadratic spaces for all
prime ideals of $R$. We recall some notation and results from Section 2.4. We fix
local isomorphisms of quadratic $G$-spaces $\phi _{\mathfrak{p}}:( A_{
\mathfrak{p}},Tr_{A_{\mathfrak{p}}}^{\prime }) \cong ( B_{
\mathfrak{p}},Tr_{B_{\mathfrak{p}}}^{\prime }) \ $ and we set $\phi_{\mathfrak{p}}( \lambda l) =b_{\mathfrak{p}}$  which generates $
B_{\mathfrak{p}} $ over $A_{\mathfrak{p}}^{D}$;  then there exists $\theta_{\mathfrak{p}} \in U(A^D_{K_{\mathfrak{p}}})$ such that 
\begin{equation*}
\phi _{\mathfrak{p}}( \lambda l)=b_{\mathfrak{p}}=\theta_{\mathfrak{p}}b_0=\theta_{\mathfrak{p}}\phi_0(\lambda l).
\end{equation*}  This then immediately gives $\phi_{\mathfrak{p}}=\theta_{\mathfrak{p}}\phi_0$  and whence  $\phi_0^{-1}\circ \phi _{\mathfrak{p}}=\theta_{\mathfrak{p}}$.
 Since $b_{\mathfrak{p}}=\theta_{\mathfrak{p}}b_0$ 
it follows from the resolvend formula that 
$$r(b_{\mathfrak{p}})=\theta_{\mathfrak{p}}r(b_0)$$
and so that  $\lambda^{-1}r(b_{\mathfrak{p}})=\theta_{\mathfrak{p}}\lambda^{-1}r(b_0)$. 
By Proposition \ref{expli} we know that $\lambda^{-1}r(b_{\mathfrak{p}})\in U(A^D_{\mathfrak{p}}\otimes B_{{\mathfrak{p}}})$ and $\lambda^{-1}r(b_0)\in U(A^D_K\otimes B_K)$. We conclude  
\begin{equation*}
\phi _{0}^{-1}\circ \phi _{\mathfrak{p}}=\theta _{\mathfrak{p}}=(\lambda^{-1}r( b_{\mathfrak{p}}))(\lambda r(
b_{0}) ^{-1})
\end{equation*}
and therefore that the unitary class $\phi(B)$ of $B$    is represented in $\mathrm{CU}( A^{D})  $ by
\begin{equation*}
\mathrm{Det}( \lambda ^{-1}r( b_{0}))
^{-1}\prod_{\mathfrak{p}}\mathrm{Det}( \lambda ^{-1}r( b_{
\mathfrak{p}})).
\end{equation*}

Writing $e$ for the exponent of $G^{ab}$, then $\mathrm{Det}(
g) ^{e}=1,$ and so by the Galois action formula (\ref{act}) we know that $
\mathrm{Det}( \lambda ^{-1}r( b_{p})) ^{e}$ and $
\mathrm{Det}( \lambda ^{-1}r( b_{0})) ^{e}$ are both $
G$-fixed and therefore,  since

$$\prod_{\mathfrak{p}}\mathrm{Det}(\phi _{0}^{-1}\circ \phi _{
\mathfrak{p}}) ^{e} =\mathrm{Det}(( \lambda ^{-1}r(
b_{0})) ^{e}) ^{-1}\prod_{\mathfrak{p}}\mathrm{
Det}(( \lambda ^{-1}r( b_{p}) ^{e})),  
$$ we conclude that 
\begin{equation}\label{clecle}
\prod_{\mathfrak{p}}\mathrm{Det}(\phi _{0}^{-1}\circ \phi _{
\mathfrak{p}}) ^{e}\in  \mathrm{Det}( U( A_{K}^{D}\otimes B_{K})
) ^{G} \prod_{\mathfrak{p}}\mathrm{Det}(
 U( A_{\mathfrak{p}}^{D}\otimes B_{\mathfrak{p}})
 )^{G}.
\end{equation}

\noindent{\bf  Proof of Theorem 6.5.}
 Indeed,  if the $G$-forms $(B, Tr'_B)$ and $(A, Tr'_A)$ are isomorphic, then $B$ is isomorphic to $A$ as an $A^D$-module and thus  is a free $A^D$-module. We now assume that $B$ is a free $A^D$-module. Then $B$ and $A$ are isomorphic $A^D$-modules and thus $\psi(B)=1$. Using the commutativity of the diagram (\ref{didi}) we deduce that $\phi(B)\in \mathrm{Ker}(\xi_A)$. Since $G$ is  of odd order,   the exponent of $\mathrm{Ker}(\xi_A)$ divides $2$ by Proposition 5.11 and so $\phi(B)^2=1$.

Since $G$ is abelian $\mathrm{Det}$ is
an isomorphism (see (\ref{isab}))  and so 
\begin{eqnarray*}
\mathrm{Det}(U( A_{K}^{D}\otimes B_{K})) ^{ G}
&=&U( A_{K}^{D}\otimes B_{K}) ^{G}=U(A_{K}^{D})
=\mathrm{Det}(U( A_{K}^{D}) ) \\
\mathrm{Det}(U( A_{\mathfrak{p}}^{D}\otimes B_{\mathfrak{p}
}) )^{ G} &=&U( A_{\mathfrak{p}}^{D}\otimes B_{
\mathfrak{p}})^{ G} =U( A_{\mathfrak{p}}^{D})
=\mathrm{Det}( U( A_{\mathfrak{p}}^{D})) .
\end{eqnarray*}
Hence  it follows from (\ref {clecle}) that $\phi(B)^e=1$. Since $e$ is odd we deduce that $\phi(B)=1$. It follows from Section 6.2 that $\phi(B)=\mathrm{Det} (\alpha_1(B))=1$. Since $\mathrm{Det}$ is an isomorphism we conclude that $\alpha_1(B)=1$ and therefore that the $G$-forms $(B, Tr'_B)$ and $(A, Tr'_A)$ are isomorphic. \qed

\vskip 0.2 truecm

\noindent {\bf  Proof of Theorem 6.6.}  
Since $\alpha_m$ is a bijection for every $m$ (see Section 6.2.),  it suffices  to show  that if $\phi(B)^n=1$ then $\alpha_{m_n} (B)$  is the trivial class of $c(R, U_{m_n})$. 

We start by treating the case where $n=1$.  We write $U_{A^D}$ for $U_{1, A^D}$ and  $\alpha$ for $\alpha_1$.  
Since 
$$\phi(B)=\mathrm{Det} (\alpha(B))=1$$
 it suffices to check that 
$$Det: c(R, U_{ A^D})\rightarrow CU(A^D)$$ is a bijection of pointed sets, respectively pointed by the trivial class and the unit element.
The surjectivity is evident; it remains to prove the injectivity. Given $$
\gamma =\prod_{\mathfrak{p}}\gamma _{\mathfrak{p}}\in U( 
\mathbb{A}_{K}[ G])$$ such that  
\begin{equation*}
\mathrm{Det}( \gamma) =\prod_{\mathfrak{p}}\mathrm{Det}
( \gamma _{\mathfrak{p}}) =\mathrm{Det}( \beta
_{0}^{-1}\prod_{\mathfrak{p}} \delta _{
\mathfrak{p}})
\end{equation*}
with $\beta _{0}\in U( K[ G]) $ and $\delta _{
\mathfrak{p}}\in U( A_{\mathfrak{p}}^{D}),$ then 
\begin{equation*}
\beta _{0}\gamma \prod_{\mathfrak{p}}\delta _{
\mathfrak{p}}^{-1}\in SU( \mathbb{A}_{K}[ G]).
\end{equation*}

Since by Proposition \ref{Kimp} we know that $SU_G$ has Kneser Strong Approximation, so, for any non empty open subset  $V$ of $SU_G(\mathbb A_K)$,  one has the equalities:
\begin{equation}\label{kn1} SU(\mathbb A_K[G])=SU_G(\mathbb A_K)=SU_G(K)V=SU(K[G])V.\end{equation} 
The group  $U_G$ is the generic fiber of $U_{A^D}$ and thus we know that $\prod_{\frak{p}}U_{A^D}(R_{\frak{p}})$ is an open subgroup of $U_G(\mathbb A_K)$. Since $SU_G$ is a closed subgroup of $U_G$, then the topology of $SU_G(\mathbb A_K)$ is induced from the topology of $U_G(\mathbb A_K)$. Therefore 
$$V:=SU_G(\mathbb A_K)\cap \prod_{\frak{p}}U_{A^D}(R_{\frak{p}})$$ is an open subgroup of $SU_G(\mathbb A_K)$.

It follows from (\ref{kn1})  that   we may write 
\begin{equation}
\beta _{0}\gamma \prod_{\mathfrak{p}}\delta _{\mathfrak{p}%
}^{-1}=\sigma _{0}\prod_{\mathfrak{p}}\sigma _{\mathfrak{p}}, 
\end{equation} with $\sigma_0\in SU(K[G])$ and $\sigma_{\mathfrak{p}}\in SU(A_{\mathfrak{p}}^{D}):=SU(K_{\frak{p}})\cap U(A^D_{\frak{p}})$. 
We conclude that  
\begin{equation*}
\gamma =\beta _{0}^{-1}\sigma _{0}\prod_{\mathfrak{p}}\sigma _{
\mathfrak{p}}\delta _{\mathfrak{p}}
\end{equation*}
and hence $\gamma $ lies in the trivial double coset. Therefore  the injectivity is established and the proof of Theorem 6.6 for $n$=1 is now complete. 
\vskip 0.1 truecm

We now consider the case $n>1$ and $m=2n$.  First we note that $SU_{m, G}$ does not have Kneser's Strong Approximation since $SO_{m, K}$ is not a simply connected algebraic group for $m>1$. 

   It follows from (\ref{rep})  that $\alpha_m(B)$ is represented by 
  $\theta^{\perp m}\in U(M_{m}(\mathbb A_K[G]) $. 
  More precisely $\theta^{\perp m}=\prod_{\frak{p}}\theta_{\frak{p}}^{\perp m}\in \prod_{\frak{p}}U(M_{m}(K_{\frak{p}}[G])$, where $\theta_{\frak{p}}^{\perp m}$ denotes the diagonal matrix 
  $m\times m$ with diagonal entries all equal to   $\theta_{\frak{p}}$.  By choosing each $\theta_{\frak{p}} $ normalized (see (115) Section 6.2) we note that $\theta^{\perp m} \in SU^\varepsilon_{m, G}(\mathbb A_K)$.    
   We consider the map  
$$\mathrm{Det} : c(R, U_{m, A^D})\rightarrow CU(A^D).$$ It follows from the above that $\mathrm{Det} (\alpha_m(B))$ is represented  in  $ Det(U(M_{2n}(\mathbb A_K[G]))=Det(U(\mathbb A_K[G])$    by $Det(\theta^{\perp  m})$. Since $\phi(B)^n=1$, there exists $a\in U(K[G])$ and $u\in \prod_{\mathfrak{p}}U(A^D_{\mathfrak{p}})$ such that  
$$\mathrm{Det}(\theta)^n =\mathrm{Det}(a) \mathrm{Det}(u)$$ and thus 
  $$Det(\theta^{\perp m})=Det(\theta)^{2n}=Det(a)^2Det(u)^2 .$$ 
  
  Since  by Lemma \ref{detred} we know that 
  $\mathrm{Det}(U(M_{m}(A^D_{
\mathfrak{p}}))=\mathrm{Det}(U(A^D_{
\mathfrak{p}}))$ and  $\mathrm{Det}(U(M_{m}(K[G]))=\mathrm{Det}(U(K[G]))$, 
 we conclude that there exist $x\in U(M_{m}(K[G])$ and  $v_{\frak{p}} \in  U_{m}(A^D_{\frak {p}}))$, for every $\frak{p}$, such that
  $$Det(a^2)=Det(x)\ \mathrm{and}\ \ Det(u_{\frak{p}}^2)=Det(v_{\frak{p}}).$$
 
   Using  (\ref{decu})  we know that

  \begin{equation} U_{m,G}(K)=O_{m}(K)\times \prod_{I\in I}U_{A_{m,i}}(K)\times  \prod_{j\in J}U_{B_{m,j}}(K).\end{equation} 
  Hence the element $x$ can be written
  $$x=(x_0  \times \prod_{i \in I}y_i \times \prod_{j\in J} z_j)$$ where $\mathrm{det}(\varepsilon (x_0))=\varepsilon (a)^2$.
Since $a\in U(K[G])$ then $\varepsilon (a)=\pm 1$ and so $\varepsilon (a^2) =1$.
Therefore we deduce that 
$$Det(x)=Det(x_0\times \prod_{i\in I } y_i\times \prod_{j\in J} z_j)=Det(1\times \prod_{i\in I } y_i\times \prod_{j\in J} z_j)=Det(x')$$ with $x'=(1\times \prod_{i\in I } y_i\times \prod_{j\in J} z_j) \in U^{\varepsilon}_{m ,G}(K)$.

We write $v=(v_{\frak{p}}) \in \prod_{\frak{p}}U(M_{m}(A^D_{\frak{p}})$. Since $\theta$ is normalized we obtain that  
$$Det_{\varepsilon}(\theta^{\perp m})=Det_{\varepsilon}(x')Det_{\varepsilon}(v)=1.$$ Then we deduce that $Det_{\varepsilon}(v)=1$ and so $Det_{\varepsilon}(v_\frak{p})=1$  for all $\frak{p}$. More precisely $\varepsilon (v_{\frak{p}})\in U(M_{m}(R_{\frak{p}}))$ and $det (\varepsilon (v_{\frak{p}}))=1$. Since $R_{\frak{p}}\subset A^D_{\frak{p}}$ then $\varepsilon (v_{\frak{p}})\in U(M_{m}(A^D_{\frak{p}}))$. We set $v'_{\frak{p}}=v_{\frak{p}}(\varepsilon (v_{\frak{p}})^{-1})$. Then $v'_{\frak{p}}\in U(M_{m}(A^D_{\frak{p}}), 
\varepsilon (v'_{\frak{p}})=1$ and $Det(v'_{\frak{p}})=Det(v_{\frak{p}})$. We write $v'=(v'_{\frak{p}})$ and we  conclude that 
$$Det(\theta^{\perp m})=Det(x')Det(v').$$
Therefore $z:=x'^{-1}\theta^{\perp m}v'^{-1}\in U_{m, G}^{\varepsilon}(\mathbb A_K)$ and $Det(z)=1$ and so   $z\in SU_{m, G}^{\varepsilon}(\mathbb A_K)$. We may now  
conclude as in the case $m=1$.  Since $SU_{m, G}^{\varepsilon}$ is a closed subgroup of $U_{m, G}$, then $V:=SU_{m, G}^{\varepsilon}\cap \prod_{\frak{p}}U_{m, A^D}(R_{\frak{p}})$  is an open subgroup of $SU_{m, G}^{\varepsilon}(\mathbb A_K)$ and therefore Kneser's Strong Approximation Theorem provides us with the equality
$$SU_{m, G}^{\varepsilon}(\mathbb A_K)=SU_{m, G}^{\varepsilon}(K)V.$$ Therefore one can write $z=t.r$ with 
 $t\in SU^{\varepsilon}_{m, A^D}(K)$ and $r \in V$. We conclude that 
$$\theta^{\perp m}=x't.rv'\in U_{m, A^D}(K) U_{m, A^D}(\tilde R).$$
Therefore the double coset defined  by $\theta^{\perp m}$ is the trivial coset   and thus 
$$(B, Tr')^{\perp m}\simeq (A, Tr')^{\perp m}.$$ This completes the proof of  Theorem 6.6.
\qed 
\vskip 0.2 truecm

\noindent{ \bf  Proof of Theorem 6.8.}
Since   $\psi ( B) ^{m}=1$, then using that  $\psi =\xi \circ \phi$
and the fact that, as shown in Section 5, $\ker \xi_A $ is killed by 2 since $
G$ has odd order,  we  deduce that $\phi(B)^{2m}=1$. Therefore it follows  from Theorem 6.6  that  $(B, Tr_B')^{\perp{4m}}$ and $(A, Tr'_A)^{\perp{4m}}$ are isomorphic $G$-forms. 
\qed
\vskip 0.2 truecm

\noindent{\bf Proof of Theorem 6.7 and Theorem 6.9.}
 The proof of these theorems are similar. They both use the fixed points Theorems of Section 2.5 . 
 
 Let $m$ be an integer, $m\geq 1$. We know that  the unitary class $\phi(B)$ of $B$ is  represented in the  
 group $\mathrm{CU}(A^D)$ by $\prod_{\frak p}\mathrm{Det}(\phi_0^{-1}\circ \phi_{\frak p})$. Suppose now that the exponent $e$ of $G^{ab}$ divides $m$. Then it follows from 
 (\ref{clecle}) that 
\begin{equation} \prod_{\mathfrak{p}}\mathrm{Det}(\phi _{0}^{-1}\circ \phi _{
\mathfrak{p}}) ^{m}\in  \mathrm{Det}( U( A_{K}^{D}\otimes B_{K})
) ^{G} \prod_{\mathfrak{p}}\mathrm{Det}(
 U( A_{\mathfrak{p}}^{D}\otimes B_{\mathfrak{p}})
 )^{G}.\end{equation}
 Therefore,   if  the following equalities hold, 
$$ \mathrm{Det}( U( A_{K}^{D}\otimes B_{K}))^{G}=\mathrm{Det}(U(A_K^D))\ \mathrm{and}\ \mathrm{Det}(U( A_{\mathfrak{p}}^{D}\otimes B_{\mathfrak{p}
}) )^{ G} =\mathrm{Det}( U( A_{\mathfrak{p}}^{D})) $$
then we conclude that $\phi(B)^m=1$.

From Theorem 2.16 (i) we obtain that 
\begin{equation}\label{k_1}\mathrm{Det}((A^D_K\otimes_KB_K)^\times)^G=\mathrm{Det}((Z_\chi\otimes_KB_K)^\times)^G=\prod_{\chi}Z_{\chi}^\times=\mathrm{Det}((A^D_K)^\times).\end{equation} Therefore, using Theorem 5.2, we deduce from (\ref{k_1}) that
 $$\mathrm{Det}(U(A^D_K\otimes_KB_K))^G=\mathrm{Det}((A^D_K\otimes_KB_K)^\times)_-^G=\mathrm{Det}((A^D_K)^\times)_-=\mathrm{Det}(U(A^D_K)). $$
Suppose now that for any maximal ideal $\frak{p}$ of $R$ the fixed point theorem  holds and so that 
$$\mathrm{Det}((A^D_\frak{p}\otimes_RB_\frak{p})^\times)^G=\mathrm{Det}((A^D_\frak{p})^\times)). $$
Since by  Theorem 5.3 we have the equalities 
$$\mathrm{Det}(U(A^D_\frak{p}\otimes_{R_\frak{p}}B_\frak{p}))=\mathrm{Det}((A^D_\frak{p}\otimes_{R_\frak{p}}B_\frak{p})^\times)_- \ \ \mathrm{and}\ \   
 \mathrm{Det}(U(A^D_\frak{p}))=\mathrm{Det}({A^D_\frak{p}}^\times)_-$$ we conclude that
$$ \mathrm{Det}(U( A_{\mathfrak{p}}^{D}\otimes B_{\mathfrak{p}
})) ^{G}=\mathrm{Det}( (A_{\mathfrak{p}}^{D}\otimes B_{
\mathfrak{p}})^\times) _{-}^{G}=\mathrm{Det}( A_{\mathfrak{p}}^{D \times})
_{-}=\mathrm{Det}(U( A_{\mathfrak{p}}^{D})) .$$
It follows from the equalities above that   whenever the fixed point theorems hold then  $\Phi(B)^m=1$ and therefore,  using Theorem 6.6,  that we have an isomorphism of $G$-quadratic spaces 
\begin{equation}(B, Tr')^{\perp 2m}\simeq (A, Tr')^{\perp 2m}.\end{equation}
Since we know by Theorem 2.17 and Theorem 2.18 that under the hypotheses of Theorem 6.7 and Theorem 6.9 we have  fixed point theorems we deduce from (123) that in both cases we have the equality 
$$(B, Tr')^{\perp 2e}\simeq (A, Tr')^{\perp 2e}$$
as required.

 \qed
\vskip0.2 truecm

 \section{Appendix.  Lang-Steinberg Theorem for algebraic connected group schemes (by Dajano Tossici)}

In this appendix we prove the following result, which generalizes the classical Lang-Steinberg Theorem to the case of non-smooth group schemes.

\begin{thm}\label{SLa} 
Let $k$ be a perfect field of characteristic $p$ and let $G$ be an algebraic group scheme. We assume that one of the following conditions holds:
\begin{enumerate}
\item $G$ is  finite and connected.

\item  $G$ is affine and connected and the cohomological dimension of $k$ is $\le 1$, 
\item 
$k$ is finite and $G$ connected.
\end{enumerate}

Then $H^1(\mathrm{Spec}(k), G)=0$.
\end{thm}

\begin{proof}
For a proof of  Theorem 7.1. in the linear smooth case see \cite{Serre} III, Theorem 1, which reproduces the original proof due to Steinberg \cite{St65} Theorem 1.9. The entire article of Steinberg is also reproduced as an appendix of Serre's book.
The non-linear smooth case over a finite field was proved previously by Lang \cite{La56} Theorem 2.  
We stress that for finite and connected group schemes the result is stronger since it is  true for any perfect field. 

We firstly suppose that $G$ is finite and connected.

Since $G$ is finite and connected then there is some power  $\mathrm{F}^n_G:G\rightarrow G^{(p^n)}$ of Frobenius which is trivial, where for any scheme over $k$, we set $X^{(p)}$ as the base change through the absolute Frobenius of $k$ and this procedure can be iterated. 
Let $T\rightarrow \mathrm{Spec} (k)$ be any $G$-torsor. We remark that $T^{(p^n)}\rightarrow   \mathrm{Spec }(k)$, obtained by base change, is again a $G$-torsor. Since  Frobenius is functorial we have that the diagram
   \[
   \xymatrix{
   G\times_k T\ar[d]^{\mathrm{Fr}_G^n\times \mathrm{Fr}_T^n}\ar[r]^{\mu} &
  T\ar^{\mathrm{Fr}^n_T}[d]\\
  G^{(p^n)}\times_k T^{(p^n)}\ar[r]^{\mu}& T^{(p^n)}}
   \]
   commutes.
   
   Since  $\mathrm{Fr}^n_G:G\rightarrow  G^{(p^{n})}$ is trivial then the map $\mathrm{Fr}_T^n$ factorizes through $ T/G=  \mathrm{Spec }(k)$. This means that the $G$-torsor  $T^{(p^n)}\rightarrow    \mathrm{Spec} (k)$ is trivial.  Since $k$ is perfect then $\mathrm{Fr}_k$ is an isomorphism and then we can obtain a section 
  of $T \rightarrow \mathrm{Spec}(k)$. This proves that $H^1(\mathrm{Spec}( k), G)=0$ if $G$ is finite and connected.

The following lemma allows us to combine Steinberg's Theorem with the finite and connected case to obtain the proof of  Theorem \ref{SLa} in the remaining cases. 

\begin{lem}\label{lemma: smooth finite}
Let $k$ be a perfect field. Then any connected algebraic group scheme $G$ over $k$ sits in a short exact sequence of  $k$-algebraic group schemes
$$
0\rightarrow H \rightarrow G \rightarrow A\rightarrow 0
$$
where $H$ is connected and finite  over $k$ and $A$ is connected and  smooth over $k$.
\end{lem}
%

\begin{proof}
First of all we consider the induced reduced sub-scheme $G_{red}$ which is a closed subgroup scheme of $G$ since $k$ is perfect. Since it is reduced it is therefore smooth. Obviously it is also connected. Now, since $G/G_{red}$ is a geometric quotient then topologically it is clearly a point so it is a finite local scheme over $k$. Then there exists some positive integer $n$ such that
the Frobenius $\mathrm{Fr}^n: G/G_{red}\to (G/G_{red})^{(p^n)}$ factorizes through $\mathrm{Spec}( k)$. 

Now by the functoriality of Frobenius   we have the following commutative diagram
$$
   \xymatrix{G_{red} \ar[d]\ar[r]^{\mathrm{Fr}_{G_{red}}^n}& (G_{red})^{(p^n)}\ar[d]\\G \ar[d]\ar[r]^{\mathrm{Fr}_{G}^n}& G^{(p^n)}\ar[d]\\
   G/G_{red}\ar[r]^(0.4){\mathrm{Fr}_{G/G_{red}}^n} &
  (G/G_{red})^{(p^n)}}
  $$
so the induced map $G\rightarrow (G/G_{red})^{(p^n)}$ factorizes through $\mathrm{Spec} (k)$. Therefore, since  $(G/G_{red})^{(p^n)}$ is canonically isomorphic to $G^{(p^n)}/(G_{red})^{(p^n)} $,  $\mathrm{Fr}^n_G$ factorizes through $(G_{red})^{(p^n)}$. We set $A=(G_{red})^{(p^n)}$ and we observe that the  morphism
$
\phi:G\rightarrow A
$
  induced by $\mathrm{Fr}^n_G$ is a quotient map since $\mathrm{Fr}^n_{G_{red}}$ is a faithfully flat map, because $G_{red}$ is smooth. So now we take $H=\mathrm{Ker} \phi$, which is clearly finite and connected since it is the kernel of a power of the Frobenius.

 \end{proof}

Now we complete the proof of the Theorem. Let $G$ be a connected algebraic group scheme over a perfect field $k$ of cohomological dimension $\le 1$. By Lemma \ref{lemma: smooth finite} we have an exact sequence
\begin{equation}\label{eq:smooth finite}
0\rightarrow H\rightarrow G\rightarrow A\rightarrow 0
\end{equation}
with $H$ connected and finite over $k$ and $A$ smooth over $k$. Moreover,  $A$ is connected since 
$G$ is connected.
Now, by the finite and connected case, we have that $H^1(\mathrm{Spec} k,H)=0$. We observe that if $G$ is also affine then $A$ is affine. So, by Steinberg's Theorem, we have that $H^1(\mathrm{Spec} (k), A)=0$ if $A$ is affine. So by the long exact sequence of pointed sets associated to \eqref{eq:smooth finite} we conclude that $H^1(\mathrm{spec}( k),G)=0$, if $G$ is finite. If $k$ is finite and $G$ is not necessarily affine, then we use Lang's result instead of  Steinberg's
result.

\end{proof}

\end{document}